\pdfoutput=1
\documentclass[10pt]{aart}

\PassOptionsToPackage{natbib=true}{biblatex}
\usepackage{arxiv-setting}
\usepackage{booktabs}
\usepackage{placeins}
\usepackage{tabularx}

\usepackage{float}

\theoremstyle{definition}
\newtheorem{question}{Question}
\theoremstyle{remark}

\renewcommand{\epsilon}{\varepsilon}

 \makeatletter
 \define@key{todonotes}{ID}[]{     \setkeys{todonotes}{author=ID, inline, color=blue}}%
 \makeatother
 \makeatletter
 \define@key{todonotes}{MJ}[]{%
     \setkeys{todonotes}{author=MJ, inline, color=orange}}%
 \makeatother
 \makeatletter
 \define@key{todonotes}{HW}[]{%
     \setkeys{todonotes}{author=HW, inline, color=red!20!white}}%
 \makeatother

\newcommand{\toromanlower}[1]{\ifcase#1 \or i\or ii\or iii\or iv\or v\or vi\or vii\or viii\or ix\or x\fi}

\title{Majority Dynamics on Resampled Sparse Erd\H{o}s--R\'enyi Graphs:
Gaussian Winner Selection in Collaborative Agent Networks\blfootnote{Author names are listed in alphabetical order.}}
\author{
Ioana Dumitriu%
\thanks{\mbox{\scriptsize Department of Mathematics, University of California San Diego, La Jolla, CA 92093, USA; \texttt{idumitriu@ucsd.edu}.}}
\and
Muchen Ju\orcidlink{0009-0004-6131-6911}%
\thanks{\mbox{\scriptsize Department of Mathematics, University of Pennsylvania, Philadelphia, PA 19104, USA; \texttt{muchenju@sas.upenn.edu}.}}
\and
Hai-Xiao Wang\orcidlink{0000-0003-2730-1439}%
\thanks{\mbox{\scriptsize Department of Applied Mathematics, University of Washington, Seattle, WA 98195, USA; \texttt{haixwang@uw.edu}.}}
}

\date{\today}

\begin{document}

\maketitle

\begin{abstract}
    Collective decisions in agent networks emerge from repeated local interactions, as agents update their opinions (colors) based on their neighbors, potentially leading to consensus. In sparse communication graphs, repeated interactions can amplify both the initial majority preference and the random fluctuations caused by individual agents. Understanding when the initial majority determines the final preference and how quickly an agreement emerges is, therefore, a basic question in collective decision-making. We study binary majority dynamics as an idealized model of this process. At each round, the interaction graph is independently resampled from the Erd\H{o}s--R\'enyi model $\mathds{G}(N,p)$, with $p=b\log N/N$ and fixed $b>1$. Each agent adopts the majority opinion (color) among its neighbors, retaining its current opinion if there is a tie. For a fixed initial configuration, let $\Delta_0$ denote the initial difference between the number of blue agents and the number of red agents. We identify three regimes for winner selection and consensus time, determined by the size of $\Delta_0$. First, in the critical window $\Delta_0\sqrt p=O(1)$, the probability of blue unanimity is
    $$
    \Phi\left(\Delta_0\sqrt{2p/\pi}\right) + o(1),
    $$
    where $\Phi$ is the cumulative distribution function of a standard Gaussian random variable. Moreover, consensus is reached after $(1+o(1))\log N/\log\log N$ rounds with high probability. Second, in the intermediate regime $1\ll\Delta_0 \sqrt{p}\lesssim \sqrt{N}$, we establish explicit high probability upper and lower bounds on the time to blue-unanimity. Lastly, when the initial blue advantage is above an explicit constant multiple of $N/\sqrt{\log N}$, blue unanimity is guaranteed within two rounds with high probability. Our findings provide a complete description of how the initial advantage influences winner selection and the speed of consensus in this type of majority dynamics with sparse interactions.
\end{abstract}


\newpage
\setcounter{tocdepth}{1}
\tableofcontents

\newpage
\section{Introduction}
\label{sec:introduction}

In a large collaboration network, agents exchange opinions repeatedly through local interactions until they reach a collective decision. Due to communication constraints, the number of contacts with which an agent can interact each round may be limited when compared to the entire population, leading to a \emph{sparse} interaction graph. Hence, the opinion favored in its neighborhood may differ from the opinion favored globally. While repeated interactions can reinforce the initial population preference, they may also amplify random fluctuations or propagate contrarian opinions. Recent work on the collective behavior of AI agents motivates studying these effects through models of interacting opinions \citep{el2026physics}. These considerations raise the following two questions regarding communication dynamics and the correctness of information propagation.
\begin{question}[Consensus]
Starting from any fixed configuration, under what conditions is the initial preference preserved, and how quickly is consensus reached?
\end{question}
\begin{question}[Correctness]
When initial opinions are random choices between two options, slightly in favor of the correct one, how likely is it that the final consensus is aligned with the correct choice?
\end{question}

We explore answers to both questions by studying \emph{majority dynamics}, a fundamental model for opinion evolution introduced by \cite{mossel2017opinion}. Each agent holds one of two possible opinions, \emph{blue} or \emph{red}. Let $\gB_t$ and $\gR_t$ denote the sets of blue and red agents after round $t\geq 0$, respectively, and $\Delta_0 = |\gB_0| - |\gR_0|$ denote the initial blue advantage. In each round, all agents update their opinions synchronously by following the majority opinion among their neighbors, retaining their current opinion if there is a tie. We say that \emph{blue unanimity} is achieved after round $T$ if $\gB_{T} = \gV$. For the sake of theoretical analysis, we assume that the interaction graph $\gG$ is sampled from the Erd\H{o}s--R\'enyi model $\mathds{G}(N, p)$, where each pair of distinct agents is independently connected with probability $p$. We say that $\gG$ is a \emph{dense} graph when $p \in (0, 1)$ is a fixed constant; in contrast, it \emph{is sparse} when $p=o(1)$.

Previous literature largely focuses on the scenario that $\gG$ is \emph{dense} and \emph{static}. \citet{tran2023reaching} proved that, with high probability, an initial blue advantage of at least $6$ agents ($\Delta_0 \geq 6$) leads to blue unanimity after $4$ rounds.
\citet{berkowitz2020central} established a central limit theorem for the number of agents of each opinion after one round and improved the result of \citet{tran2023reaching} by decreasing the necessary initial blue advantage to $\Delta_0 \geq 3$. \citet{sah2024majority} established an exact winner-selection law and proved that any nonzero initial blue advantage ($\Delta_0 \geq 1$) will lead to blue unanimity after $3$ rounds. For the sparse case where $pN \asymp \log(N)$ but $\gG$ remains fixed, \cite{tran2025power} proved that an initial blue advantage satisfying $p\Delta_0 \geq 10$ will lead to blue unanimity after $\log(N)/\log\log(N)$ rounds. They conjectured that, similar to the dense case, a winner-selection law can be established when $\Delta_0 \sqrt{p} \asymp 1$. There has been progress toward resolving this conjecture for sparse, static interaction graphs \citep{kim2025new,jaffe2025new,goel2026majority}, but it has not yet been fully solved. Readers may refer to Section~\ref{app:related-literature} for a more detailed literature review.

Despite the remarkable progress noted above, we note that two major constraints remain. First, in large-scale collaborative networks, each agent is restricted to interacting with only a vanishing fraction of the overall population due to limited attention and communication capacity. This motivates an understanding of consensus formation under sparse communications. Second, the static model restricts the potential neighbors of each agent throughout the evolution. 
The assumption of the same contacts every day is somewhat unrealistic in many communication settings, as interaction partners typically change over time, especially in scenarios where opinions evolve dynamically.


\paragraph{Our contributions.} Suppose that the interaction graph $\gG$ is independently resampled from $\mathds{G}(N, p)$ every round. We focus on the \emph{critically sparse} regime\footnote{The $\mathds{G}(N,p)$ is connected if and only if $pN > \log(N)$,} where $pN \asymp \log(N)$ and each agent interacts with $\log(N)$ many peers per round. Our main results can be summarized as follows.

\begin{enumerate}
    \item \emph{Gaussian winner selection (Theorem~\ref{thm:er-critical-window-winner-selection}).} When $\Delta_0\sqrt p = O(1)$, the probability of blue unanimity is $\Phi(\Delta_0\sqrt{2p/\pi})+o(1)$. The consensus time is $\log N/\log\log N$, regardless of the final winner.

\item \emph{Consensus time and majority amplification (Theorem~\ref{thm:er-polylogarithmic-days-to-unanimity}).} 
When $\Delta_0 \sqrt{p}\gg 1$, blue wins with high probability. We also provide explicit upper and lower bounds on the consensus time, based on a sharp two-sided amplification estimate for blue advantage per round, which is typically of order $\sqrt{\log N}$.

\item \emph{Rapid consensus (Theorem~\ref{thm:er-two-day-unanimity}).} When the initial advantage $\Delta_0$ is above an explicit constant multiple of $N/\sqrt{\log N}$, blue unanimity is reached within two rounds with high probability. If a further lower bound on $\Delta_0$ is satisfied, blue unanimity is achieved after one round. The critical thresholds for two- and one round unanimity are explicitly characterized.

\item \emph{Correctness from random votes (Corollary~\ref{cor:er-random-vote-aggregation}).} When the initial opinions are random votes on colors, each correct with probability $1/2+\theta/\sqrt N$ for some fixed $\theta\ge0$, the correct unanimity occurs with probability $\Phi(2\theta)+o(1)$. This matches the asymptotic accuracy of centralized majority voting on the initial opinions, with consensus time again asymptotic to $\log N/\log\log N$.

\item \emph{Precise Gaussian approximation for flip probabilities (Lemma~\ref{lem:cc-gaussian-flip-probabilities}).} 
On the technical side, when both opinion groups are of linear size, for a fixed vertex, we obtain a more precise Gaussian approximation: the Berry--Esseen error bound is improved from $O((\log N)^{-1/2})$ to $O((\log N)^{-1})$, incorporating a specific correction for half-integer continuity. Our approach relies on Le Cam's inequality and maximal coupling to approximate the sparse binomial degree distributions with Poisson variables. By introducing independent uniform noise and using Fourier inversion, we obtain a continuity-corrected Gaussian approximation for the difference between these counts. Employing local Gaussian estimates and a precise telescoping argument, we achieve an $O(p\Delta_t/\log N)$ bound on the error for the flip probability difference within the critical window.
\end{enumerate}


\paragraph{Organization.} The rest of the paper is organized as follows. We introduce the mathematical definitions of our model in Section~\ref{sec:model}. Our main results are presented in Section~\ref{sec:main-results}, with a proof overview and technical contributions. Numerical experiments that complement our theoretical findings are provided in Section~\ref{sec:experiments}. Some possible future directions are discussed in Section~\ref{sec:future-directions}. Detailed proofs of the main results are deferred to the Appendix.

\section{Model and Problem Formulation}
\label{sec:model}

We introduce the precise definitions of the model and the problem here.

\subsection{Resampled majority dynamics}
Let $\gV=[N]=\{1,\ldots,N\}$ and $\rvy_t\in\{\pm1\}^{\gV}$ record the opinions after round $t \geq 0$, with blue represented by $+1$ and red by $-1$. Define the two camps by
\begin{equation}
\gB_t=\{v:\rvy_t(v)=+1\},\qquad
\gR_t=\{v:\rvy_t(v)=-1\}.
\label{eqn:blue-red-camps}
\end{equation}
At round $t \geq 0$, an undirected graph $\gG_{t}=(\gV,\gE_{t})$ is sampled from $\mathds{G}(N,p)$, independent of the initial opinions and all previous graphs.
\begin{assumption}[Critical sparsity]
\label{ass:sparse-er}
Suppose $p=d/N$ with $d=b\log(N)$ for some constant $b>1$.
\end{assumption}
For each $v\in\gV$ , we define its blue and red neighbors at round $t\geq 0$ by
\begin{equation}
N_t^{\gB}(v)=\#\{w\in\gB_t:\{v,w\}\in\gE_{t+1}\},\qquad
N_t^{\gR}(v)=\#\{w\in\gR_t:\{v,w\}\in\gE_{t+1}\}.
\label{eqn:neighbor-count}
\end{equation}
All opinions update simultaneously, with ties preserving the previous
opinion:
\begin{equation}
\ervy_{t+1}(v)= \indi{N_t^{\gB}(v)>N_t^{\gR}(v)} - \indi{N_t^{\gB}(v)<N_t^{\gR}(v)} + \ervy_t(v) \cdot \indi{N_t^{\gB}(v)=N_t^{\gR}(v)}.
\label{eqn:opinion-update}
\end{equation}
The population advantage satisfies
\begin{align}
\Delta_t&=|\gB_t|-|\gR_t|,
& |\gB_t|&=(N+\Delta_t)/2,
& |\gR_t|=(N-\Delta_t)/2.\label{eqn:camp-sizes-from-advantage-intro}
\end{align}
We define the blue and red unanimity times by
\begin{equation}
\tau_{\mathrm B}=\inf\{t\geq0:\gR_t=\emptyset\},\qquad
\tau_{\mathrm R}=\inf\{t\geq0:\gB_t=\emptyset\},
\label{eqn:er-blue-unanimity-time}
\end{equation}
with the convention $\inf\emptyset=\infty$. We denote the consensus time by $T=\min\{\tau_{\mathrm B},\tau_{\mathrm R}\}$. The opinion process is a time-homogeneous Markov chain whose next camp sizes depend in distribution only on the current camp sizes. Its natural filtration is
\begin{align}
\mathcal F_0&=\sigma(\rvy_0),\qquad
\mathcal F_t=\sigma(\rvy_0,\gG_1,\ldots,\gG_t),\quad t\ge1.
\label{eqn:filtration-def}
\end{align}
Conditional on $\mathcal F_t$, the next graph is independent of the past and has law $\mathds{G}(N,p)$.

\subsection{Initial opinions}
\label{sec:fixed-initial-model}
\label{sec:random-vote-model}
We consider two models for the initial opinions.

\noindent\textbf{Fixed initial configuration.}
For each $N$, fix any deterministic configuration
$\rvy_0\in\{\pm1\}^{\gV}$. We are interested in the probability of blue unanimity $\P(\tau_{\mathrm B}<\infty\mid\rvy_0)$, its red counterpart, and the distribution of $T$. These quantities describe the agreement and preservation of an initial majority.

\noindent\textbf{Random Initial Vote Model.}
Fix a correct opinion $z\in\{-1,+1\}$ and a constant $\theta\ge0$. Independently, over vertices and all interaction graphs, we draw the initial opinions according to
\begin{equation}
\P(\rvy_0(v)=z)=1/2 + \theta/\sqrt{N},
\qquad v\in[N],
\label{eq:random-initial-votes}
\end{equation}
The correct opinion $z$ is used to evaluate the outcome and is not supplied to the update rule. The votes are drawn once; subsequent rounds use only the shared majority dynamics.
We ask for the probability of correct unanimity,
$\P(\exists t\ge0:\rvy_t\equiv z)$, and the consensus time $T$.
Here, probabilities include both the initial votes and the interaction graphs. The benchmark is centralized majority voting on the initial opinions, with a rule for breaking ties.

\section{Main results}
\label{sec:main-results}

We first present our main results, followed by a technical overview of the proof.

\subsection{Rapid unanimity: quantitative bounds}
\label{sec:rapid-completion}
Define the following constants
\begin{align}
r_*&=\big(1-b^{-1}\sqrt{2b-1}\big)/2,
&
K_2^{\mathrm{ER}}(b)&=b^{-1/2}\cdot\Phi^{-1}(1-r_*).
\label{eqn:K2-def}
\end{align}
Here, $r_*\in(0,1/2)$ is the solution to the equation $b(\sqrt{1-r}-\sqrt{r})^2=1$. Below, we provide sufficient conditions for rapid unanimity.

\begin{theorem}[Two- and one-day unanimity]
\label{thm:er-two-day-unanimity}
Under Assumption~\ref{ass:sparse-er}, the following statements hold for all
sufficiently large \(N\).
\begin{enumerate}[label=\textup{(\roman*)}]
\item \textup{(Two days.)} Fix \(K>K_2^{\mathrm{ER}}(b)\). There exists \(\xi=\xi(b,K)>0\) such that $\gR_2=\emptyset$ with probability at least $1-2\exp(-\left(\log N\right)^2)-2N^{-\xi}$ whenever
\begin{align}
\Delta_0\ge KN/\sqrt{\log N},
\label{eqn:er-two-day-unanimity-condition}
\end{align}
\item \textup{(One day.)} Fix \(0<r<r_*\). If
\begin{align}
\Delta_0\ge (1-2r)N,
\label{eqn:er-one-day-unanimity-condition}
\end{align}
then there exists \(\xi=\xi(b,r)>0\) such that $\gR_1=\emptyset$ with probability at least $1-2N^{-\xi}$.
\end{enumerate}
\end{theorem}

The proof is in Appendix~\ref{app:rapid-proofs}. Intuitively, a large advantage first reduces the minority to a sufficiently small fraction; the next independent graph removes the remaining minority with high probability.

\subsection{Amplification and the unanimity-time bracket}
An initial advantage above the critical scale dominates the first-round
fluctuations. The next theorem quantifies its subsequent amplification
and brackets the blue-unanimity time $\tau_{\mathrm B}$.

Fix
\(K>K_2^{\mathrm{ER}}(b)\), and let
\(c_1=c_1(b,K)>0\), \(C_2=C_2(b,K)>0\), and \(c=c(b,K)>0\) be the
constants supplied by Lemma~\ref{lem:stepwise-advantage-amplification} with
upper-window coefficient \(K\). Let \(\xi=\xi(b,K)>0\) be the constant in
Theorem~\ref{thm:er-two-day-unanimity}. For any \(s\geq0\), define
\begin{align}
\eta_N(s)
\coloneqq
8\exp\left(-cs^2\log(N)/{N}\right)
+2\exp\left(-(\log N)^2\right)
+2N^{-\xi}.
\label{eqn:er-polylog-joint-failure-probability}
\end{align}
For \(\Delta_0>0\), define the deterministic horizons
\begin{subequations}
\begin{align}
\underline{T}
&\coloneqq
\max\left\{
0,
\left\lfloor
\log\left(KN/(\Delta_0\sqrt{\log N})\right)
\big/\log\left(C_2\sqrt{\log N}\right)
\right\rfloor
\right\},
\label{eqn:er-polylog-necessary-time}\\
\overline{T}
&\coloneqq
\max\left\{
0,
\left\lceil
\log\left(KN/(\Delta_0\sqrt{\log N})\right)
\big/\log\left(c_1\sqrt{\log N}\right)
\right\rceil
\right\}.
\label{eqn:er-polylog-sufficient-time}
\end{align}
\end{subequations}

\begin{theorem}[Logarithmic time to blue unanimity]
\label{thm:er-polylogarithmic-days-to-unanimity}
Under Assumption~\ref{ass:sparse-er}, suppose that
\(\gR_0\neq\emptyset\), \(1\leq h_N=o(\sqrt{\log N})\), and
\begin{align}
\Delta_0\geq \sqrt{N}/h_N.
\label{eqn:er-polylog-initial-advantage-condition}
\end{align}
Then, for all sufficiently large \(N\),
\begin{align}
\P\left(
\underline{T}+1
\leq\tau_{\mathrm B}\leq
\overline{T}+2
\,\middle|\,
\rvy_0
\right)
\geq
1-\eta_N(\Delta_0).
\label{eqn:er-polylog-unanimity-time-bracket}
\end{align}
\end{theorem}
\begin{corollary}\label{cor:er-polylogarithmic-days-to-unanimity-asymptotic}
Under the conditions of
Theorem~\ref{thm:er-polylogarithmic-days-to-unanimity}, we further suppose that
\[
\Delta_0\asymp \sqrt{N}/{h_N}.
\]
Then for every fixed \(0<\delta<1\) and all sufficiently large
\(N\),
\begin{align}
&\P\left(
\gR_{\left\lfloor(1-\delta)\log N/\log\log N\right\rfloor}
\neq\emptyset,\
\gR_{\left\lceil(1+\delta)\log N/\log\log N\right\rceil}
=\emptyset
\,\middle|\,
\rvy_0
\right)
\geq
1-\eta_N(\Delta_0).
\label{eqn:er-polylog-unanimity-time-asymptotics}
\end{align}
\end{corollary}

The theorem and corollary are proved in Appendix~\ref{app:amplification-proofs}.
The theorem's initial-advantage condition is equivalent to
$\Delta_0\sqrt p\to\infty$: for any such sequence, one may choose
$h_N=\max\{1,\sqrt N/\Delta_0\}=o(\sqrt{\log N})$.
In this regime, $\eta_N(\Delta_0)=o(1)$, so blue wins with high probability.
The lower and upper time bounds describe the fastest and slowest
amplification rates allowed by our estimates before the process reaches
$KN/\sqrt{\log N}$, followed by at most two completion rounds.
In the corollary's narrower regime,
$\tau_{\mathrm B}/(\log N/\log\log N)\to1$ in probability.

\subsection{Critical-window winner selection: uniform errors and horizon}
At the critical scale $\Delta_0=O(\sqrt{N/\log N})$, the first update's
mean and fluctuations have comparable magnitudes. The normalized
advantage $\Delta_1/\sqrt N$ has a nondegenerate Gaussian limit
(Lemma~\ref{lem:critical-window-one-step-clt}), whose sign determines the
eventual winner. Outside a vanishing window around zero, its magnitude
enters the amplification regime of
Theorem~\ref{thm:er-polylogarithmic-days-to-unanimity}.

To state the result, retain the fixed \(K\) and associated constants introduced above, and let \(h_N\) satisfy the following conditions:
\[
h_N\to\infty
\quad\text{and}\quad
h_N=o(\sqrt{\log N}).
\]
For \(C_0>0\), the total error, combining the one-step Gaussian
approximation, the probability of the vanishing central window, and the
subsequent amplification failure probability, is
\begin{align}
\mathfrak r_N(C_0)
&\coloneqq
\rho_N(C_0)
+1/{h_N}
+\eta_N({\sqrt N}/{h_N}),
\label{eq:critical-window-total-error}
\end{align}
where \(\rho_N(C_0)\) is the uniform one-step Gaussian approximation error
defined in \eqref{eq:critical-window-one-step-clt}. The corresponding
worst-case amplification horizon from the threshold \(\sqrt N/h_N\) is
\begin{align}
T_{N}
&\coloneqq
\max\left\{
0,
\left\lceil
{\log\left(Kh_N\sqrt{N/\log N}\right)}
\big/
{\log\left(c_1\sqrt{\log N}\right)}
\right\rceil
\right\}.
\label{eq:critical-window-amplification-horizon}
\end{align}

\begin{theorem}[Critical-window winner selection]
\label{thm:er-critical-window-winner-selection}
Under Assumption~\ref{ass:sparse-er}, retain the preceding choices of
\(K\) and \(h_N\), and fix \(C_0>0\). Consider deterministic initial
configurations satisfying
\begin{align}
x_0
\coloneqq
\Delta_0\sqrt{{2b\log N}/{\pi N}},
\qquad
|x_0|\le C_0.
\label{eq:critical-window-condition}
\end{align}
Uniformly over all such configurations, the blue- and red-unanimity probabilities at \(T_{N}+3\) satisfy
\begin{subequations}
\begin{align}
\P\left(
\gR_{T_{N}+3}=\emptyset
\,\middle|\,
\rvy_0
\right)
&=
\Phi(x_0)+O\left(\mathfrak r_N(C_0)\right),
\label{eq:critical-window-blue-winning-probability}\\
\P\left(
\gB_{T_{N}+3}=\emptyset
\,\middle|\,
\rvy_0
\right)
&=
\Phi(-x_0)+O\left(\mathfrak r_N(C_0)\right).
\label{eq:critical-window-red-winning-probability}
\end{align}
\end{subequations}
In particular, \(\mathfrak r_N(C_0)=o(1)\), and the time horizon satisfies \(T_{N}+3=(1+o(1))\log N/\log\log N\).
\end{theorem}

Since the two unanimous states are absorbing and the probability that
neither has been reached by $T_N+3$ is $O(\mathfrak r_N(C_0))$, the theorem
also gives the eventual blue- and red-winning probabilities
$\Phi(x_0)+O(\mathfrak r_N(C_0))$ and
$\Phi(-x_0)+O(\mathfrak r_N(C_0))$, respectively.
Thus, if $x_0\to x>0$, the initial blue majority loses with a limiting probability of $\Phi(-x)>0$; at balance, each opinion wins with a probability of $1/2+o(1)$. The following corollary shows that the same leading consensus-time scale holds throughout the critical window.

\begin{corollary}[Critical-window consensus time]
\label{cor:er-critical-window-consensus-time}
Under Assumption~\ref{ass:sparse-er}, fix \(C_0>0\). For every fixed
\(0<\varepsilon<1\), uniformly over deterministic initial configurations
satisfying \eqref{eq:critical-window-condition},
\begin{align}
\P\left(
(1-\varepsilon)\log N/\log\log N
\le \min\{\tau_{\mathrm B},\tau_{\mathrm R}\}
\le (1+\varepsilon)\log N/\log\log N
\,\middle|\,\rvy_0
\right)=1-o(1).
\label{eq:critical-window-consensus-time-asymptotic}
\end{align}
\end{corollary}

The theorem and consensus-time corollary are proved in
Appendix~\ref{app:gaussian-proof}. Figure~\ref{fig:er-phase-diagram}
illustrates the agreement regimes for fixed initial configurations.

\begin{figure}[h]
\centering
\captionsetup{skip=6pt}
\resizebox{0.98\linewidth}{!}{%
\begin{tikzpicture}[
    x=1cm,
    y=1cm,
    axis/.style={-{Latex[length=4mm,width=3mm]},thick},
    tick/.style={black,line width=0.8pt},
    thresholdarrow/.style={
        -{Latex[length=1.5mm,width=1.2mm]},
        black!70,
        line width=0.65pt
    },
    regime/.style={line width=4pt},
    critical/.style={regime,gray!85},
    polylog/.style={regime,green!55!black},
    2day/.style={regime,orange!85!black},
    1day/.style={regime,cyan!80!blue},
    threshold/.style={font=\footnotesize,align=center,anchor=south}
]

\def\xA{0.0}
\def\xB{2.8}
\def\xC{8.0}
\def\xD{11.8}
\def\xE{13.8}
\def\yThreshold{0.72}

\draw[axis] (\xA,0) -- (\xE+0.4,0);
\node[font=\small,anchor=west] at (\xE+0.5,0) {$\Delta_0/N$};

\draw[critical] (\xA,0) -- (\xB,0);
\draw[polylog] (\xB,0) -- (\xC,0);
\draw[2day] (\xC,0) -- (\xD,0);
\draw[1day] (\xD,0) -- (\xE,0);

\foreach \x in {\xB,\xC,\xD} {
    \draw[tick] (\x,0.2) -- (\x,-0.2);
    \draw[thresholdarrow] (\x,0.24) -- (\x,0.62);
}

\node[threshold] at (\xB,\yThreshold)
    {$\displaystyle \frac{\Delta_0}{N}
      \asymp\frac{1}{\sqrt{N\log N}}$};
\node[threshold] at (\xC,\yThreshold)
    {$\displaystyle \frac{\Delta_0}{N}
      =\frac{K_2^{\mathrm{ER}}(b)}{\sqrt{\log N}}$};
\node[threshold] at (\xD,\yThreshold)
    {$\displaystyle \frac{\Delta_0}{N}=1-2r_*$};

\begin{scope}[shift={(0,-1.55)}]
    \draw[1day] (0.0,0.25) -- (0.7,0.25);
    \node[font=\footnotesize,anchor=west] at (0.85,0.25)
        {Thm.~\ref{thm:er-two-day-unanimity} \textup{(ii)}: at most \(1\) round};

    \draw[2day] (7.1,0.25) -- (7.8,0.25);
    \node[font=\footnotesize,anchor=west] at (7.95,0.25)
        {Thm.~\ref{thm:er-two-day-unanimity} \textup{(i)}: at most \(2\) rounds};

    \draw[polylog] (0.0,-0.45) -- (0.7,-0.45);
    \node[font=\footnotesize,anchor=west] at (0.85,-0.45)
        {Thm.~\ref{thm:er-polylogarithmic-days-to-unanimity}: logarithmic};

    \draw[critical] (7.1,-0.45) -- (7.8,-0.45);
    \node[font=\footnotesize,anchor=west] at (7.95,-0.45)
        {Thm.~\ref{thm:er-critical-window-winner-selection}: critical selection};
\end{scope}
\end{tikzpicture}%
}
\caption{Schematic agreement regimes for the resampled ER model with
fixed $b>1$ and $\Delta_0\ge0$; exchanging blue and red gives the case
$\Delta_0<0$. Horizontal lengths are illustrative. The markers show the
critical Gaussian scale and the constants in the sufficient two- and
one-round conditions. The latter guarantees require fixed margins
$K>K_2^{\mathrm{ER}}(b)$ and $r<r_*$.}
\label{fig:er-phase-diagram}
\end{figure}

\subsection{Correctness under random initial votes}
\label{sec:random-vote-correctness}
We now evaluate agreement against the correct opinion in the random
initial vote model. With high probability, the initial majority has an
advantage in the amplification regime. Applying the agreement bounds
conditionally on those votes and then averaging gives the following
correctness guaranty.

\begin{corollary}[Correctness from random initial votes]
\label{cor:er-random-vote-aggregation}
Under Assumption~\ref{ass:sparse-er}, consider the random initial vote
model of Section~\ref{sec:random-vote-model}, with a fixed correct opinion
$z$ and parameter $\theta\ge0$. Then
\begin{align}
\P\bigl(\exists t\ge0:\rvy_t\equiv z\bigr)
=\Phi(2\theta)+o(1).
\label{eq:random-vote-correct-unanimity}
\end{align}
Moreover, \eqref{eq:critical-window-consensus-time-asymptotic} holds as well under the current random initial vote model.
\end{corollary}

\subsection{Proof Overview and Technical Contributions}
\label{sec:proof-overview}
We summarize the proof outline here. Readers may refer to Appendices~\ref{app:gaussian-proof}--\ref{app:probabilistic-tools} for the complete proof.

\paragraph{1. The first update selects a sign.}
In the critical window, the mean estimate and the one-step central limit theorem in \citet{berkowitz2020central} lead to
\begin{equation}
\sup_{z\in\mathbb R}
\left|\P\!\left(\Delta_1/\sqrt N\le z\mid\rvy_0\right)
-\Phi(z-x_0)\right|=o(1),
\label{eq:main-first-step-clt}
\end{equation}
uniformly for bounded $x_0$. Thus a critical initial advantage of order
$\sqrt{N/\log N}$ becomes a random advantage of order $\sqrt N$.
The Gaussian law supplies both the sign probabilities and
anti-concentration: for every $h_N\to\infty$,
$\P(|\Delta_1|<\sqrt N/h_N\mid\rvy_0)=o(1)$.
Choosing such a sequence with $h_N=o(\sqrt{\log N})$ ensures that
\(
|\Delta_1|\sqrt p\ge \sqrt{b\log N}/h_N\to\infty
\)
outside this event. We therefore apply Theorem~\ref{thm:er-polylogarithmic-days-to-unanimity}
starting from the configuration after round one, interchanging
blue and red if $\Delta_1<0$. We use the CLT of \citet{berkowitz2020central} to obtain the
first-round Gaussian approximation with an $o(1)$ error.
For subsequent rounds, we use concentration inequalities to
obtain explicit bounds on the probability that amplification fails.
For $h_N=(\log N)^{1/4}$, Gaussian tightness also gives
$|\Delta_1|\le h_N\sqrt N$ with high probability. Applying both sides of
the amplification bracket uniformly between these endpoints yields the
consensus-time asymptotic in Corollary~\ref{cor:er-critical-window-consensus-time}.

\paragraph{2. Sparse drift requires a local comparison.}
Let $p_t^{\gR}$ and $p_t^{\gB}$ denote the conditional probabilities of
red-to-blue and blue-to-red flips. The exact drift identity is
\begin{equation}
\E[\Delta_{t+1}\mid\rvy_t]
=N(p_t^{\gR}-p_t^{\gB})
+\Delta_t(1-p_t^{\gR}-p_t^{\gB}).
\label{eq:main-drift-identity}
\end{equation}
Approximating each flip probability separately by a Gaussian and subtracting
is insufficient near balance: the additive errors can exceed the difference
being estimated. We instead express this difference as a sum of local
probabilities.

Condition on $\rvy_t$, and suppose that $\Delta_t\ge0$ are both nonempty camps. For each integer $0\le j\le\Delta_t$, introduce auxiliary
variables with distributions
\[
\begin{aligned}
\rD_{t,j}^{\gR}
&\stackrel{\mathrm d}{=}
\Bin(|\gR_t|+j,p)-\Bin(|\gR_t|-1,p),\\
\rD_{t,j}^{\gB}
&\stackrel{\mathrm d}{=}
\Bin(|\gR_t|,p)-\Bin(|\gR_t|-1+j,p),
\end{aligned}
\]
where the two binomial variables in each difference are independent.
At $j=0$, the two differences have the same distribution. At
$j=\Delta_t$, their positive-tail probabilities are $p_t^{\gR}$ and
$p_t^{\gB}$, respectively. Increasing $j$ by one adds an independent
$\Ber(p)$ trial to the first difference and subtracts one from the
second. Their positive-tail probabilities, therefore, change by
$p\P(\rD_{t,j}^{\gR}=0)$ and $-p\P(\rD_{t,j}^{\gB}=1)$.
Telescoping gives the exact identity
\begin{equation}
p_t^{\gR}-p_t^{\gB}
=
p\sum_{j=0}^{\Delta_t-1}
\left[
\P(\rD_{t,j}^{\gR}=0)
+\P(\rD_{t,j}^{\gB}=1)
\right],
\label{eq:main-flip-telescoping}
\end{equation}
where the probabilities refer to distributions conditional on $\rvy_t$, and the sum is zero when $\Delta_t=0$.

A Poisson coupling and a local Gaussian estimate control these
probabilities. Uniformly for
$0\le\Delta_t\le KN/\sqrt{\log N}$ with fixed $K$, each local
approximation has error $O((\log N)^{-1})$. Its contribution after
summation in \eqref{eq:main-flip-telescoping} is therefore
$O(p\Delta_t/\log N)$; see
Lemma~\ref{lem:cc-gaussian-flip-probabilities}.
In particular, for every fixed $C_0>0$, uniformly over initial
configurations with $|x_0|\le C_0$,
\begin{equation}
p_0^{\gR}-p_0^{\gB}
=\sqrt{\frac2\pi}\,
\frac{p\Delta_0}{\sqrt{(N-1)p(1-p)}}
+O\!\left(\frac{p|\Delta_0|}{\log N}\right).
\label{eq:main-local-difference}
\end{equation}
For negative advantages the formula follows by color symmetry.
Substitution into \eqref{eq:main-drift-identity} identifies the critical
mean $\E[\Delta_1\mid\rvy_0]/\sqrt N=x_0+o(1)$.
More generally, throughout a fixed window
$\sqrt{N/\log N}\le\Delta_t\le KN/\sqrt{\log N}$, the conditional mean
is bounded above and below by positive constant multiples of
$\sqrt{\log N}\,\Delta_t$.

\paragraph{3. Dependent updates concentrate and amplify.}
Conditional on $\rvy_t$, the indicator that a vertex is red after the
next update is a function of its incident independent edges.
Each edge enters only the two endpoint indicators, so these indicators
form a read-$2$ family.
The read-$2$ concentration bound
\citep{gavinsky2015tail} gives, for every $s>0$,
\[
\P\!\left(
\left||\gR_{t+1}|-\E[|\gR_{t+1}|\mid\rvy_t]\right|\ge s
\,\middle|\,\rvy_t
\right)
\le 2\exp\!\left(-{s^2}/{N}\right).
\]
Since $\Delta_{t+1}=N-2|\gR_{t+1}|$, setting
$\mu_t=\E[\Delta_{t+1}\mid\rvy_t]$ yields
\begin{equation}
\P\!\left(
|\Delta_{t+1}-\mu_t|\ge u
\,\middle|\,\rvy_t
\right)
\le 2\exp\!\left(-{u^2}/{4N}\right),
\qquad u>0.
\label{eq:main-advantage-concentration}
\end{equation}
Combining \eqref{eq:main-advantage-concentration} with the
conditional-mean bounds above, and taking $u\asymp \sqrt{\log N}\,\Delta_t$, gives that for suitable constants $c_1,C_2,c>0$ depending only on $b,K$, we have
\begin{equation}
\P\!\big(
c_1\sqrt{\log N}\,\Delta_t\le\Delta_{t+1}
\le C_2\sqrt{\log N}\,\Delta_t
\,|\,\rvy_t
\big)
\ge
1-2\exp\!\left(-c{\Delta_t^2\log(N)}/{N}\right)
\label{eq:main-amplification-step}
\end{equation}
throughout
$\sqrt{N/\log N}\le\Delta_t\le KN/\sqrt{\log N}$,
The failure probability vanishes when the advantage is asymptotically
above the lower endpoint of this window. Along successful steps,
the advantage grows geometrically, making the later failure probabilities
summable. Iterating the two bounds yields the deterministic time bracket
in Theorem~\ref{thm:er-polylogarithmic-days-to-unanimity}.

\paragraph{4. Fresh communication completes the decision.}
We first identify a minority fraction that disappears after one round. Suppose that $|\gR_t|\le rN$ for a fixed $0<r<1/2$, and define
\[
J_b(r)=b\bigl(\sqrt{1-r}-\sqrt r\bigr)^2.
\]
For each agent, the probability of being red after the next update is nondecreasing in the current red-camp size. It therefore suffices
to consider $|\gR_t|=\lfloor rN\rfloor$.
The one-vertex large-deviation estimate bounds this probability by
$N^{-J_b(r)+o(1)}$. A union bound over vertices gives
\begin{equation}
\P\!\left(
\gR_{t+1}\ne\emptyset
\,\middle|\,\rvy_t
\right)
\le N^{1-J_b(r)+o(1)},
\label{eq:main-extinction-bound}
\end{equation}
uniformly over configurations with $|\gR_t|\le rN$.
The right-hand side vanishes when $J_b(r)>1$, equivalently $r<r_*$.
Thus $J_b(r_*)=1$ identifies where the exponent in
\eqref{eq:main-extinction-bound} changes sign, yielding the sufficient
condition $r<r_*$ in the one-round statement of
Theorem~\ref{thm:er-two-day-unanimity}; the quantitative failure probability bound is given in
Lemma~\ref{lem:one-day-extinction}.

To reach this small-minority regime, suppose that
$\Delta_t\ge KN/\sqrt{\log N}$ for fixed $K>0$.
At $\Delta_t=KN/\sqrt{\log N}+O(1)$, a Gaussian estimate gives
an expected red fraction $\Phi(-\sqrt b\,K)+o(1)$.
Monotonicity extends this to all larger advantages:
\begin{equation}
{\E[|\gR_{t+1}|\mid\rvy_t]}/{N}
\le q_K+o(1),
\qquad
q_K=\Phi(-\sqrt b\,K).
\label{eq:main-minority-reduction}
\end{equation}
The condition $K>K_2^{\mathrm{ER}}(b)$ is precisely $q_K<r_*$.
Choose a fixed $r$ with $q_K<r<r_*$. The concentration bound above
then gives $|\gR_{t+1}|\le rN$ with high probability, as quantified
in Lemma~\ref{lem:one-step-reduction}.
Applying \eqref{eq:main-extinction-bound} to the next independent
graph completes blue unanimity in one further round.

Independence of future graphs allows the amplification and completion
estimates to be applied conditionally to the random configuration
selected after round one. Combining this with
\eqref{eq:main-first-step-clt} transfers the Gaussian sign probability
to the eventual winner and proves
Theorem~\ref{thm:er-critical-window-winner-selection}.

\section{Numerical Experiments}
\label{sec:experiments}

We simulate the majority dynamics with independently resampled ER graphs.
The experiments examine Gaussian winner selection, amplification of the
population advantage, and rapid unanimity. All runs use $N=10^4$,
$p=b\log N/N$, and deterministic, parity-admissible initial camp sizes.

\paragraph{Gaussian winner selection.}
For each $b\in\{1.5,2,4\}$, we use $13$ equally spaced target values of
$x_0$ in $[-3,3]$. We round
$\Delta_0=\sqrt{\pi N/(2b\log N)}\,x_0$
to the nearest integer satisfying $\Delta_0\equiv N\pmod 2$ and plot the
resulting actual value of $x_0$. Each setting uses $200$ independent trajectories, simulated until either opinion reaches unanimity or through round $100$.
Figure~\ref{fig:er-critical-window-experiment} compares these empirical probabilities with $\Phi(x_0)$. Among the three densities, the curves track the Gaussian prediction, with a crossover near balance, supporting
the normalization in Theorem~\ref{thm:er-critical-window-winner-selection}.
\begin{figure}[h]
\centering
\includegraphics[width=0.32\linewidth]{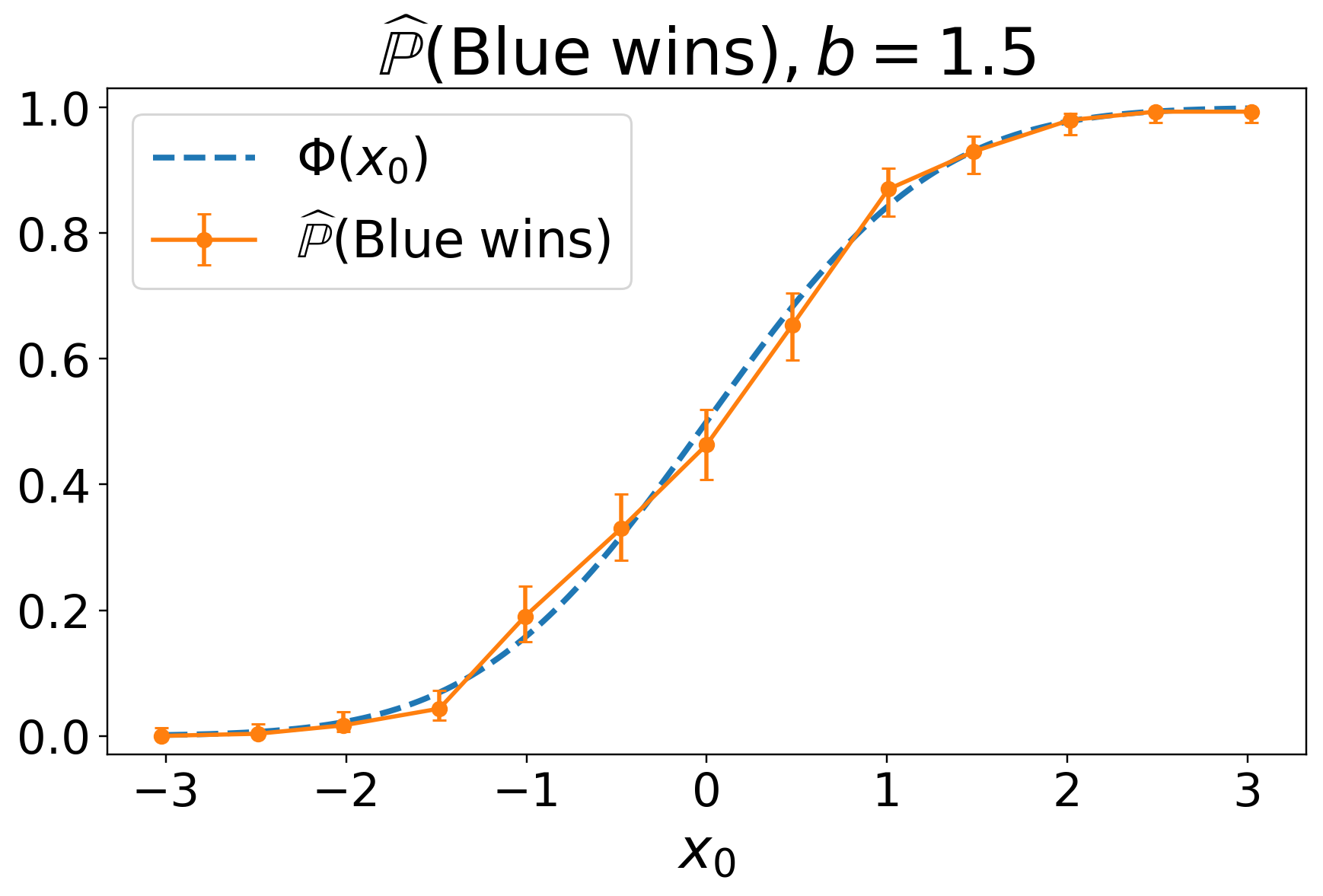}\hfill
\includegraphics[width=0.32\linewidth]{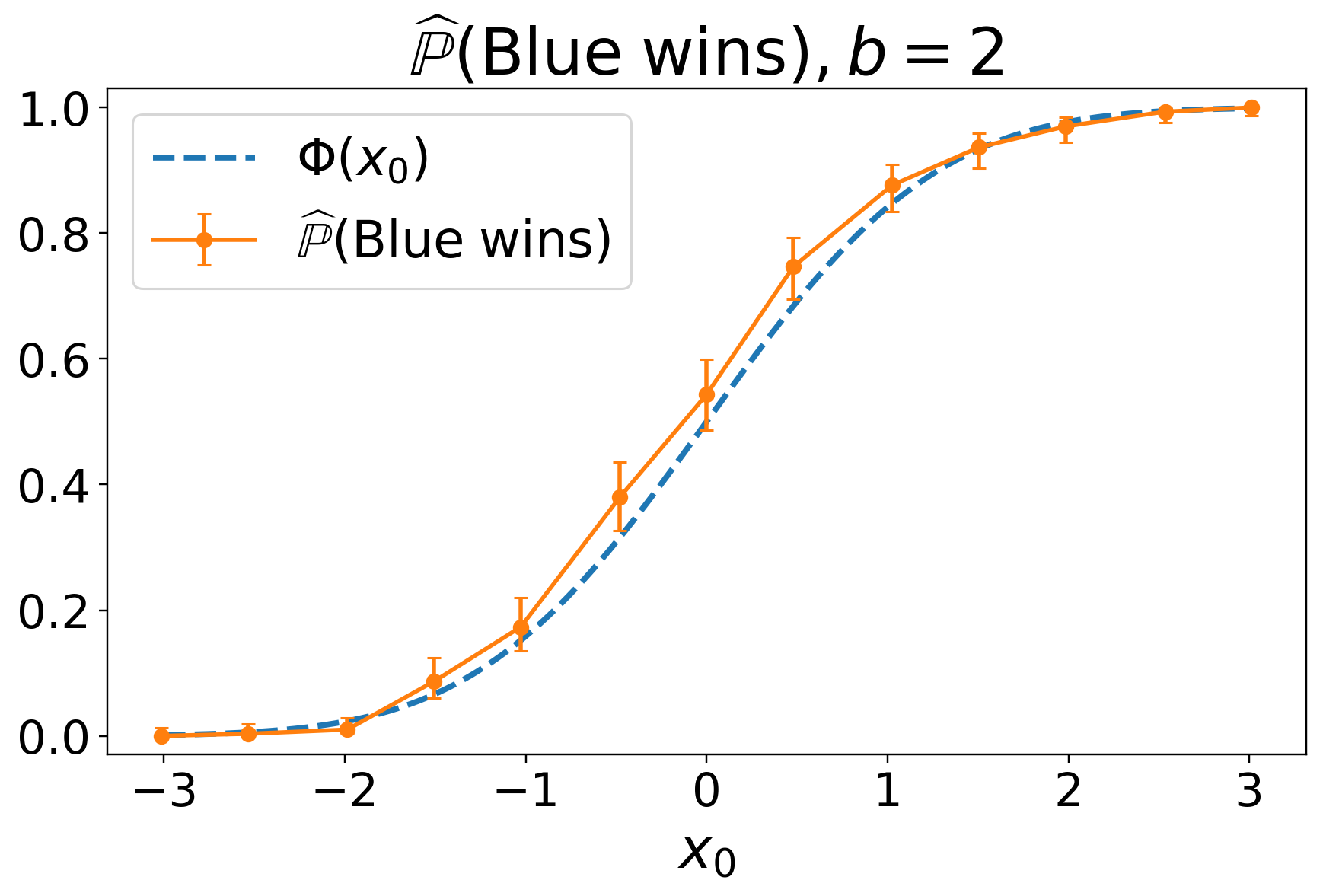}\hfill
\includegraphics[width=0.32\linewidth]{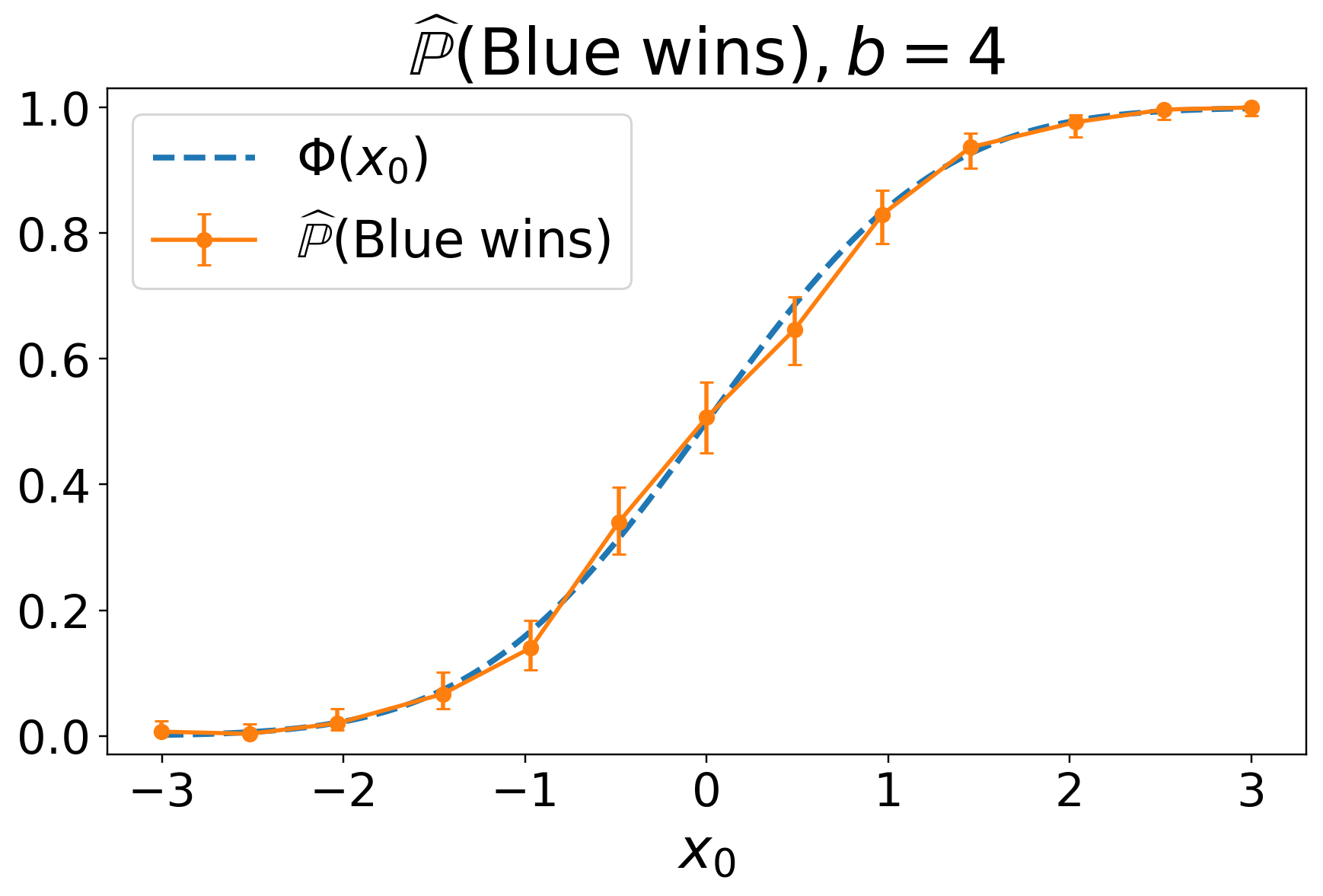}
\caption{Gaussian winner selection for $b=1.5$ (left), $b=2$ (middle),
and $b=4$ (right), with $200$ independent paths per setting.
Solid curves show empirical probabilities of blue unanimity by round
$100$, and dashed curves show
$\Phi(x_0)$. The pointwise $95\%$ Wilson intervals quantify Monte Carlo uncertainty at the simulated population size. }
\label{fig:er-critical-window-experiment}
\end{figure}

\paragraph{Amplification and entrance time.}
\begin{figure}[h]
\centering
\includegraphics[width=0.47\linewidth]{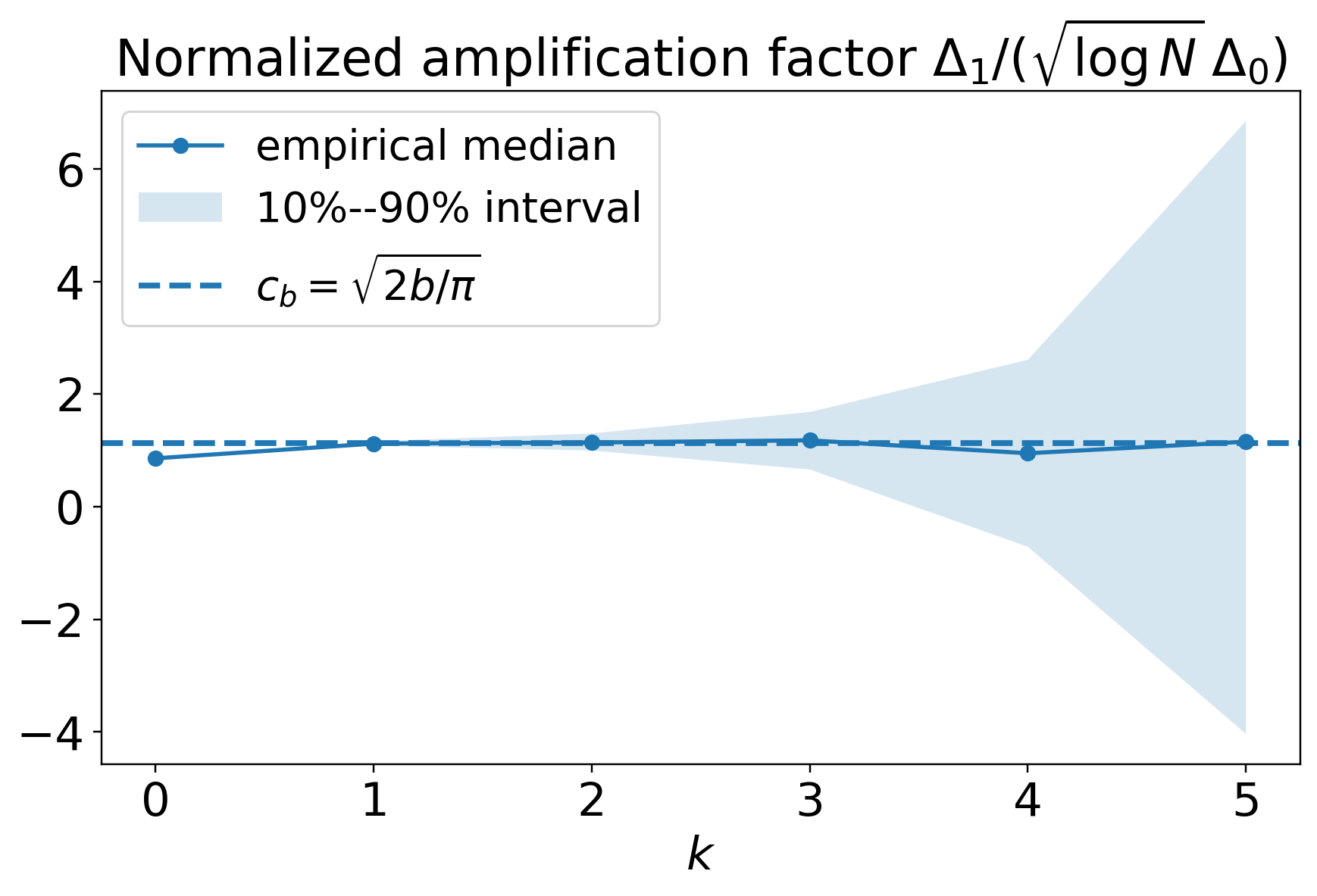}\hfill
\includegraphics[width=0.47\linewidth]{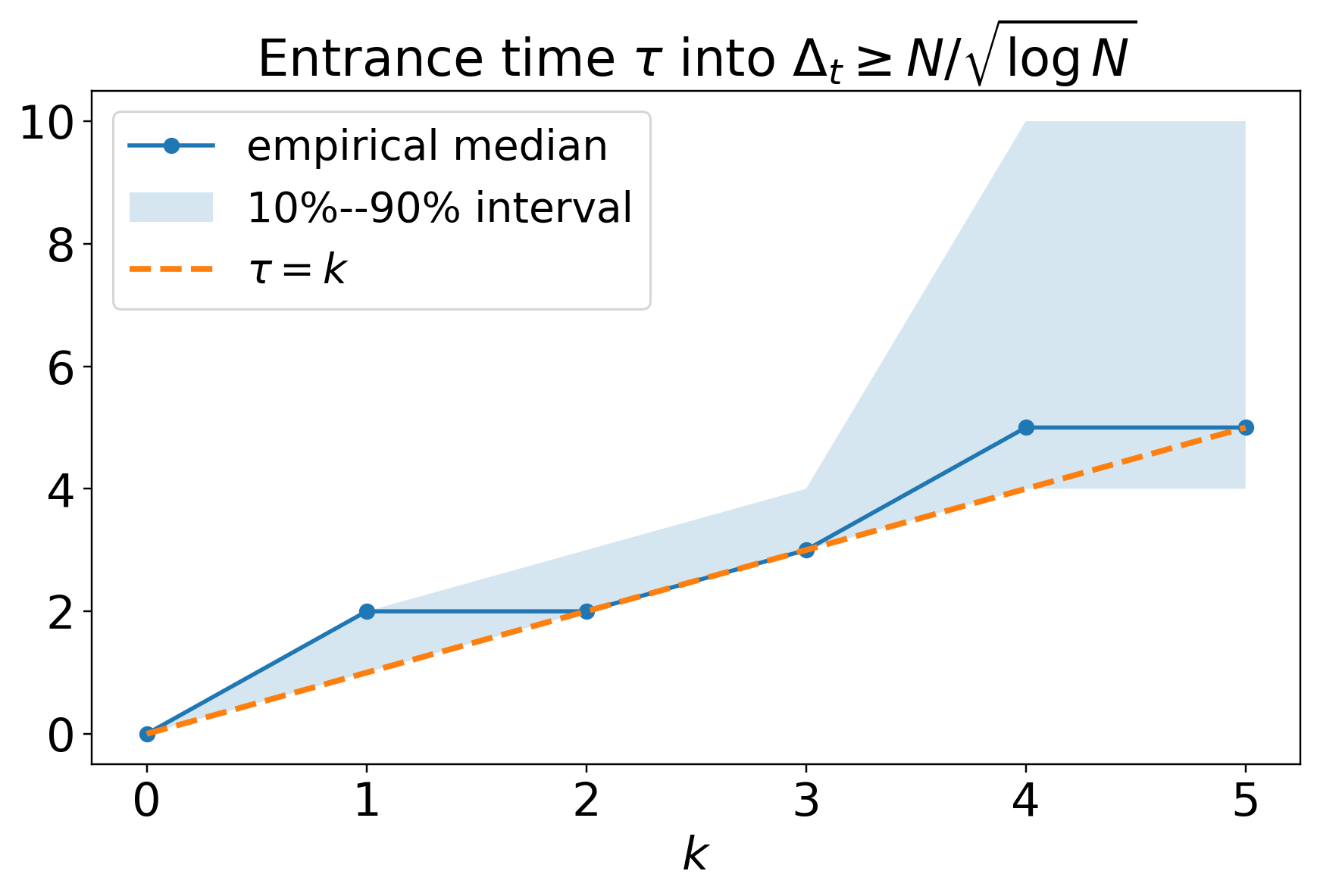}
\caption{Amplification and entrance into $N/\sqrt{\log N}$, with $b=2$
and $300$ paths per $k$. Solid curves are medians, and shaded regions
are $10\%$--$90\%$ pathwise intervals. Dashed references are
$c_b=\sqrt{2b/\pi}$ (left) and $\tau=k$ (right).
Paths run through round $9$; nonentrance is recorded at $\tau=10$.}
\label{fig:er-stepwise-experiment}
\end{figure}

For $b=2$ and $k=0,\ldots,5$, choose
\[
\Delta_0\approx
N/\left[\sqrt{\log N}(c_b\sqrt{\log N})^k\right],
\qquad
c_b=\sqrt{2b/\pi},
\]
rounding the initial blue camp upward to an integer.
Define the entrance time
\[
\tau=\inf\{t\ge0:\Delta_t\ge N/\sqrt{\log N}\}.
\]
Near balance, the linear-response prediction
$\Delta_{t+1}\approx c_b\sqrt{\log N}\,\Delta_t$
suggests entrance near round $k$.
Figure~\ref{fig:er-stepwise-experiment} shows the normalized first-step
growth and the entrance time, with references $c_b$ and $\tau=k$.
Under each setting, $300$ independent paths are simulated through round $9$. A path that has not reached the target scale is recorded at $\tau=10$.
Solid curves are empirical medians, and shaded bands are
$10\%$--$90\%$ pathwise intervals.
The variation increases for larger $k$, when the initial advantage
approaches or falls below the critical scale.
These settings illustrate the boundary of reliable amplification.
Entrance into $N/\sqrt{\log N}$ alone does not imply the two-round guaranty in Theorem~\ref{thm:er-two-day-unanimity}, which requires
a coefficient $K>K_2^{\mathrm{ER}}(b)$.

\paragraph{Rapid unanimity.}
\begin{figure}[htbp]
\centering
\includegraphics[width=0.32\linewidth]{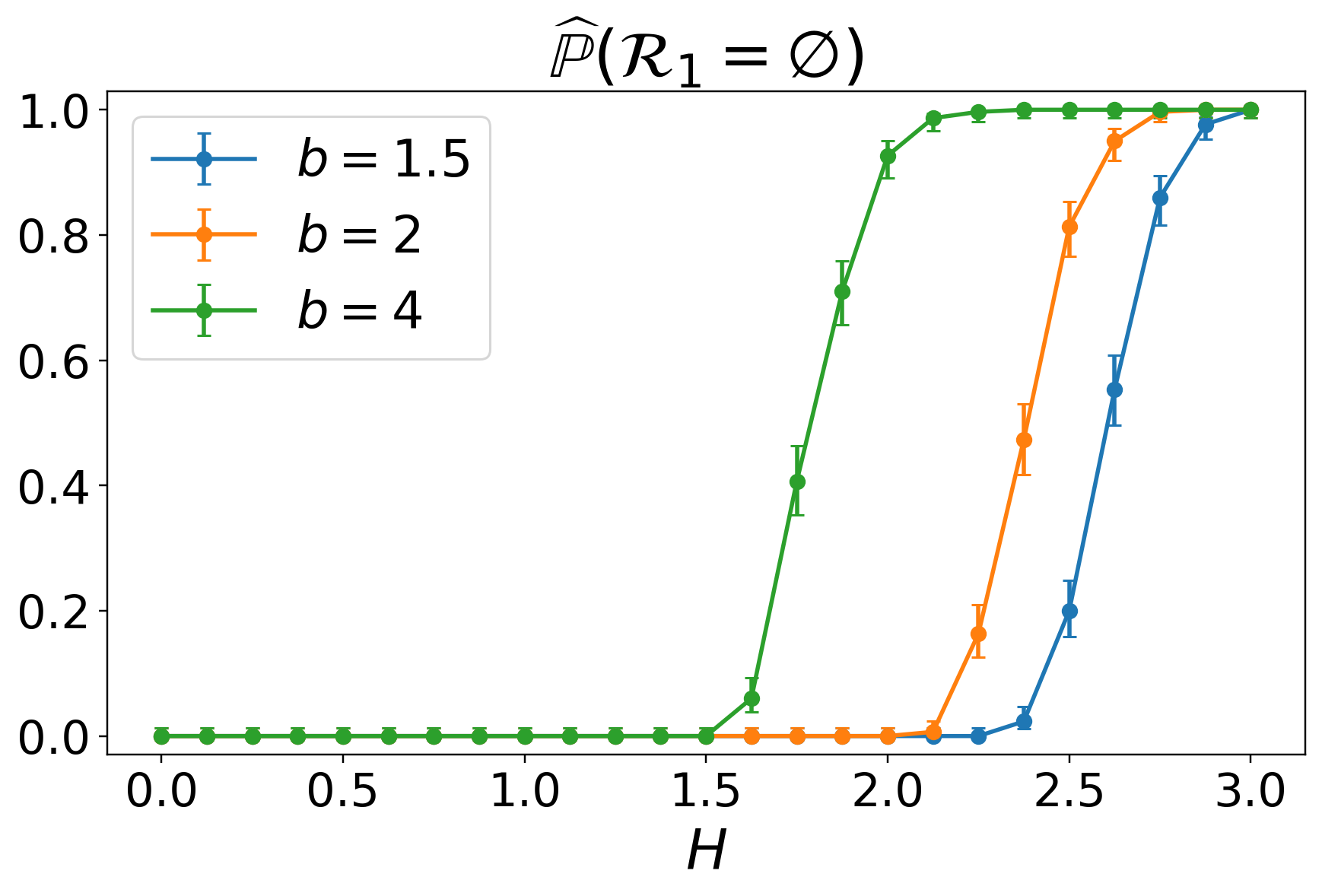}\hfill
\includegraphics[width=0.32\linewidth]{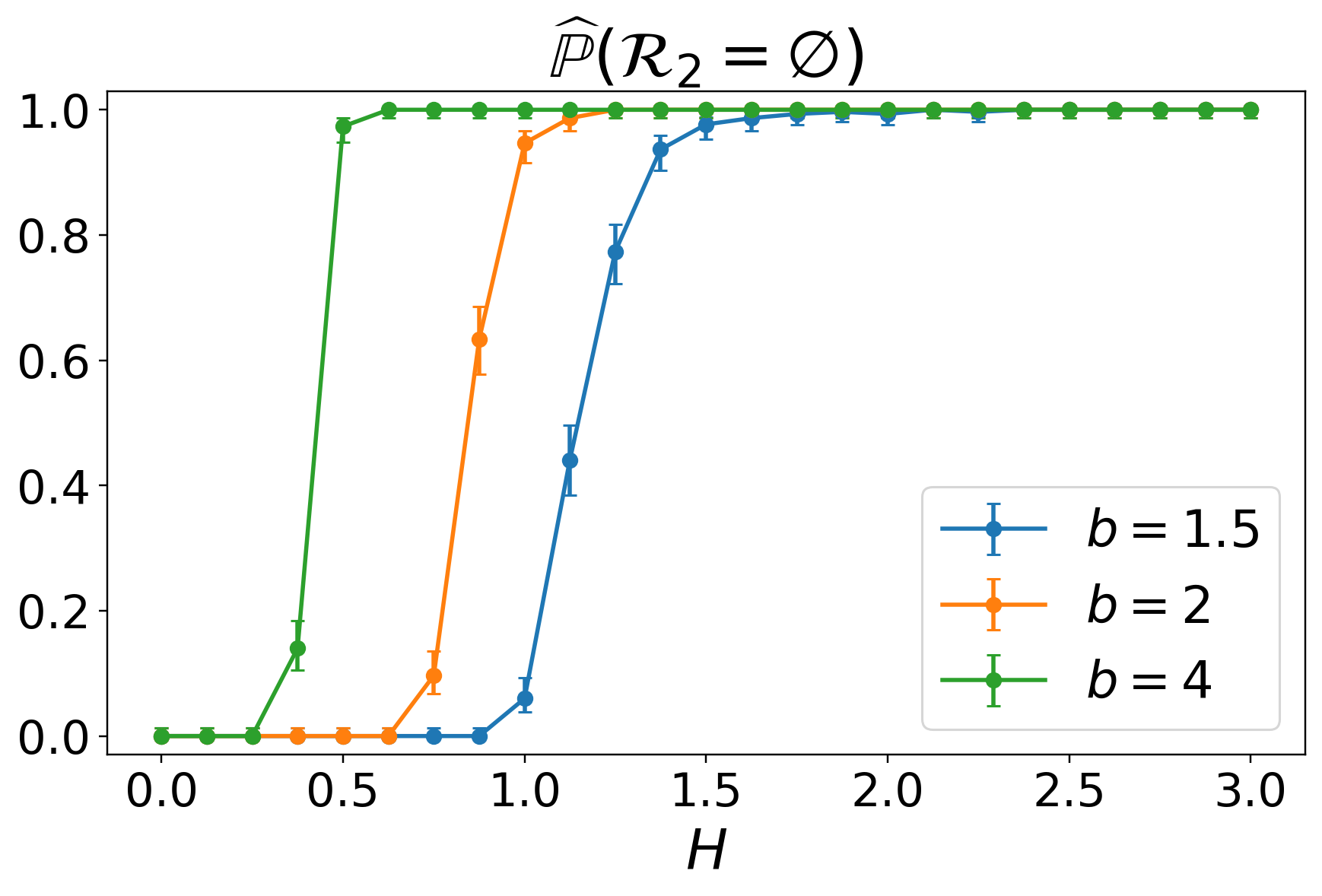}\hfill
\includegraphics[width=0.32\linewidth]{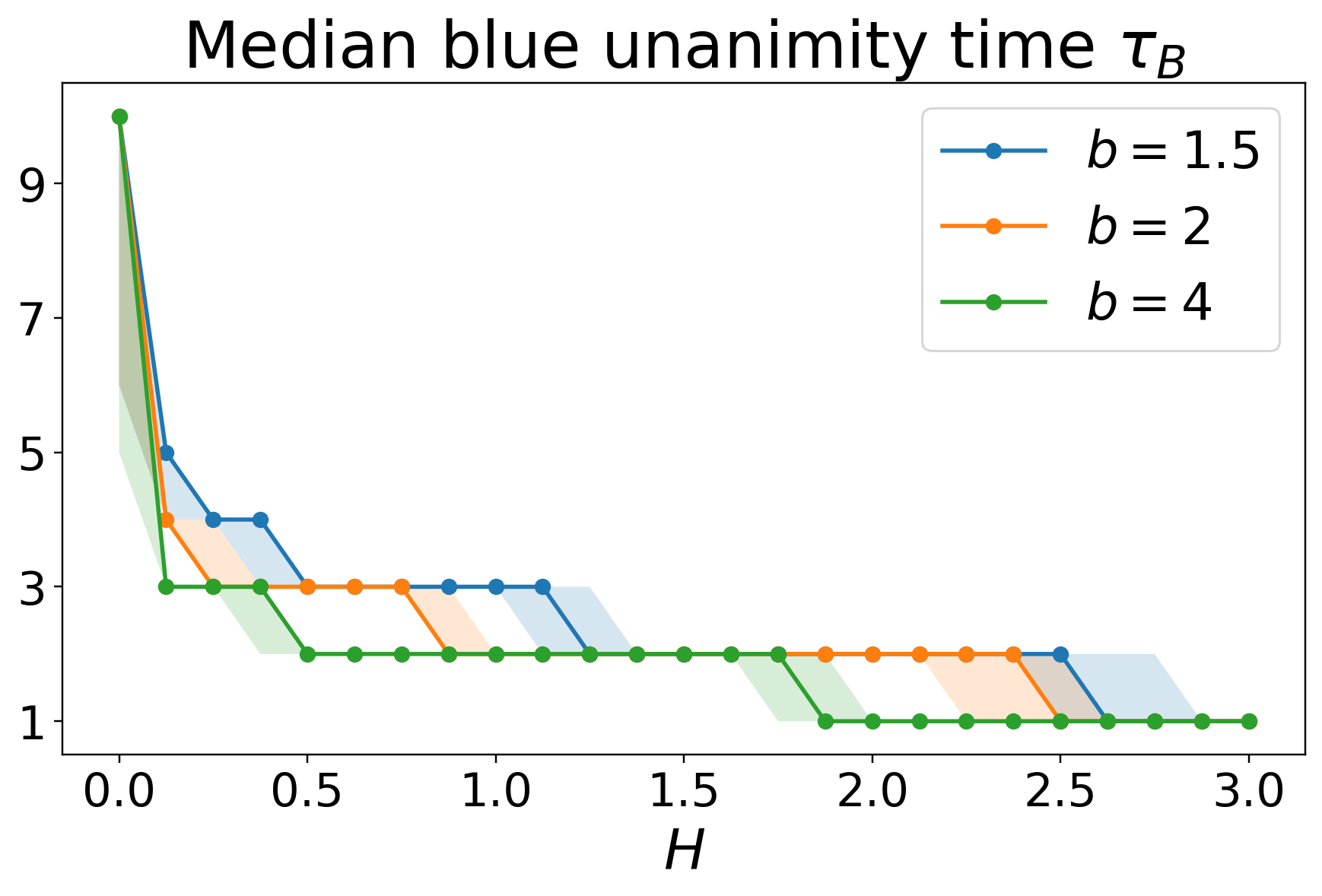}
\caption{Rapid unanimity for $N=10^4$ and $b\in\{1.5,2,4\}$,
with $300$ independent paths per parameter value.
The left and middle panels show empirical probabilities of blue
unanimity by rounds $1$ and $2$, respectively, with pointwise
$95\%$ Wilson intervals.
The right panel shows the median of $\min\{\tau_{\mathrm B},10\}$,
with the $10\%$--$90\%$ pathwise interval.}
\label{fig:er-short-time-experiment}
\end{figure}

We vary the normalized initial advantage for $b\in\{1.5,2,4\}$ by changing it.
\begin{equation}
H=\Delta_0\sqrt{\log N}/N,
\label{eqn:normalized-initial-advantage}
\end{equation}
using $300$ independent paths per parameter value. Figure~\ref{fig:er-short-time-experiment} shows the empirical probabilities
of blue unanimity by rounds $1$ and $2$, together with the capped
blue unanimity time $\min\{\tau_{\mathrm B},10\}$.
The observed transitions occur at smaller values of $H$ as $b$ increases.
The transition for blue unanimity by round $2$ also occurs at a smaller
initial advantage than the transition by round $1$.
These observations illustrate the rapid-completion behavior described
by Theorem~\ref{thm:er-two-day-unanimity}.
The final panel shows how the capped blue-unanimity time decreases
as the initial advantage grows.

\section{Future Directions}
\label{sec:future-directions}

\paragraph{Effect of spatial locality.}
The Erd\H{o}s--R\'{e}nyi (ER) model treats vertices as statistically interchangeable, whereas a \emph{Random Geometric Graph} (RGG), denoted by \(\mathbb{G}(N,p,d)\), embeds vertices as points in a \(d\)-dimensional feature space and introduces geometric dependence and local bottlenecks. For fixed \(p\in(0,1)\), Bubeck, Ding, Eldan, and R\'{a}cz~\cite{bubeck2016testing} proved that the total variation distance \(\mathrm{TV}(\mathbb{G}(N,p),\mathbb{G}(N,p,d))\) converges to $1$ when \(d\ll N^3\), whereas the same distance tends uniformly to zero when \(d\gg N^3\). For the sparse regime \(p=c/N\), the same work proved distinguishability in terms of total variation distance when \(d\ll(\log N)^3\), while Brennan, Bresler, and Nagaraj~\cite{brennan2020phase}, and subsequently Liu, Mohanty, Schramm, and Yang~\cite{liu2022testing} established equivalence also in terms of total variation distance for $d \gg N^{3/2}$ and $d\gg (\log N)^{36}$, respectively, narrowing the gap toward the conjectured threshold $d\asymp(\log N)^3$. At the spectral level, Cao and Zhu~\cite{cao2025spectra} proved that when $d\gtrsim Np\log(1/p)$ and $Np\to\infty$, the \emph{Empirical Spectral Distribution} (ESD) of the normalized adjacency matrix $\rmA/\sqrt{Np}$ under \(\mathbb{G}(N,p,d)\) converges to the semicircle law; and when $d\gg \log N$, $p = c/N$ for some fixed $c>0$, the ESD of $\rmA/\sqrt{c}$ converges to the limiting spectral distribution under \(\mathbb{G}(N,c/N)\). Preceding results indicate that RGGs in sufficiently high dimensions behave similarly to ER graphs, while those in lower dimensions preserve geometric features and can display more intricate dynamics. This leads us to investigate how spatial locality influences both the consensus threshold and the rate at which consensus is reached in the RGG model relative to the ER benchmark.

\paragraph{Other interaction mechanisms.}
Beyond resampled random graphs, it is of interest to explore how different interaction mechanisms affect winner selection and the speed of reaching unanimity. For instance, one may consider exact-majority population protocols featuring random pairwise interactions \cite{alistarh2015fast}, systems in which social and opinion variables co-evolve continuously \cite{zimper2026clustering}, or preferential-attachment networks that grow and incorporate reinforced opinion adoption \cite{podder2026model}. An open question is whether phenomena such as the critical window for winner selection, stepwise amplification, and sharp regimes for convergence time continue to manifest in these alternative settings, and to what extent these effects rely on the structure of synchronous, graph-based majority updates.

\FloatBarrier

\printbibliography

\newpage
\appendix
\phantomsection
\addcontentsline{toc}{section}{Appendices}
\section{Related Literature}\label{app:related-literature}
\subsection{Static Erd\H{o}s--R\'enyi update graphs}
Majority dynamics on an update graph sampled from \(\mathbb{G}(N,p)\) and
then held fixed is governed jointly by the density \(p\), the initial advantage
\(\Delta_0\), and the initial coloring scheme. For an i.i.d.\ initial vertex
coloring with bias \(\delta_N\), \citet{zehmakan2020opinion} proved
high-probability unanimity for \(p\ge (1+\varepsilon)\log N/N\),
with fixed \(\varepsilon>0\), whenever
\(\delta_N\gg(Np)^{-1/2}\), equivalently
\(\lvert\Delta_0\rvert\gg\sqrt{N/p}\). Under the unbiased random initial
coloring, in which every vertex is independently assigned either color with
probability \(1/2\), the initial advantage is typically of order \(\sqrt N\).
\citet{benjamini2016convergence}
obtained winning probability at least \(0.4\) within four updates for
\(p\ge C N^{-1/2}\), and \citet{fountoulakis2020resolution} subsequently strengthened this to a
high-probability result. The sufficient density was then lowered progressively
by \citet{chakraborti2023majority} to
\(p\ge C N^{-3/5}\log N\), by \citet{kim2025new} to
\(p\gg N^{-2/3}(\log N)^{2/3}\), and by \citet{jaffe2025new} to
\(p\ge C N^{-2/3}\), where \(C\) denotes a sufficiently large constant that
may differ between results.
More recently, \citet{goel2026majority} proved that, for every fixed
\(\varepsilon>0\), unbiased random initial opinions reach unanimity
with high probability within \(O(1/\varepsilon)\) rounds whenever
\(pN\ge N^\varepsilon\).
This extends the static-graph results to every fixed positive power of
\(N\) for the expected degree. It does not cover the logarithmic expected
degree \(pN=b\log N\) considered here, or the Gaussian winner-selection
law for prescribed initial camp sizes in our critical window.

For prescribed initial camp sizes and fixed \(p\), \citet{tran2023reaching} showed that a constant advantage can attain any fixed winning probability below one, while \citet{berkowitz2020central} showed that \(\Delta_0=3\) on \(\mathbb{G}(N,1/2)\) already gives a winning probability of at least \(51\%\). \citet{sah2024majority} obtained the exact asymptotic winning probability
\[
\Phi\left(\Delta_0\sqrt{\frac{2p}{\pi(1-p)}}\right)+o(1),
\]
together with three-step unanimity for \((\log N)^{-1/16}\le p\le1-(\log N)^{-1/16}\) and $0<\Delta_0\le(\log N)^{1/4}$, thereby establishing the \emph{power-of-one} phenomenon for fixed \(p\).

For prescribed camp sizes and vanishing \(p\), \citet{tran2025power} proved that an advantage satisfying \(\lvert\Delta_0\rvert\geq 20/p\) in our full-gap convention leads to unanimity within \(O(\log_{Np}N)\) updates throughout the regime \((1+c)\log N\le pN \le N-10\) for some constant \(c>0\). \citet{kim2025new} improved the sufficient scale to \(\lvert\Delta_0\rvert\gg N^{-1/2}p^{-3/2}\log N\) for \(N^{-1}(\log N)^2\ll p\ll N^{-1/2}(\log N)^{1/4}\). More recently, \citet{jaffe2025new} proved that throughout the regime \((1+c)\log N\le pN \le N-10\), the sufficient scale, where the majority dynamics reaches unanimity within \(O(\log_{Np}N)\) updates, is
\[
\Delta_0 \geq \max\left\{
p^{-1/2}\exp\left(A\sqrt{\log(1/p)}\right),
\;Bp^{-3/2}N^{-1/2}
\right\}
\]
for some suitable constants \(A,B>0\). These results establish a sparse \emph{power-of-few} phenomenon, but the conjectured scale \(\lvert\Delta_0\rvert\asymp p^{-1/2}\) \citep{tran2025power} remains open. At \(p=b\log N/N\), the sufficient static bound in that work has order
\(N/(\log N)^{3/2}\) {\citep[Theorem 4.1]{jaffe2025new}}, rather than the conjectured \(p^{-1/2}\asymp\sqrt{N/\log N}\).

\begin{conjecture}[Optimal power of few, {\citealp[Conjecture 1.11]{tran2025power}}]
\label{conj:optimal-power-of-few}
There is a function $f$ satisfying $f(x)>1/2$ for $x>0$ and
$\lim_{x\to\infty}f(x)=1$ with the following property: majority dynamics on
$\mathbb{G}(N,p)$ with initial advantage $\Delta_0$ for blue and
$p\geq(1+c)\log(N)/N$ for some constant $c>0$ ends unanimously blue with
probability $f(\Delta_0\sqrt p)-o(1)$.
\end{conjecture}

For the dense regime where $p \geq (\log N)^{-1/16}$, a corollary of \citet{sah2024majority} verifies Conjecture~\ref{conj:optimal-power-of-few}. However, for smaller values of $p$, the conjecture remains unresolved.

Our Gaussian selection law concerns independently resampled graphs.
It establishes the corresponding critical scale for this model, while the
sparse static-graph conjecture remains a distinct problem because the
opinions become correlated with the graph reused in subsequent updates.

\subsection{Resampled stochastic block models}
An alternative choice for modeling the interactions between agents is the \emph{Stochastic Block Model} (SBM), where vertices within the same community are connected with probability $\alpha$, and vertices in different communities are connected with probability $\beta$. For fixed $\alpha, \beta \in (0, 1]$, \citet{wang2022consensus} studied a Markovian model that resamples every edge and a non-Markovian model that resamples an edge only when one of its endpoints changes opinion. They proved a \emph{power-of-one} winner-selection result for the Markovian model; for the non-Markovian model, they located, up to the second-order term, the initial-lead threshold between halting and unanimity within three updates and provided sufficient thresholds for unanimity in one or two updates.

In the sparse regime where $\alpha = a\log N/N$ and $\beta = b\log N/N$ for some constants $a, b > 0$, \citet{dumitriu2026majority} analyzed majority dynamics under the assortative assumption ($a > b$). In that context, the progression toward unanimity is determined by the weighted advantage \(\widetilde\Delta_t \coloneqq b|\gB_t| - a|\gR_t|\), and new phenomena in the pace toward blue unanimity are uncovered, exhibiting constant-time, subpolynomial-time, and polynomial-time scales.

The homogeneous ER model considered here assigns the same edge
probability to every pair of vertices. The sparse assortative analysis
assumes $a>b$ and does not directly cover the case $a=b$. Our local
comparison near balance resolves the smaller $p^{-1/2}$ initial-advantage
scale and supports uniform Gaussian winner selection and two-sided
consensus-time bounds. We reuse general one-step probability and
concentration tools where applicable, with their sources identified in
the preceding appendices.

\subsection{Other models for opinion dynamics}
In the population-protocol model, \citet{alistarh2015fast} considered finite-state agents that interact in uniformly random pairs and seek exact consensus on the initial majority. Their \emph{Average-and-Conquer} protocol provides a tunable state--time tradeoff and achieves polylogarithmic expected parallel time when the per-agent state space is sufficiently large relative to the initial bias.

\citet{podder2026model} studied a growing directed preferential-attachment network in which each arriving agent samples past agents, adopts one of two immutable opinions through a reinforced stochastic rule, and links to sampled agents sharing that opinion. For fixed and suitably growing sample sizes, they characterized the almost-sure asymptotics of opinion shares, influence capital, and network activity; for fixed sample size, they also obtained second-order fluctuations.

\citet{zimper2026clustering} considered agent-based systems with continuous social positions and scalar opinions, allowing either one-way influence from social positions to opinions or two-way feedback. They derived reduced SPDEs for the resulting cluster dynamics and showed numerically that these reproduce the underlying systems at substantially lower computational cost, including in an application to General Social Survey data.

\subsection{Spatial interaction graphs}
The Erd\H{o}s--R\'{e}nyi (ER) model treats vertices as statistically interchangeable, whereas a \emph{Random Geometric Graph} (RGG), denoted by \(\mathbb{G}(N,p,d)\), embeds vertices as points in a \(d\)-dimensional feature space and introduces geometric dependence and local bottlenecks. For fixed \(p\in(0,1)\), \citet{bubeck2016testing} proved that the total variation distance \(\mathrm{TV}(\mathbb{G}(N,p),\mathbb{G}(N,p,d))\) converges to $1$ when \(d\ll N^3\), whereas the same distance tends uniformly to zero when \(d\gg N^3\). For the sparse regime \(p=c/N\), the same work proved distinguishability in terms of total variation distance when \(d\ll(\log N)^3\), while \citet{brennan2020phase}, and subsequently \citet{liu2022testing} established equivalence also in terms of total variation distance for $d \gg N^{3/2}$ and $d\gg (\log N)^{36}$, respectively, narrowing the gap toward the conjectured threshold $d\asymp(\log N)^3$. At the spectral level, \citet{cao2025spectra} proved that when $d\gtrsim Np\log(1/p)$ and $Np\to\infty$, the \emph{Empirical Spectral Distribution} (ESD) of the normalized adjacency matrix $\rmA/\sqrt{Np}$ under \(\mathbb{G}(N,p,d)\) converges to the semicircle law; and when $d\gg \log N$, $p = c/N$ for some fixed $c>0$, the ESD of $\rmA/\sqrt{c}$ converges to the limiting spectral distribution under \(\mathbb{G}(N,c/N)\).

The preceding results indicate that RGGs in sufficiently high dimensions behave similarly to ER graphs, while those in lower dimensions preserve geometric features and can display more intricate dynamics. This leads us to investigate how spatial locality influences both the consensus threshold and the rate at which consensus is reached in the RGG model relative to the ER benchmark.

\section{Proof of Gaussian Winner Selection}\label{app:gaussian-proof}
\label{sec:er-critical-window-winner-selection}

We present the proof of Theorem~\ref{thm:er-critical-window-winner-selection} in this section. The proofs of Lemmas~\ref{lem:critical-window-mean-advantage}, \ref{lem:critical-window-one-step-clt}, and \ref{lem:critical-window-one-step-anticoncentration} are given in Section~\ref{sec:deferred-proofs-critical-window-winner-selection} below. For this winner-selection analysis, we condition on a deterministic initial configuration \(\rvy_0\) satisfying \eqref{eq:critical-window-condition} and work under Assumption~\ref{ass:sparse-er}. The final subsection instead considers the random initial votes of Corollary~\ref{cor:er-random-vote-aggregation}.

The main idea is that the first update selects the sign of the eventual winner according to a Gaussian crossover and that the selected color subsequently reaches unanimity through the amplification mechanism of Theorem~\ref{thm:er-polylogarithmic-days-to-unanimity}. We first provide the three one-step estimates needed for the proof below.

\subsection{One-step estimates in the critical window}

We begin with the conditional mean of the first-update advantage in the critical window \eqref{eq:critical-window-condition}. 

\begin{lemma}[Critical-window mean advantage]
\label{lem:critical-window-mean-advantage}
For every fixed \(C_0<\infty\), uniformly over deterministic initial
configurations satisfying \eqref{eq:critical-window-condition},
we have
\begin{align}
\frac{\E[\Delta_1\mid\rvy_0]}{\sqrt N}
=
x_0
+O\left(\frac1{\sqrt{\log N}}\right).
\label{eq:critical-window-mean-asymptotic}
\end{align}
The implicit constant depends only on \(b\) and \(C_0\).
\end{lemma}

To upgrade the one-step estimate to a uniform approximation, we follow the approach in \citet{berkowitz2020central} to introduce the configuration-dependent factor and the two uniform approximation errors that appear in Theorem~\ref{thm:er-critical-window-winner-selection}. For a deterministic initial configuration \(\rvy_0\), let \(\rW_0^{\gR}\) and \(\rW_0^{\gB}\) be independent auxiliary random variables with the following
distributions:
\[
\rW_0^{\gR}\sim\Bin(|\gR_0|,p),
\qquad\text{and}\qquad
\rW_0^{\gB}\sim\Bin(|\gB_0|,p).
\]
The corresponding normalization factor, introduced in \citet{berkowitz2020central}, is defined as
\begin{align}
\chi_N(\rvy_0)
&\coloneqq
\P(\rW_0^{\gR}\ge\rW_0^{\gB}) \cdot \P(\rW_0^{\gR}\le\rW_0^{\gB}).
\label{eq:critical-window-chi-definition}
\end{align}
For \(C_0<\infty\), the centered approximation error, uniform over the critical window \(|x_0|\le C_0\), is
\begin{align}
\eta_N^{\mathrm{BD}}(C_0)
\coloneqq
\sup_{\substack{\rvy_0\ \mathrm{deterministic}:\\
|x_0|\le C_0}}
\sup_{z\in\R}
\left|
\P\left(
\frac{\Delta_1-\E[\Delta_1\mid\rvy_0]}
{2\sqrt{N\chi_N(\rvy_0)}}
\le z
\,\middle|\,
\rvy_0
\right)
-\Phi(z)
\right|.
\label{eq:critical-window-berkowitz-devlin-error}
\end{align}
The corresponding uniform error after restoring the original centering and
normalization is
\begin{align}
\rho_N(C_0)
\coloneqq
\sup_{\substack{\rvy_0\ \mathrm{deterministic}:\\
|x_0|\le C_0}}
\sup_{z\in\R}
\left|
\P\left(
\frac{\Delta_1}{\sqrt N}\le z
\,\middle|\,
\rvy_0
\right)
-
\Phi\left(
z-x_0
\right)
\right|.
\label{eq:critical-window-one-step-clt}
\end{align}

\begin{lemma}[Uniform critical-window one-step CLT]
\label{lem:critical-window-one-step-clt}
For every fixed \(C_0<\infty\),
\(\eta_N^{\mathrm{BD}}(C_0)=o(1)\). Moreover, there exists
\(C=C(b,C_0)>0\) such that, for all sufficiently large \(N\),
\begin{align}
\rho_N(C_0)
\le
\eta_N^{\mathrm{BD}}(C_0)
+\frac{C}{\sqrt{\log N}}
=o(1).
\label{eq:critical-window-one-step-clt-error-bound}
\end{align}
\end{lemma}

We then provide the uniform anti-concentration estimate for \(\Delta_1\).

\begin{lemma}[Uniform critical-window anti-concentration]
\label{lem:critical-window-one-step-anticoncentration}
Fix \(C_0<\infty\). For every deterministic positive sequence
\(\ell_N\to0\), uniformly over deterministic initial configurations
satisfying \eqref{eq:critical-window-condition},
\begin{align}
\sup_{u\in\R}
\P\left(
u\sqrt N\le\Delta_1\le(u+\ell_N)\sqrt N
\,\middle|\,
\rvy_0
\right)
\le
\frac{\ell_N}{\sqrt{2\pi}}+2\rho_N(C_0)
=o(1).
\label{eq:critical-window-one-step-anticoncentration}
\end{align}
\end{lemma}

\subsection{Proofs of the one-step estimates}
\label{sec:deferred-proofs-critical-window-winner-selection}

\subsubsection{Proof of Lemma~\ref{lem:critical-window-mean-advantage}}
\begin{proof}[Proof of Lemma~\ref{lem:critical-window-mean-advantage}]
    By symmetry under interchanging blue and red, it suffices to consider
    \(\Delta_0\ge0\). We denote
    \[
    s_N\coloneqq\sqrt{(N-1)p(1-p)},
    \qquad
    z_N\coloneqq\frac{p\Delta_0}{s_N},
    \qquad
    \delta_N\coloneqq\frac{p-1/2}{s_N}.
    \]
    Condition~\eqref{eq:critical-window-condition} gives
    \begin{align}
    0\le\Delta_0
    &\le
    C_0\sqrt{\frac{\pi N}{2b\log N}},
    &
    z_N&=O(N^{-1/2}),
    &
    \delta_N&=O((\log N)^{-1/2}).
    \label{eq:critical-window-basic-scales}
    \end{align}
    Thus both color classes have size at least \(N/3\) for all sufficiently
    large \(N\), and the standardized means remain uniformly bounded. Hence
    Lemma~\ref{lem:cc-gaussian-flip-probabilities} applies uniformly.
    
    Its critical-window estimate
    \eqref{eq:cc-gaussian-critical-flip-difference} gives
    \begin{align}
    p_0^{\gR}-p_0^{\gB}
    =
    \sqrt{\frac{2}{\pi}}\frac{p\Delta_0}{s_N}
    +O\left(\frac{p\Delta_0}{\log N}\right).
    \label{eq:critical-window-flip-difference}
    \end{align}
    Moreover, Lemma~\ref{lem:expected-variance-degree-differences} gives the
    exact identities
    \[
    x_0^{\gR}-\frac1{2\sqrt{v_0^{\gR}}}=z_N+\delta_N,
    \qquad
    x_0^{\gB}-\frac1{2\sqrt{v_0^{\gB}}}=-z_N+\delta_N.
    \]
    Therefore, \eqref{eq:cc-gaussian-pR}--\eqref{eq:cc-gaussian-pB}, the identity
    \(\Phi(z)+\Phi(-z)=1\), and the Lipschitz continuity of \(\Phi\) yield
    \begin{align}
    p_0^{\gR}+p_0^{\gB}
    &=
    \Phi(z_N+\delta_N)+\Phi(-z_N+\delta_N)
    +O((\log N)^{-1})
    =
    1+O((\log N)^{-1/2}).
    \label{eq:critical-window-flip-sum}
    \end{align}
    Substituting \eqref{eq:critical-window-flip-difference} and
    \eqref{eq:critical-window-flip-sum} into
    \eqref{eq:advantage-expectation-expanded} and dividing by \(\sqrt N\) gives
    \begin{align}
    \frac{\E[\Delta_1\mid\rvy_0]}{\sqrt N}
    &=
    \sqrt{\frac{2}{\pi}}\frac{\sqrt N\,p\Delta_0}{s_N}
    +O\left(
    \frac{\sqrt N\,p\Delta_0}{\log N}
    +\frac{\Delta_0}{\sqrt{N\log N}}
    \right).
    \label{eq:critical-window-mean-decomposition}
    \end{align}
    Since \(p=b\log N/N\), the leading term satisfies
    \[
    \sqrt{\frac{2}{\pi}}\frac{\sqrt N\,p\Delta_0}{s_N}
    =
    x_0\sqrt{\frac{N}{N-1}}\frac1{\sqrt{1-p}}
    =
    x_0+O\left(\frac{\log N}{N}\right),
    \]
    whereas the two remainders in
    \eqref{eq:critical-window-mean-decomposition} are respectively
    \(O((\log N)^{-1/2})\) and \(O((\log N)^{-1})\). This proves
    \eqref{eq:critical-window-mean-asymptotic} for \(\Delta_0\ge0\).
    
    If \(\Delta_0<0\), interchange the two colors. This changes both \(x_0\)
    and \(\Delta_1\) to their negatives, so the same conclusion follows.
    \end{proof}
    
\subsubsection{Proof of Lemma~\ref{lem:critical-window-one-step-clt}}
\begin{proof}[Proof of Lemma~\ref{lem:critical-window-one-step-clt}]
    Fix an admissible initial configuration, abbreviate
    \(\chi_N=\chi_N(\rvy_0)\), and set
    \[
    \sigma_N\coloneqq\sqrt{Np(1-p)},
    \qquad
    D_N\coloneqq \rW_0^{\gR}-\rW_0^{\gB},
    \qquad
    F_N(k)\coloneqq\P(D_N\le k).
    \]
    Condition~\eqref{eq:critical-window-condition} gives
    \begin{align}
    \frac{p|\Delta_0|}{\sigma_N}=O(N^{-1/2}),
    \qquad
    \sigma_N\sim\sqrt{b\log N}\longrightarrow\infty.
    \label{eq:critical-window-bd-scales}
    \end{align}
    The variable \(D_N\) has mean \(-p\Delta_0\), variance \(\sigma_N^2\),
    and total sparse-binomial density parameter \(b\). Hence
    Lemma~\ref{lem:gaussian-sparse-binomial-difference}, evaluated at
    \(k\in\{-1,0\}\), gives
    \[
    F_N(k)
    =
    \Phi\left(\frac{k+p\Delta_0}{\sigma_N}\right)
    +O((\log N)^{-1/2})
    =
    \frac12+O((\log N)^{-1/2}),
    \]
    uniformly over the initial configurations under consideration. Therefore,
    \[
    \P(\rW_0^{\gR}\le \rW_0^{\gB})=F_N(0)
    =\frac12+O((\log N)^{-1/2}),
    \]
    and
    \[
    \P(\rW_0^{\gR}\ge \rW_0^{\gB})=1-F_N(-1)
    =\frac12+O((\log N)^{-1/2}).
    \]
    Multiplying these estimates yields
    \begin{align}
    \chi_N
    =
    \frac14+O((\log N)^{-1/2}).
    \label{eq:berkowitz-devlin-chi-limit}
    \end{align}
    In particular,
    \(\log(1/\chi_N)=O(1)=o(\log\sigma_N)\), so the hypotheses of
    Theorem~\ref{thm:berkowitz-devlin-one-step-clt} hold.
    
    Consider any sequence of admissible deterministic initial configurations.
    Conditional on each \(\rvy_0\), the next update uses a graph sampled
    independently from \(\mathbb{G}(N,p)\), so
    Theorem~\ref{thm:berkowitz-devlin-one-step-clt}
    applies with \(W_R=\rW_0^{\gR}\), \(W_B=\rW_0^{\gB}\), and normalization
    factor \(\chi_N=\chi_N(\rvy_0)\). The hypotheses verified above therefore
    give
    \[
    \frac{|\gR_1|-\E[|\gR_1|\mid\rvy_0]}
    {\sqrt{N\chi_N}}
    \xrightarrow{\mathrm d}\mathcal N(0,1).
    \]
    On the other hand, \(\Delta_1=N-2|\gR_1|\), and hence
    \[
    \Delta_1-\E[\Delta_1\mid\rvy_0]
    =
    -2\left(
    |\gR_1|-\E[|\gR_1|\mid\rvy_0]
    \right).
    \]
    Thus \(Y_N\) below is the negative of the standardized red-camp size in the
    preceding display. Since the standard normal distribution is symmetric,
    \[
    Y_N
    \coloneqq
    \frac{\Delta_1-\E[\Delta_1\mid\rvy_0]}
    {2\sqrt{N\chi_N}}
    \xrightarrow{\mathrm d}\mathcal N(0,1).
    \]
    Because \(\Phi\) is continuous, this convergence is also convergence in
    Kolmogorov distance. It is uniform over the admissible configurations:
    otherwise, for some \(\varepsilon>0\), one could select a sequence of
    admissible configurations whose Kolmogorov distances are at least
    \(\varepsilon\), contradicting the preceding sequential conclusion. Thus
    \[
    \eta_N^{\mathrm{BD}}(C_0)=o(1).
    \]
    
    It remains to restore the original centering and normalization. Define
    \[
    m_N\coloneqq\frac{\E[\Delta_1\mid\rvy_0]}{\sqrt N},
    \qquad
    \alpha_N\coloneqq2\sqrt{\chi_N}.
    \]
    Equations~\eqref{eq:berkowitz-devlin-chi-limit} and
    \eqref{eq:critical-window-mean-asymptotic} give, uniformly,
    \[
    \alpha_N=1+O((\log N)^{-1/2}),
    \qquad
    m_N=x_0+O((\log N)^{-1/2}).
    \]
    For every \(z\in\R\),
    \[
    \P\left(\frac{\Delta_1}{\sqrt N}\le z\,\middle|\,\rvy_0\right)
    =
    \P\left(Y_N\le\frac{z-m_N}{\alpha_N}\,\middle|\,\rvy_0\right).
    \]
    For \(\alpha\) in a fixed neighborhood of \(1\),
    \[
    \sup_{w\in\R}|\Phi(w/\alpha)-\Phi(w)|
    \le C|\alpha-1|,
    \]
    while the Lipschitz continuity of \(\Phi\) controls translations. Hence
    \begin{align*}
    \sup_{z\in\R}
    \left|
    \P\left(\frac{\Delta_1}{\sqrt N}\le z\,\middle|\,\rvy_0\right)
    -\Phi(z-x_0)
    \right|
    \le &\,
    \eta_N^{\mathrm{BD}}(C_0)
    +C|\alpha_N-1|+C|m_N-x_0|\\
    \le &\,
    \eta_N^{\mathrm{BD}}(C_0)
    +\frac{C}{\sqrt{\log N}}.
    \end{align*}
    Taking the supremum over the admissible initial configurations and using
    the definition \eqref{eq:critical-window-one-step-clt} proves
    \eqref{eq:critical-window-one-step-clt-error-bound}.
    \end{proof}

\subsubsection{Proof of Lemma~\ref{lem:critical-window-one-step-anticoncentration}}
    \begin{proof}[Proof of Lemma~\ref{lem:critical-window-one-step-anticoncentration}]
        For a fixed admissible initial configuration, define
        \[
        F_{\rvy_0}(z)
        \coloneqq
        \P\left(
        \frac{\Delta_1}{\sqrt N}\le z
        \,\middle|\,
        \rvy_0
        \right),
        \qquad
        G_{\rvy_0}(z)
        \coloneqq
        \Phi(z-x_0).
        \]
        By the definition of \(\rho_N(C_0)\) in
        \eqref{eq:critical-window-one-step-clt},
        \[
        \sup_{z\in\R}
        |F_{\rvy_0}(z)-G_{\rvy_0}(z)|
        \le\rho_N(C_0).
        \]
        Since \(G_{\rvy_0}\) is continuous, taking the left limit at \(u\) gives
        \[
        F_{\rvy_0}(u^-)
        \ge
        G_{\rvy_0}(u)-\rho_N(C_0).
        \]
        Therefore, for every \(u\in\R\),
        \begin{align*}
        \P\left(
        u\sqrt N\le\Delta_1\le(u+\ell_N)\sqrt N
        \,\middle|\,
        \rvy_0
        \right)
        =&\,
        F_{\rvy_0}(u+\ell_N)-F_{\rvy_0}(u^-)\\
        \le &\,
        G_{\rvy_0}(u+\ell_N)-G_{\rvy_0}(u)
        +2\rho_N(C_0)\\
        \le &\,
        \frac{\ell_N}{\sqrt{2\pi}}+2\rho_N(C_0),
        \end{align*}
        where the last inequality uses
        \(\sup_x\varphi(x)=1/\sqrt{2\pi}\). The bound is uniform in \(u\) and in
        the admissible initial configuration. Taking the supremum over \(u\), and
        then using \(\ell_N\to0\) and \(\rho_N(C_0)=o(1)\), proves
        \eqref{eq:critical-window-one-step-anticoncentration}.
        \end{proof}

\subsection{Proof of Theorem~\ref{thm:er-critical-window-winner-selection}}

The proof separates winner selection on the first update from the subsequent amplification to unanimity. Lemma~\ref{lem:critical-window-mean-advantage} and Lemma~\ref{lem:critical-window-one-step-clt} show uniformly that \(\Delta_1/\sqrt N\) is approximated by a Gaussian random variable with mean \(x_0\) and variance one. In particular, along any sequence with
\(x_0\to x\in\R\),
\[
\frac{\Delta_1}{\sqrt N}
\xrightarrow{\mathrm d}
x+Z,
\qquad
Z\sim\mathcal N(0,1),
\]
so the first update produces a positive or negative advantage with limiting probabilities \(\Phi(x)\) and \(\Phi(-x)\), respectively. Lemma~\ref{lem:critical-window-one-step-anticoncentration} shows that the probability of the event \(|\Delta_1|<\sqrt N/h_N\) is vanishing. Consequently, with probability \(1-o(1)\), the first update either has already reached unanimity or produces a signed advantage of magnitude at least \(\sqrt N/h_N\). Conditional on the first update, the future graphs are independent of the resulting configuration; thus the time-shifted form of Theorem~\ref{thm:er-polylogarithmic-days-to-unanimity} drives a positive advantage to blue unanimity and, after exchanging the colors, a negative advantage to red unanimity. The proof below makes this argument uniform and tracks its total error precisely.

\begin{proof}[Proof of Theorem~\ref{thm:er-critical-window-winner-selection}]
We first verify that the total error vanishes. Since
\(h_N=o(\sqrt{\log N})\), we have
\(\log N/h_N^2\to\infty\). Hence
\eqref{eqn:er-polylog-joint-failure-probability} gives
\begin{align}
\eta_N\left(\frac{\sqrt N}{h_N}\right)
&=
8\exp\left(-c\frac{\log N}{h_N^2}\right)
+2\exp\left(-(\log N)^2\right)
+2N^{-\xi}
=o(1).
\label{eq:critical-window-polylog-error-order}
\end{align}
Together with \(h_N\to\infty\),
\eqref{eq:critical-window-one-step-clt-error-bound}, and
\eqref{eq:critical-window-total-error}, this gives
\begin{align}
\mathfrak r_N(C_0)
&=
\rho_N(C_0)
+\frac1{h_N}
+\eta_N\left(\frac{\sqrt N}{h_N}\right)
=o(1).
\label{eq:critical-window-total-error-order}
\end{align}

Recall that \((\mathcal F_t)_{t\ge0}\) in \eqref{eqn:filtration-def} denotes the natural filtration generated by the initial configuration and the resampled graphs through time \(t\). Define the following three \(\mathcal F_1\)-measurable events, which partition the sample space:
\[
\mathcal E_N^+ \coloneqq \left\{\Delta_1\ge\frac{\sqrt N}{h_N}\right\}, \quad \mathcal E_N^- \coloneqq \left\{\Delta_1\le-\frac{\sqrt N}{h_N}\right\}, \quad \mathcal E_N^0 \coloneqq \left\{|\Delta_1|<\frac{\sqrt N}{h_N}\right\}.
\]
Lemma~\ref{lem:critical-window-one-step-anticoncentration} with
\(u=-h_N^{-1}\) and \(\ell_N=2/h_N\) gives, uniformly under
\eqref{eq:critical-window-condition},
\begin{align}
\P(\mathcal E_N^0\mid\rvy_0)
\le
\frac{2}{\sqrt{2\pi}\,h_N}
+2\rho_N(C_0).
\label{eq:critical-window-small-day-one-advantage}
\end{align}

The definition of \(\rho_N(C_0)\) controls the distribution function of
\(\Delta_1/\sqrt N\) and, because the Gaussian comparison distribution is
continuous, also its left limits. Evaluating at the thresholds
\(\pm h_N^{-1}\) therefore gives
\begin{subequations}
\begin{align}
\P(\mathcal E_N^+\mid\rvy_0)
&=
1-\Phi(h_N^{-1}-x_0)+O(\rho_N(C_0))
=
\Phi(x_0)+O\left(h_N^{-1}+\rho_N(C_0)\right),
\label{eq:critical-window-positive-day-one-advantage}\\
\P(\mathcal E_N^-\mid\rvy_0)
&=
\Phi(-h_N^{-1}-x_0)+O(\rho_N(C_0))
=
\Phi(-x_0)+O\left(h_N^{-1}+\rho_N(C_0)\right).
\label{eq:critical-window-negative-day-one-advantage}
\end{align}
\end{subequations}
Here we used \(1-\Phi(z)=\Phi(-z)\) and the Lipschitz continuity of
\(\Phi\). All three estimates are uniform over the configurations under
consideration.

Given \(\mathcal F_1\), the configuration \(\rvy_1\) is fixed and the future
interaction graphs are independent of it. Moreover, for all sufficiently
large \(N\), we have \(h_N\ge1\) and
\(c_1\sqrt{\log N}>1\). Hence the upper horizon in
\eqref{eqn:er-polylog-sufficient-time} is nonincreasing in the initial
advantage.

Blue unanimity either already holds at time \(1\), or
Theorem~\ref{thm:er-polylogarithmic-days-to-unanimity} applies from time
\(1\). On \(\mathcal E_N^+\), we have
\(\Delta_1\ge\sqrt N/h_N\), so its upper horizon is at most \(T_N\) and
the failure probability is at most \(\eta_N(\sqrt N/h_N)\). Since
\(\mathcal E_N^+\in\mathcal F_1\), this gives
\begin{align*}
\P\left(
\gR_{T_N+3}\ne\emptyset,\mathcal E_N^+
\,\middle|\,
\mathcal F_1
\right)
&=
\indi{\mathcal E_N^+}
\P\left(
\gR_{T_N+3}\ne\emptyset
\,\middle|\,
\mathcal F_1
\right)
\le
\eta_N\left(\frac{\sqrt N}{h_N}\right)
\indi{\mathcal E_N^+}.
\end{align*}
The same argument with the colors exchanged gives
\begin{align*}
\P\left(
\gB_{T_N+3}\ne\emptyset,\mathcal E_N^-
\,\middle|\,
\mathcal F_1
\right)
&=
\indi{\mathcal E_N^-}
\P\left(
\gB_{T_N+3}\ne\emptyset
\,\middle|\,
\mathcal F_1
\right)
\le
\eta_N\left(\frac{\sqrt N}{h_N}\right)
\indi{\mathcal E_N^-}.
\end{align*}
Here we used that \(\eta_N(s)\) is nonincreasing for \(s\ge0\). Taking
conditional expectations yields
\begin{align*}
\P\left(
\gR_{T_N+3}\ne\emptyset,\mathcal E_N^+
\,\middle|\,
\rvy_0
\right)
&\le
\eta_N\left(\frac{\sqrt N}{h_N}\right)
\P(\mathcal E_N^+\mid\rvy_0)
\le
\eta_N\left(\frac{\sqrt N}{h_N}\right),\\
\P\left(
\gB_{T_N+3}\ne\emptyset,\mathcal E_N^-
\,\middle|\,
\rvy_0
\right)
&\le
\eta_N\left(\frac{\sqrt N}{h_N}\right)
\P(\mathcal E_N^-\mid\rvy_0)
\le
\eta_N\left(\frac{\sqrt N}{h_N}\right).
\end{align*}

The all-blue and all-red events are disjoint. Therefore, the second bound
also implies
\[
\P\left(
\gR_{T_N+3}=\emptyset,\mathcal E_N^-
\,\middle|\,
\rvy_0
\right)
\le
\eta_N\left(\frac{\sqrt N}{h_N}\right).
\]
Using the partition
\(\mathcal E_N^+,\mathcal E_N^0,\mathcal E_N^-\), the lower bound follows
from
\begin{align*}
\P\left(
\gR_{T_N+3}=\emptyset
\,\middle|\,
\rvy_0
\right)
&\ge
\P\left(
\gR_{T_N+3}=\emptyset,\mathcal E_N^+
\,\middle|\,
\rvy_0
\right)\\
&=
\P(\mathcal E_N^+\mid\rvy_0)
-\P\left(
\gR_{T_N+3}\ne\emptyset,\mathcal E_N^+
\,\middle|\,
\rvy_0
\right)\\
&\ge
\P(\mathcal E_N^+\mid\rvy_0)
-\eta_N\left(\frac{\sqrt N}{h_N}\right).
\end{align*}
For the upper bound, the same partition gives
\begin{align*}
&\, \P\left(
\gR_{T_N+3}=\emptyset
\,\middle|\,
\rvy_0
\right)\\
=&\,
\P\left(
\gR_{T_N+3}=\emptyset,\mathcal E_N^+
\,\middle|\,
\rvy_0
\right)
+
\P\left(
\gR_{T_N+3}=\emptyset,\mathcal E_N^0
\,\middle|\,
\rvy_0
\right)
+
\P\left(
\gR_{T_N+3}=\emptyset,\mathcal E_N^-
\,\middle|\,
\rvy_0
\right)\\
\le &\,
\P(\mathcal E_N^+\mid\rvy_0)
+\P(\mathcal E_N^0\mid\rvy_0)
+\eta_N\left(\frac{\sqrt N}{h_N}\right).
\end{align*}
Combining these two bounds with
\eqref{eq:critical-window-small-day-one-advantage} and
\eqref{eq:critical-window-positive-day-one-advantage} gives
\[
\P\left(
\gR_{T_N+3}=\emptyset
\,\middle|\,
\rvy_0
\right)
=
\Phi(x_0)+O\left(\mathfrak r_N(C_0)\right).
\]
By exchanging the colors and using
\eqref{eq:critical-window-negative-day-one-advantage}, we similarly obtain
\[
\P\left(
\gB_{T_N+3}=\emptyset
\,\middle|\,
\rvy_0
\right)
=
\Phi(-x_0)+O\left(\mathfrak r_N(C_0)\right).
\]
These are \eqref{eq:critical-window-blue-winning-probability} and
\eqref{eq:critical-window-red-winning-probability}.

Finally, the assumptions on \(h_N\) imply
\(1\le h_N\le\sqrt{\log N}\) for all sufficiently large \(N\), and hence
\(\log h_N=O(\log\log N)\). Therefore,
\begin{align*}
\log\left(Kh_N\sqrt{N/\log N}\right)
&=
\frac12\log N-\frac12\log\log N+\log h_N+O(1)\\
&=
\frac12\log N+O(\log\log N),\\
\log\left(c_1\sqrt{\log N}\right)
&=
\frac12\log\log N+O(1).
\end{align*}
The quotient in \eqref{eq:critical-window-amplification-horizon} is positive
for all sufficiently large \(N\), while the maximum and ceiling affect it by
at most \(O(1)\). Therefore,
\(T_{N}+3=(1+o(1))\log N/\log\log N\), completing the proof.
\end{proof}

\subsection{Proof of Corollary~\ref{cor:er-critical-window-consensus-time}}
\begin{proof}
Fix \(C_0>0\) and \(K>K_2^{\mathrm{ER}}(b)\), set
\(h_N=(\log N)^{1/4}\), and define
\[
\ell_N=\frac{\sqrt N}{h_N},\qquad
u_N=h_N\sqrt N,\qquad
\mathcal H_N=\{\ell_N\le |\Delta_1|\le u_N\}.
\]
The uniform CLT and anti-concentration estimate in
Lemmas~\ref{lem:critical-window-one-step-clt}
and~\ref{lem:critical-window-one-step-anticoncentration} give
\[
\P(\mathcal H_N^c\mid\rvy_0)
\le \frac{2}{\sqrt{2\pi}\,h_N}
+4\rho_N(C_0)+2\Phi(C_0-h_N)=o(1),
\]
uniformly over \(|x_0|\le C_0\). Here the last term bounds the two
Gaussian tails outside \([-h_N,h_N]\).

Write \(\underline T(s),\overline T(s)\) for the deterministic horizons
in \eqref{eqn:er-polylog-necessary-time}--\eqref{eqn:er-polylog-sufficient-time}
with initial advantage \(s\). On \(\mathcal H_N\), select blue if
\(\Delta_1>0\) and red if \(\Delta_1<0\), and let \(\widehat\tau\) be the
number of further rounds after time \(1\) until this selected color is
unanimous, with value \(\infty\) if it never is. Both the initial and
round-one configurations are nonunanimous for all sufficiently large
\(N\) on \(\mathcal H_N\), since \(|\Delta_0|=O(\sqrt{N/\log N})\)
and \(u_N=o(N)\). Conditional on \(\mathcal F_1\), future graphs are
independent of the round-one configuration. Applying
Theorem~\ref{thm:er-polylogarithmic-days-to-unanimity}, with colors
exchanged when necessary, gives
\[
\P\left(
\underline T(|\Delta_1|)+1\le\widehat\tau
\le\overline T(|\Delta_1|)+2
\,\middle|\,\mathcal F_1
\right)
\ge 1-\eta_N(\ell_N)
\qquad\text{on }\mathcal H_N.
\]
The error tends to zero because
\(\ell_N^2\log N/N=\sqrt{\log N}\to\infty\).
On the event in this display, \(\widehat\tau<\infty\), so the opposite
color cannot have reached unanimity earlier: either unanimous state is
absorbing. Therefore
\(\min\{\tau_{\mathrm B},\tau_{\mathrm R}\}=1+\widehat\tau\).
Both horizons are nonincreasing in \(s\) for all sufficiently large
\(N\). Taking conditional expectations consequently yields, with
probability \(1-o(1)\) uniformly in the initial configuration,
\[
\underline T(u_N)+2
\le\min\{\tau_{\mathrm B},\tau_{\mathrm R}\}
\le\overline T(\ell_N)+3.
\]
Finally,
\begin{align*}
\log\frac{KN}{u_N\sqrt{\log N}}
&=\frac12\log N-\frac34\log\log N+\log K,\\
\log\frac{KN}{\ell_N\sqrt{\log N}}
&=\frac12\log N-\frac14\log\log N+\log K,
\end{align*}
whereas each denominator defining the horizons is
\(\frac12\log\log N+O(1)\). Thus both endpoints of the preceding
time bracket are \((1+o(1))\log N/\log\log N\), proving
\eqref{eq:critical-window-consensus-time-asymptotic}.
\end{proof}

\subsection{Proof of Corollary~\ref{cor:er-random-vote-aggregation}}
\begin{proof}
Let \(X_v=z\rvy_0(v)\), so \(X_v=+1\) represents a correct initial
vote, and set
\[
M_N=\sum_{v=1}^N X_v=z\Delta_0.
\]
The variables \(X_v\) are independent, with
\(\E M_N=2\theta\sqrt N\), \(\Var(M_N)=N-4\theta^2\), and uniformly
bounded centered absolute third moments. Berry--Esseen
(Lemma~\ref{lem:berry-esseen}), followed by the change of scale from
\(\sqrt{N-4\theta^2}\) to \(\sqrt N\), gives
\[
\sup_{w\in\R}
\left|\P(M_N/\sqrt N\le w)-\Phi(w-2\theta)\right|
=O_\theta(N^{-1/2}).
\]
In particular,
\(\P(M_N>0)=\Phi(2\theta)+o(1)\) and
\(\P(M_N=0)=o(1)\), establishing the claimed limiting accuracy of the
centralized initial majority.

Choose \(K>K_2^{\mathrm{ER}}(b)\), set \(h_N=(\log N)^{1/4}\), and
write
\[
\ell_N=\sqrt N/h_N,\qquad u_N=h_N\sqrt N,\qquad
\mathcal H_N=\{\ell_N\le|M_N|\le u_N\}.
\]
The preceding Gaussian approximation gives
\[
\P(\mathcal H_N^c)
\le\frac{2}{\sqrt{2\pi}\,h_N}
+2\Phi(2\theta-h_N)+O_\theta(N^{-1/2})=o(1).
\]
On \(\mathcal H_N\), both camps are nonempty and the initial majority
has advantage \(|\Delta_0|=|M_N|\ge\ell_N\). Conditional on the entire
initial configuration, the interaction graphs retain their independent
\(\mathbb G(N,p)\) laws. By
Theorem~\ref{thm:er-polylogarithmic-days-to-unanimity}, with colors
exchanged if necessary, the initial majority reaches unanimity within
its two-sided time bracket except with conditional probability at most
\(\eta_N(\ell_N)=o(1)\). Hence, with unconditional probability
\(1-o(1)\), the eventual unanimous opinion agrees with the initial
majority. Absorption excludes earlier unanimity of the opposite opinion
on this event. Correct unanimity therefore occurs exactly when
\(M_N>0\), apart from an event of probability \(o(1)\). This proves
\eqref{eq:random-vote-correct-unanimity}.

For the time claim, use the horizon notation
\(\underline T(s),\overline T(s)\) from the preceding proof.
Their monotonicity gives, on the same event of probability \(1-o(1)\),
\[
\underline T(u_N)+1
\le\min\{\tau_{\mathrm B},\tau_{\mathrm R}\}
\le\overline T(\ell_N)+2.
\]
The horizon expansions in the proof of
Corollary~\ref{cor:er-critical-window-consensus-time} show that both
endpoints are \((1+o(1))\log N/\log\log N\), completing the proof.
\end{proof}

\section{Proofs of Amplification and Consensus-Time Bounds}\label{app:amplification-proofs}
\label{sec:polylogarithmic-days-to-unanimity}
We present the proof of Theorem~\ref{thm:er-polylogarithmic-days-to-unanimity} in this section. The proof proceeds by first upgrading the one-step estimates from Section~\ref{sec:one-step-evolution-advantage} into a high-probability amplification bound. We then iterate this bound until the advantage reaches the constant-time regime of Theorem~\ref{thm:er-two-day-unanimity}, thereby proving Theorem~\ref{thm:er-polylogarithmic-days-to-unanimity} and Corollary~\ref{cor:er-polylogarithmic-days-to-unanimity-asymptotic}. The one-step failure probabilities are summable because every successful update multiplies both the advantage and the exponent in the next concentration bound.

\subsection{Stepwise advantage amplification}
We first present the stepwise advantage amplification lemma, with its proof deferred to Section~\ref{sec:deferred-Proofs-polylogarithmic-days-to-unanimity}.

\begin{lemma}[Stepwise advantage amplification]
\label{lem:stepwise-advantage-amplification}
For every fixed \(K>0\), there are constants
\(c_1=c_1(b,K)>0\), \(C_2=C_2(b,K)>0\), and \(c=c(b,K)>0\)
such that, whenever
\begin{align}
\sqrt{N/\log N}\le\Delta_t\le KN/\sqrt{\log N},
\label{eqn:stepwise-advantage-amplification-condition}
\end{align}
the following two-sided amplification bound holds:
\begin{align}
\P\left(
 c_1\sqrt{\log N}\,\Delta_t
 \le \Delta_{t+1}
 \le C_2\sqrt{\log N}\,\Delta_t
\,\middle|\,\rvy_t\right)
\ge
1-2\exp\left(-c\frac{\Delta_t^2\log N}{N}\right).
\label{eqn:first-day-two-sided-amplification}
\end{align}
\end{lemma}
The main idea of the proof of Lemma~\ref{lem:stepwise-advantage-amplification} can be summarized as follows: First, Lemma~\ref{lem:expected-variance-degree-differences} places the standardized
degree differences in a bounded Gaussian window. Second, Lemma~\ref{lem:cc-gaussian-flip-probabilities} and Lemma~\ref{lem:conditional-expectations-next-day} show that the conditional mean of the next advantage is of order \(\sqrt{\log N}\,\Delta_t\). Finally, we apply the read-2 concentration inequality in Lemma~\ref{lem:camp-size-read-two-concentration} to upgrade this mean estimate to a high-probability two-sided amplification bound. The detailed proof is deferred to Section~\ref{sec:deferred-Proofs-polylogarithmic-days-to-unanimity}.

\subsection{Proof of the amplification lemma}\label{sec:deferred-Proofs-polylogarithmic-days-to-unanimity}

\begin{proof}[Proof of Lemma~\ref{lem:stepwise-advantage-amplification}]
    Condition throughout on \(\rvy_t\), and write
    \[
    s_N\coloneqq\sqrt{(N-1)p(1-p)}.
    \]
    The upper bound in
    \eqref{eqn:stepwise-advantage-amplification-condition} gives
    \(\Delta_t=o(N)\), so both camps have size at least \(N/3\) for all
    sufficiently large \(N\). Moreover,
    Lemma~\ref{lem:expected-variance-degree-differences} and
    \eqref{eqn:stepwise-advantage-amplification-condition} give
    \[
    \max\{|x_t^{\gR}|,|x_t^{\gB}|\}
    \le
    \frac{p(\Delta_t+1)}{s_N}
    \le M(b,K).
    \]
    The lower bound in
    \eqref{eqn:stepwise-advantage-amplification-condition} also gives
    \(\Delta_t\ge1\). Hence Lemma~\ref{lem:cc-gaussian-flip-probabilities}
    applies, and \eqref{eq:cc-gaussian-flip-difference}, together with
    \(p=b\log N/N\), yields
    \[
    N\bigl(p_t^{\gR}-p_t^{\gB}\bigr)
    \asymp
    \frac{Np}{s_N}\Delta_t
    \asymp
    \sqrt{\log N}\,\Delta_t,
    \]
    where the implicit constants depend only on \(b\) and \(K\).
    
    We denote \(\mu_t\coloneqq\E[\Delta_{t+1}\mid\rvy_t]\). By
    \eqref{eq:advantage-expectation-expanded} and
    \(0\le p_t^{\gR},p_t^{\gB}\le1\),
    \[
    \left|
    \mu_t-N\bigl(p_t^{\gR}-p_t^{\gB}\bigr)
    \right|
    =
    \Delta_t\left|1-p_t^{\gR}-p_t^{\gB}\right|
    \le
    \Delta_t.
    \]
    Since \(\sqrt{\log N}\to\infty\), we may therefore choose constants
    \(c_1,C_2>0\), depending only on \(b,K\), such that
    \begin{align}
    2c_1\sqrt{\log N}\,\Delta_t
    \le
    \mu_t
    \le
    (C_2-c_1)\sqrt{\log N}\,\Delta_t.
    \label{eqn:first-day-expectation-window}
    \end{align}
    
    Finally, Lemma~\ref{lem:camp-size-read-two-concentration} gives, for every
    \(u>0\),
    \[
    \P\left(
    |\Delta_{t+1}-\mu_t|\ge u
    \,\middle|\,\rvy_t
    \right)
    \le
    2\exp\left(-\frac{u^2}{4N}\right).
    \]
    Taking \(u=c_1\sqrt{\log N}\,\Delta_t\) and using
    \eqref{eqn:first-day-expectation-window}, we obtain
    \[
    c_1\sqrt{\log N}\,\Delta_t
    \le
    \Delta_{t+1}
    \le
    C_2\sqrt{\log N}\,\Delta_t
    \]
    except with conditional probability at most
    \[
    2\exp\left(-\frac{c_1^2}{4}
    \frac{\Delta_t^2\log N}{N}\right).
    \]
    Taking \(c=c_1^2/4\) proves
    \eqref{eqn:first-day-two-sided-amplification}.
    \end{proof}

\subsection{Proof of Theorem~\ref{thm:er-polylogarithmic-days-to-unanimity}}
\begin{proof}[Proof of Theorem~\ref{thm:er-polylogarithmic-days-to-unanimity}]
Let \(c_1,C_2,c>0\) be the constants from the preceding
Lemma~\ref{lem:stepwise-advantage-amplification}, and define
\begin{align}
a_N&\coloneqq c_1\sqrt{\log N},
&
A_N&\coloneqq C_2\sqrt{\log N},
&
\lambda_N&\coloneqq \frac{\Delta_0^2\log N}{N}.
\label{eqn:polylog-a-lambda}
\end{align}
Condition~\eqref{eqn:er-polylog-initial-advantage-condition} gives
\begin{align}
\lambda_N
\ge
\frac{\log N}{h_N^2}
\longrightarrow
\infty,
\label{eqn:polylog-lambda-diverges}
\end{align}
because \(h_N=o(\sqrt{\log N})\). In particular, for all sufficiently large
\(N\),
\[
a_N>1,
\qquad
A_N>1,
\qquad
(a_N)^2\ge2,
\qquad
\Delta_0\ge\sqrt{N/\log N}.
\]
For both the lower and upper bounds, we will use the following geometric estimate: for every integer \(m\ge0\), with the sum understood as zero when \(m=0\),
\begin{align}
2\sum_{t=0}^{m-1}
\exp\left(-c\lambda_N(a_N)^{2t}\right)
&\le
2\sum_{t=0}^{\infty}
\exp\left(-c\lambda_N(t+1)\right) =
\frac{2\exp(-c\lambda_N)}
{1-\exp(-c\lambda_N)}
\le
4\exp(-c\lambda_N),
\label{eqn:polylog-geometric-failure-bound}
\end{align}
where we used \((a_N)^{2t}\ge2^t\ge t+1\) and \eqref{eqn:polylog-lambda-diverges}.

We first prove the lower bound on the time to unanimity. If
\(\underline{T}=0\), then \(\gR_0\ne\emptyset\) gives
\(\tau_{\mathrm B}\geq1=\underline{T}+1\), so the desired bound is
immediate. Assume therefore that \(\underline{T}\geq1\). For \(m\ge0\), define
\begin{align}
\mathcal E_m^-
\coloneqq
\bigcap_{s=0}^{m-1}
\left\{
a_N\Delta_s\le\Delta_{s+1}\le A_N\Delta_s
\right\}.
\label{eqn:polylog-lower-good-event}
\end{align}
Here and below, an intersection over an empty index set is the sure event;
in particular, \(\mathcal E_0^-\) is the sure event. By
\eqref{eqn:polylog-lower-good-event}, on \(\mathcal E_t^-\) every update
from time \(0\) through time \(t-1\) satisfies the two-sided amplification
inequality. Consequently,
\begin{align}
(a_N)^t\Delta_0
\le
\Delta_t
\le
(A_N)^t\Delta_0.
\label{eqn:polylog-two-sided-pathwise-amplification}
\end{align}
For every \(0\le t<\underline{T}\), the definition of
\(\underline{T}\) gives
\[
(A_N)^t\Delta_0
\le
(A_N)^{\underline{T}}\Delta_0
\le
K\frac{N}{\sqrt{\log N}}.
\]
Together with \eqref{eqn:polylog-two-sided-pathwise-amplification} and the
initial condition, this shows that, on \(\mathcal E_t^-\),
\[
\sqrt{\frac{N}{\log N}}
\le
\Delta_t
\le
K\frac{N}{\sqrt{\log N}}.
\]
Thus the amplification lemma applies conditionally on \(\mathcal F_t\).
Moreover, the definition in \eqref{eqn:polylog-lower-good-event} gives
\begin{align*}
(\mathcal E_{t+1}^-)^{\complement}\cap\mathcal E_t^-
=
\mathcal E_t^-
\cap
\left(
\{\Delta_{t+1}<a_N\Delta_t\}
\cup
\{\Delta_{t+1}>A_N\Delta_t\}
\right).
\end{align*}
Using both sides of
\eqref{eqn:first-day-two-sided-amplification} and the tower property, we
therefore obtain, for every \(0\le t<\underline{T}\),
\begin{align}
\P\left(
(\mathcal E_{t+1}^-)^{\complement}\cap\mathcal E_t^-
\,\middle|\,
\rvy_0
\right) =&\,
\E\left[
\indi{\mathcal E_t^-}
\P\left(
\Delta_{t+1}\notin[a_N\Delta_t,A_N\Delta_t]
\,\middle|\,
\mathcal F_t
\right)
\,\middle|\,
\rvy_0
\right]\notag\\
\leq &\,
2\exp\left(
-c\frac{((a_N)^t\Delta_0)^2\log N}{N}
\right)
=
2\exp\left(-c\lambda_N(a_N)^{2t}\right).
\end{align}
By \eqref{eqn:polylog-lower-good-event}, the events
\(\{\mathcal E_m^-\}_{m\ge0}\) are nested. Hence a union bound over the first
failed update, followed by \eqref{eqn:polylog-geometric-failure-bound}, gives
\begin{align}
\P\left((\mathcal E_{\underline{T}}^-)^{\complement}\mid\rvy_0\right)
&\le
2\sum_{t=0}^{\underline{T}-1}
\exp\left(-c\lambda_N(a_N)^{2t}\right)
\le
4\exp(-c\lambda_N).
\end{align}
On the event \(\mathcal E_{\underline{T}}^-\), we have
\[
\Delta_{\underline{T}}
\le
(A_N)^{\underline{T}}\Delta_0
\le
K\frac{N}{\sqrt{\log N}}
<
N
\]
for all sufficiently large \(N\). Hence
\(\gR_{\underline{T}}\ne\emptyset\). Since time is integer-valued and the
all-blue configuration is absorbing, this implies
\(\tau_{\mathrm B}\geq \underline{T}+1\). Together with the
zero-horizon case, this proves
\begin{align}
\P\left(\tau_{\mathrm B}<\underline{T}+1\mid\rvy_0\right)
\le
4\exp(-c\lambda_N).
\label{eqn:polylog-lower-tail-probability}
\end{align}

We then prove the upper bound. Define the random entrance time into the
two-day completion window by
\begin{align}
\tau_K
\coloneqq
\inf\left\{
t\ge0:
\Delta_t\ge K\frac{N}{\sqrt{\log N}}
\right\},
\label{eqn:polylog-entrance-time}
\end{align}
with the convention \(\inf\emptyset=\infty\). For \(m\ge0\), define
\begin{align}
\mathcal E_m
\coloneqq
\bigcap_{s=0}^{m-1}
\left(
\{\tau_K\le s\}
\cup
\{\Delta_{s+1}\ge a_N\Delta_s\}
\right).
\label{eqn:polylog-upper-good-event}
\end{align}
By the empty-intersection convention, \(\mathcal E_0\) is the sure event.
If \(\mathcal E_t\cap\{\tau_K>t\}\) occurs, then \(\tau_K>s\) for every
\(0\le s<t\); hence \eqref{eqn:polylog-upper-good-event} forces every
preceding update to satisfy the lower-amplification inequality. Therefore,
\begin{align}
\Delta_t\ge (a_N)^t\Delta_0.
\label{eqn:polylog-pathwise-amplification}
\end{align}
Moreover, \(\tau_K>t\) gives
\(\Delta_t<K N/\sqrt{\log N}\), while
\eqref{eqn:polylog-pathwise-amplification} and the initial condition give
\(\Delta_t\ge\sqrt{N/\log N}\). Thus the amplification lemma applies
conditionally on \(\mathcal F_t\). Using only the lower side of
\eqref{eqn:first-day-two-sided-amplification} and then the tower property, we
obtain, for every \(0\le t<\overline{T}\),
\begin{align}
\P\left(
\mathcal E_{t+1}^{\complement}\cap\mathcal E_t
\,\middle|\,
\rvy_0
\right) =&\,
\P\left(
\mathcal E_t\cap\{\tau_K>t\}
\cap\{\Delta_{t+1}<a_N\Delta_t\}
\,\middle|\,
\rvy_0
\right)\notag\\
=&\,
\E\left[
\indi{\mathcal E_t\cap\{\tau_K>t\}}
\P\left(
\Delta_{t+1}<a_N\Delta_t
\,\middle|\,
\mathcal F_t
\right)
\,\middle|\,
\rvy_0
\right]\notag\\
\leq &\,
2\exp\left(
-c\frac{((a_N)^t\Delta_0)^2\log N}{N}
\right)
=
2\exp\left(-c\lambda_N(a_N)^{2t}\right).
\end{align}
The definition of \(\overline{T}\) in
\eqref{eqn:er-polylog-sufficient-time} ensures that
\begin{align}
(a_N)^{\overline{T}}\Delta_0
\ge
K\frac{N}{\sqrt{\log N}}.
\label{eqn:polylog-T-threshold}
\end{align}
This also covers \(\overline{T}=0\), in which case the initial
configuration already lies in the completion window. If
\(\mathcal E_{\overline{T}}\) occurs and
\(\tau_K>\overline{T}\), then
\eqref{eqn:polylog-pathwise-amplification} and
\eqref{eqn:polylog-T-threshold} contradict the definition of \(\tau_K\).
Consequently,
\(\{\tau_K>\overline{T}\}
\subseteq\mathcal E_{\overline{T}}^{\complement}\). By
\eqref{eqn:polylog-upper-good-event}, the events
\((\mathcal E_m)_{m\ge0}\) are nested. Hence a union bound over the first
failed update, followed by \eqref{eqn:polylog-geometric-failure-bound}, gives
\begin{align}
\P(\tau_K>\overline{T}\mid\rvy_0)
&\le
2\sum_{t=0}^{\overline{T}-1}
\exp\left(-c\lambda_N(a_N)^{2t}\right)
\le
4\exp(-c\lambda_N),
\label{eqn:polylog-entrance-probability}
\end{align}
where the sum is zero when \(\overline{T}=0\).

It remains to pass from entrance to unanimity. Let
\(\xi=\xi(b,K)>0\) be supplied by
Theorem~\ref{thm:er-two-day-unanimity}. For each deterministic \(s\ge0\), the
graphs sampled after time \(s\) are independent of \(\mathcal F_s\) and have
the same law as the graphs sampled from time zero. On \(\{\tau_K=s\}\), we have
\(\Delta_s\ge K N/\sqrt{\log N}\), so the two-day theorem, shifted to time
\(s\), gives
\[
\P\left(\gR_{s+2}\ne\emptyset\mid\mathcal F_s\right)
\le
2\exp\left(-\left(\log N\right)^2\right)+2N^{-\xi}
\qquad\text{on }\{\tau_K=s\}.
\]
Therefore, another application of the tower property yields
\begin{align}
&\, \P\left(
\tau_K\le \overline{T},\ \gR_{\tau_K+2}\ne\emptyset
\,\middle|\,
\rvy_0
\right)\notag\\
=&\,
\sum_{s=0}^{\overline{T}}
\E\left[
\indi{\{\tau_K=s\}}
\P\left(\gR_{s+2}\ne\emptyset\mid\mathcal F_s\right)
\,\middle|\,
\rvy_0
\right]
\le
2\exp\left(-\left(\log N\right)^2\right)+2N^{-\xi}.
\label{eqn:polylog-completion-probability}
\end{align}
The all-blue configuration is absorbing. Hence
\[
\{\gR_{\overline{T}+2}\ne\emptyset\}
\subseteq
\{\tau_K>\overline{T}\}
\cup
\{\tau_K\le\overline{T},\
\gR_{\tau_K+2}\ne\emptyset\}.
\]
Combining \eqref{eqn:polylog-entrance-probability} and
\eqref{eqn:polylog-completion-probability} gives
\begin{align}
\P\left(\tau_{\mathrm B}>\overline{T}+2\mid\rvy_0\right)
&\le
4\exp(-c\lambda_N)
+2\exp\left(-\left(\log N\right)^2\right)
+2N^{-\xi}.
\label{eqn:polylog-upper-tail-probability}
\end{align}
Finally, a union bound with
\eqref{eqn:polylog-lower-tail-probability} gives
\begin{align*}
&\, \P\left(
\left\{
\underline{T}+1
\le\tau_{\mathrm B}\le
\overline{T}+2
\right\}^{\complement}
\,\middle|\,
\rvy_0
\right)\\
\leq &\,
8\exp\left(-c\frac{\Delta_0^2\log N}{N}\right)
+2\exp\left(-\left(\log N\right)^2\right)
+2N^{-\xi}
=
\eta_N(\Delta_0).
\end{align*}
Taking complements proves
\eqref{eqn:er-polylog-unanimity-time-bracket}.
\end{proof}

\subsection{Proof of Corollary~\ref{cor:er-polylogarithmic-days-to-unanimity-asymptotic}}
\begin{proof}[Proof of
Corollary~\ref{cor:er-polylogarithmic-days-to-unanimity-asymptotic}]
Suppose that
\[
\Delta_0\asymp\frac{\sqrt N}{h_N}.
\]
Since \(h_N\ge1\) and \(h_N=o(\sqrt{\log N})\),
\begin{align*}
\log\left(\frac{KN}{\Delta_0\sqrt{\log N}}\right)
&=
\frac12\log N-\frac12\log\log N+\log h_N+O(1)
=
\frac12\log N+O(\log\log N),\\
\log\left(C_2\sqrt{\log N}\right)
&=
\frac12\log\log N+O(1),\\
\log\left(c_1\sqrt{\log N}\right)
&=
\frac12\log\log N+O(1).
\end{align*}
Substitution into \eqref{eqn:er-polylog-necessary-time} gives
\[
\underline{T}
=
\left(1+o(1)\right)\frac{\log N}{\log\log N}.
\]
Separately, substitution into \eqref{eqn:er-polylog-sufficient-time} gives
\[
\overline{T}
=
\left(1+o(1)\right)\frac{\log N}{\log\log N}.
\]
For every fixed \(0<\delta<1\), these two asymptotics imply, for all
sufficiently large \(N\),
\[
\left\lfloor(1-\delta)\frac{\log N}{\log\log N}\right\rfloor
\leq\underline{T},
\qquad
\overline{T}+2\leq
\left\lceil(1+\delta)\frac{\log N}{\log\log N}\right\rceil.
\]
The all-blue configuration is absorbing, so
\eqref{eqn:er-polylog-unanimity-time-asymptotics} follows from
\eqref{eqn:er-polylog-unanimity-time-bracket}.
This completes the proof.
\end{proof}

\section{Proofs of Rapid Unanimity}\label{app:rapid-proofs}\label{sec:constant-days-to-unanimity}

We present the proof of Theorem~\ref{thm:er-two-day-unanimity} in this section. The proof proceeds backward from unanimity. We first identify a red-camp size below which all remaining red vertices disappear in one update. We then show that an advantage of order \(N/\sqrt{\log N}\), with coefficient strictly above \(K_2^{\mathrm{ER}}(b)\), reaches this extinction region in one update. 

We first present two Lemmas needed, with their proofs deferred to Section~\ref{sec:deferred-Proofs-constant-days-to-unanimity}.

\subsection{Reduction to the extinction regime}

\begin{lemma}[A small red camp disappears in one day]
\label{lem:one-day-extinction}
Fix a deterministic time \(t\ge0\) and a constant \(0<r<r_*\), where \(r_*\)
is defined in \eqref{eqn:K2-def}. If
\begin{align}
|\gR_t|\le rN,
\label{eqn:one-day-extinction-condition}
\end{align}
then there exists \(\xi=\xi(b,r)>0\) such that, for all sufficiently large
\(N\),
\begin{align}
\P\left(\gR_{t+1}\neq\emptyset\,\middle|\,\rvy_t\right)
\le 2N^{-\xi}.
\label{eqn:one-day-extinction-probability}
\end{align}
\end{lemma}

For a fixed \(K>K_2^{\mathrm{ER}}(b)\), define
\begin{align}
q_K
&\coloneqq
\Phi(-\sqrt b\,K).
\label{eqn:qK-def}
\end{align}
By \eqref{eqn:K2-def}, \(K>K_2^{\mathrm{ER}}(b)\) is equivalent to
\(q_K<r_*\), so one can choose an extinction density strictly between
these two quantities.

\begin{lemma}[One-step reduction below the extinction threshold]
\label{lem:one-step-reduction}
Fix a deterministic time \(t\ge0\), condition on \(\rvy_t\), and suppose that
for some fixed \(K>K_2^{\mathrm{ER}}(b)\),
\begin{align}
\Delta_t\ge K\frac{N}{\sqrt{\log N}}.
\label{eqn:one-step-reduction-condition}
\end{align}
Then, for any fixed \(r\) satisfying
\begin{align}
q_K<r<r_*,
\label{eqn:one-step-extinction-density-condition}
\end{align}
we have, for all sufficiently large \(N\),
\begin{align}
\P\left(
|\gR_{t+1}|\le rN
\,\middle|\,
\rvy_t
\right)
\ge
1-2\exp\left(-\left(\log N\right)^2\right).
\label{eqn:one-step-reduction-probability}
\end{align}
\end{lemma}

\subsection{Proofs of the reduction lemmas}
\label{sec:deferred-Proofs-constant-days-to-unanimity}
\subsubsection{Proof of Lemma~\ref{lem:one-day-extinction}}
\begin{proof}[Proof of Lemma~\ref{lem:one-day-extinction}]
    Condition on \(\rvy_t\) and assume \eqref{eqn:one-day-extinction-condition}.
    If \(\gR_t=\emptyset\), then no blue vertex has a red neighbor, and hence
    \(\gR_{t+1}=\emptyset\) deterministically. We may therefore assume
    \(1\le |\gR_t|\le rN\), so \(|\gB_t|\ge(1-r)N\).
    
    To use the uniform large-deviation estimate even when \(|\gR_t|\) is
    sublinear, we choose
    \[
    m_N\coloneqq\lfloor rN\rfloor.
    \]
    By \eqref{eqn:D-Rt-distribution} and \eqref{eqn:D-Bt-distribution}, both the
    probability that a red vertex stays red and the probability that a blue
    vertex flips to red are nondecreasing in the current red-camp size. Since \(|\gR_t|\) is an integer, \(|\gR_t|\le m_N\). It therefore suffices to bound the two probabilities when \(|\gR_t|=m_N\), which we assume for the remainder of this estimate.
    
    In this worst-case configuration, the scaled camp sizes defined in
    \eqref{eqn:scaled-camp-sizes} satisfy
    \[
    b_t^{\gR}=\frac{bm_N}{N},
    \qquad
    b_t^{\gB}=\frac{b(N-m_N)}{N}.
    \]
    
    Since \(m_N/N\to r\in(0,1/2)\), this worst-case configuration satisfies
    \eqref{eqn:linear-camp-condition}, for example with \(c_0=r/2\), for all
    sufficiently large \(N\). Corollary~\ref{cor:flip-probability-ld} therefore
    gives
    \begin{align*}
    q_t^{\gR}
    &=
    N^{-I(b_t^{\gB},b_t^{\gR})+o(1)},
    &
    p_t^{\gB}
    &=
    N^{-I(b_t^{\gB},b_t^{\gR})+o(1)}.
    \end{align*}
    Moreover,
    \begin{align*}
    I(b_t^{\gB},b_t^{\gR})
    &=
    \left(\sqrt{b_t^{\gB}}-\sqrt{b_t^{\gR}}\right)^2
    =
    b\left(\sqrt{1-r}-\sqrt{r}\right)^2+o(1).
    \end{align*}
    Because \(r<r_*\), the definition of \(r_*\) implies that
    \[
    \kappa
    \coloneqq
    b\left(\sqrt{1-r}-\sqrt{r}\right)^2-1
    >0.
    \]
    Thus, with \(\xi\coloneqq\kappa/2\), the uniform \(o(1)\) terms in
    Corollary~\ref{cor:flip-probability-ld} yield, for all sufficiently large
    \(N\),
    \[
    q_t^{\gR}\le N^{-1-\xi},
    \qquad
    p_t^{\gB}\le N^{-1-\xi}.
    \]
    By the preceding monotonicity reduction, these bounds hold for every original
    configuration satisfying \(|\gR_t|\le rN\).
    A vertex is red at time \(t+1\) only if it was red and stayed red, or it was
    blue and flipped to red. The union bound therefore gives
    \begin{align*}
    \P\left(\gR_{t+1}\neq\emptyset\,\middle|\,\rvy_t\right)
    &\le
    |\gR_t|q_t^{\gR}+|\gB_t|p_t^{\gB}\\
    &\le
    N\cdot N^{-1-\xi}+N\cdot N^{-1-\xi}
    =2N^{-\xi}.
    \end{align*}
    The constant \(\xi>0\) depends only on \(b\) and \(r\), which proves the
    lemma.
    \end{proof}

\subsubsection{Proof of Lemma~\ref{lem:one-step-reduction}}
\begin{proof}[Proof of Lemma~\ref{lem:one-step-reduction}]
    Throughout the proof, condition on \(\rvy_t\) and assume
    \eqref{eqn:one-step-reduction-condition}. Fix \(r\) satisfying
    \eqref{eqn:one-step-extinction-density-condition}, and define
    \[
    \rho\coloneqq\frac{q_K+r}{2}.
    \]
    Then \(q_K<\rho<r<r_*\). We claim that, uniformly over every
    configuration satisfying \eqref{eqn:one-step-reduction-condition},
    \begin{align}
    \E\left[|\gR_{t+1}|\,\middle|\,\rvy_t\right]
    \le \rho N+o(N).
    \label{eq:one-step-mean-rho}
    \end{align}
    Assuming \eqref{eq:one-step-mean-rho}, the desired high-probability bound
    follows directly from Lemma~\ref{lem:camp-size-read-two-concentration},
    applied at time \(t\) with \(s=\sqrt N\log N\):
    \begin{align*}
    \P\left(
    \left||\gR_{t+1}|-
    \E\left[|\gR_{t+1}|\,\middle|\,\rvy_t\right]\right|
    \ge \sqrt N\log (N)
    \,\middle|\,
    \rvy_t
    \right)
    \le
    2\exp\left(-(\log (N))^2\right).
    \end{align*}
    Since \(\rho<r\) and \(\sqrt N\log N=o(N)\),
    \eqref{eq:one-step-mean-rho} implies that \(|\gR_{t+1}|\le rN\) for all
    sufficiently large \(N\), except with conditional probability at most
    \(2\exp\left(-(\log N)^2\right)\).
    
    It remains to prove \eqref{eq:one-step-mean-rho}. We divide the proof into two
    cases according to the size of \(\gR_t\).
    
    \noindent\textbf{Case 1: \(|\gR_t|\le \rho N\)}. In this case, the red camp
    is already below the intermediate density \(\rho<r_*\). Applying
    Lemma~\ref{lem:one-day-extinction} with extinction density \(\rho\), we obtain
    some \(\xi_\rho>0\) such that
    \begin{align*}
    \E\left[|\gR_{t+1}|\,\middle|\,\rvy_t\right]
    \le
    N\P\left(\gR_{t+1}\ne\emptyset\,\middle|\,\rvy_t\right)
    \le 2N^{1-\xi_\rho}
    =o(N).
    \end{align*}
    This proves \eqref{eq:one-step-mean-rho} in Case~1.
    
    \noindent\textbf{Case 2: \(|\gR_t|>\rho N\)}. Both camps have linear size.
    Since \eqref{eqn:one-step-reduction-condition} implies \(\Delta_t>0\),
    \eqref{eqn:camp-sizes-from-advantage-intro} gives
    \begin{align*}
    |\gR_t|
    =
    \frac{N-\Delta_t}{2}
    \le
    \frac N2,
    \qquad
    |\gB_t|
    =
    \frac{N+\Delta_t}{2}
    \ge
    \frac N2.
    \end{align*}
    Together with \(|\gR_t|>\rho N\), this allows us to apply
    Corollary~\ref{cor:gaussian-flip-probabilities} uniformly in the present case.
    
    Set \(\gamma_N\coloneqq\Delta_t/N\). Then
    \(0<\gamma_N\le 1\), and \eqref{eqn:one-step-reduction-condition}
    implies
    \begin{align*}
    \gamma_N\sqrt{\log (N)}
    \ge
    K.
    \end{align*}
    We first bound the probability \(q_t^{\gR}\) that a red vertex remains red. By
    \eqref{eqn:expected-degree-differences-R-B},
    \begin{align*}
    m_t^{\gR}
    =
    p(\Delta_t+1)
    =
    b\gamma_N\log (N)+\frac{b\log (N)}{N}.
    \end{align*}
    Moreover, the variance formula in the proof of
    Lemma~\ref{lem:expected-variance-degree-differences} gives
    \begin{align*}
    v_t^{\gR}
    \le
    b\log N
    .
    \end{align*}
    Combining all the estimates above, we obtain
    \begin{align*}
    x_t^{\gR}
    &\ge
    \sqrt b\,\gamma_N\sqrt{\log (N)}+o(1)
    \ge
    \sqrt b\,K+o(1).
    \end{align*}
    Therefore, by \eqref{eq:gaussian-qR} in
    Corollary~\ref{cor:gaussian-flip-probabilities},
    \begin{align*}
    q_t^{\gR}
    &\le
    \Phi(-x_t^{\gR})
    +
    O(1/\sqrt{\log (N)})
    \le
    \Phi(-\sqrt b\,K+o(1)) + O(1/\sqrt{\log (N)})
    =
    q_K+o(1).
    \end{align*}
    
    We next bound \(p_t^{\gB}\), the probability that a blue vertex flips to red.
    By \eqref{eqn:expected-degree-differences-R-B},
    \begin{align*}
    -m_t^{\gB}
    =
    b\gamma_N\log (N)
    -
    \frac{b\log (N)}{N}.
    \end{align*}
    The same variance computation gives
    \begin{align*}
    v_t^{\gB}
    \le
    b\log N
    .
    \end{align*}
    Combining the preceding displays gives
    \begin{align*}
    -x_t^{\gB}
    &\ge
    \sqrt b\,\gamma_N\sqrt{\log (N)}-o(1)
    \ge
    \sqrt b\,K-o(1).
    \end{align*}
    Equivalently,
    \(x_t^{\gB}\le-\sqrt b\,K+o(1)\). Therefore, by
    \eqref{eq:gaussian-pB} in
    Corollary~\ref{cor:gaussian-flip-probabilities},
    \begin{align*}
    p_t^{\gB}
    &\le
    \Phi(x_t^{\gB})
    +
    O(1/\sqrt{\log (N)})
    \le
    \Phi(-\sqrt b\,K+o(1)) + O(1/\sqrt{\log (N)})
    =
    q_K+o(1).
    \end{align*}
    Combining \(q_t^{\gR}\le q_K+o(1)\) and
    \(p_t^{\gB}\le q_K+o(1)\) with \eqref{eq:R-next-expectation}, we obtain
    \begin{align*}
    \E\left[|\gR_{t+1}|\,\middle|\,\rvy_t\right]
    &=
    |\gR_t|q_t^{\gR}+|\gB_t|p_t^{\gB} \notag\\
    &\le
    (q_K+o(1))N
    \le
    \rho N+o(N),
    \end{align*}
    where the last inequality uses \(q_K<\rho\). This proves
    \eqref{eq:one-step-mean-rho} in Case 2.
    
    Therefore, \eqref{eq:one-step-mean-rho} holds in both cases, completing the
    proof.
    \end{proof}

\subsection{Proof of Theorem~\ref{thm:er-two-day-unanimity}}
The preceding lemmas form a nested constant-time mechanism. Lemma~\ref{lem:one-day-extinction} eliminates a sufficiently small red camp in one update, while Lemma~\ref{lem:one-step-reduction} reaches that extinction region in one update from a sufficiently large advantage. We now combine these two steps to prove Theorem~\ref{thm:er-two-day-unanimity}.

\begin{proof}[Proof of Theorem~\ref{thm:er-two-day-unanimity}]
We first prove one-day unanimity. Fix \(0<r<r_*\) and assume
\eqref{eqn:er-one-day-unanimity-condition}. By
\eqref{eqn:camp-sizes-from-advantage-intro},
\[
|\gR_0|
=
\frac{N-\Delta_0}{2}
\le
rN.
\]
Thus \eqref{eqn:one-day-extinction-condition} holds at time \(t=0\).
Lemma~\ref{lem:one-day-extinction} gives a constant
\(\xi=\xi(b,r)>0\) such that
\[
\P(\gR_1\neq\emptyset\mid\rvy_0)
\le
2N^{-\xi}.
\]
Taking complements proves part~\textup{(ii)}.

We then prove the two-day unanimity case. Fix \(K>K_2^{\mathrm{ER}}(b)\) and assume
\eqref{eqn:er-two-day-unanimity-condition}. Since \(q_K<r_*\), we choose \(r\) as
\[
r
\coloneqq
\frac{q_K+r_*}{2},
\]
which satisfies \eqref{eqn:one-step-extinction-density-condition}. Define the event
\[
\mathcal E_1
\coloneqq
\{|\gR_1|\le rN\}.
\]
The hypothesis \eqref{eqn:er-two-day-unanimity-condition} is precisely
\eqref{eqn:one-step-reduction-condition} at time \(t=0\). Hence
Lemma~\ref{lem:one-step-reduction} gives
\[
\P(\mathcal E_1^c\mid\rvy_0)
\le
2\exp\left(-\left(\log N\right)^2\right).
\]
On the event \(\mathcal E_1\), we have
\[
\Delta_1
=
N-2|\gR_1|
\ge
(1-2r)N.
\]
Because the graph used for the update from time \(1\) to time \(2\) is
freshly resampled independently of \(\rvy_1\), the one-day result proved
above applies conditionally at time \(t=1\). Hence there is a constant
\(\xi=\xi(b,r)>0\) such that
\[
\P(\gR_2\neq\emptyset\mid\rvy_1)
\le
2N^{-\xi}
\qquad\text{on }\mathcal E_1.
\]
Taking conditional expectations with respect to \(\rvy_1\) and separating
according to \(\mathcal E_1\), we obtain
\begin{align*}
\P(\gR_2\neq\emptyset\mid\rvy_0)
&\le
\P(\mathcal E_1^c\mid\rvy_0)
+
\E\left[
\indi{\mathcal E_1}
\P(\gR_2\neq\emptyset\mid\rvy_1)
\,\middle|\,
\rvy_0
\right]\\
&\le
2\exp\left(-\left(\log N\right)^2\right)
+
2N^{-\xi}.
\end{align*}
Because the chosen \(r\) depends only on \(b\) and \(K\), we may write
\(\xi=\xi(b,K)\). Taking complements proves part~\textup{(i)} and completes
the proof.
\end{proof}

\section{One-Step Probability Estimates}
\label{app:one-step-estimates}
This appendix develops the local estimates used throughout the proofs:
conditional degree differences and their moments, continuity-corrected
Gaussian formulas, and large-deviation bounds for the final extinction step.

\subsection{Degree differences and conditional expectations}\label{sec:one-step-evolution-advantage}
In this section, we present the necessary preliminaries to study the evolution of the advantage. The proof of Lemma~\ref{lem:cc-gaussian-flip-probabilities} is deferred to Section~\ref{sec:poisson-proof-cc-gaussian}.

Recall the opinion update rule in \eqref{eqn:opinion-update}. To facilitate our analysis, whenever the corresponding camp is nonempty, we introduce the degree differences for a blue vertex $u\in \gB_{t}$ and a red vertex $v\in \gR_{t}$:
\begin{align}
    \rD_{t}^{\gB}(u) \coloneqq N^{\gR}_{t}(u) - N^{\gB}_{t}(u), \qquad \rD_{t}^{\gR}(v) \coloneqq N^{\gB}_{t}(v) - N^{\gR}_{t}(v). \label{eqn:degree-differences}
\end{align}
Conditioned on \(\rvy_{t}\), $u$ flips to red on day \(t+1\) if and only if \(\rD_{t}^{\gB}(u)>0\), while $v$ flips to blue on day \(t+1\) if and only if \(\rD_{t}^{\gR}(v)>0\). Since all edges in \(\gG_{t+1}\) have the same probability \(p\), the laws of these differences depend only on the current camp sizes. For convenience, we omit the vertex argument and write
\begin{subequations}
\begin{align}
    \rD_{t}^{\gB}\stackrel d=&\, \Bin(|\gR_{t}|,p)-\Bin(|\gB_{t}|-1,p),\label{eqn:D-Bt-distribution}\\
    \rD_{t}^{\gR} \stackrel d=&\, \Bin(|\gB_{t}|,p)-\Bin(|\gR_{t}|-1,p).\label{eqn:D-Rt-distribution}
\end{align}
\end{subequations}
We then introduce the flip probabilities as follows whenever the corresponding
camp is nonempty:
\begin{subequations}
    \begin{align}
        p_{t}^{\gB} \coloneqq &\, \P(\rD_{t}^{\gB}>0\mid\rvy_{t}), \qquad q_{t}^{\gB} \coloneqq \P(\rD_{t}^{\gB}\le0\mid\rvy_{t}), \label{eqn:p-q-B}\\
        p_{t}^{\gR} \coloneqq &\, \P(\rD_{t}^{\gR}>0\mid\rvy_{t}), \qquad q_{t}^{\gR} \coloneqq \P(\rD_{t}^{\gR}\le0\mid\rvy_{t}), \label{eqn:p-q-R}
    \end{align}
\end{subequations}
where $p_{t}^{\gR}$ (resp. $p_{t}^{\gB}$) denotes the probability that a red (resp. blue) vertex flips to blue (resp. red) on day $t + 1$, while $q_{t}^{\gR}$ (resp. $q_{t}^{\gB}$) denotes the probability that a red (resp. blue) vertex stays red (resp. blue) on day $t + 1$. If a camp is empty, we set its flip probability to zero and its stay probability to one; with these conventions, the camp-size expectation identities below remain valid at unanimity. When both camps are nonempty, we define the means and variances of $\rD_{t}^{\gB}$ and $\rD_{t}^{\gR}$ via
\begin{subequations}
\begin{align}
    m_{t}^{\gB} \coloneqq &\, \E\left[\rD_{t}^{\gB}\mid \rvy_{t}\right], \qquad v_{t}^{\gB} \coloneqq \Var\left(\rD_{t}^{\gB}\mid \rvy_{t}\right), \label{eqn:m-v-B}\\
    m_{t}^{\gR} \coloneqq &\,\E\left[\rD_{t}^{\gR}\mid \rvy_{t}\right], \qquad v_{t}^{\gR} \coloneqq \Var\left(\rD_{t}^{\gR}\mid \rvy_{t}\right). \label{eqn:m-v-R}
\end{align}
\end{subequations}
We also define the standardized means
\begin{align}
x_t^{\gR}
&\coloneqq\frac{m_t^{\gR}}{\sqrt{v_t^{\gR}}},
&
x_t^{\gB}
&\coloneqq\frac{m_t^{\gB}}{\sqrt{v_t^{\gB}}}.
\label{eqn:standardized-degree-difference-definitions}
\end{align}
We also define the scaled camp sizes
\begin{align}
b_t^{\gR}
&\coloneqq\frac{b|\gR_t|}{N},
&
b_t^{\gB}
&\coloneqq\frac{b|\gB_t|}{N}.
\label{eqn:scaled-camp-sizes}
\end{align}

\begin{lemma}[Expectations and variances of the degree differences]\label{lem:expected-variance-degree-differences}
    Under Assumption~\ref{ass:sparse-er}, whenever both camps are nonempty, the expected degree differences on day $t$ are
    \begin{align}
        m_{t}^{\gR} = p(\Delta_t+1), \qquad m_{t}^{\gB} = p(1-\Delta_t). \label{eqn:expected-degree-differences-R-B}
    \end{align}
Furthermore, the variances of the degree differences on day $t$ satisfy
\begin{align}
    v_{t}^{\gR}=v_{t}^{\gB}=(N-1)p(1-p)\asymp \log(N). \label{eqn:variance-degree-differences-R-B}
\end{align}
\end{lemma}
\begin{proof}[Proof of Lemma~\ref{lem:expected-variance-degree-differences}]
    The proof follows directly from linearity of expectation and variance. By
    \eqref{eqn:D-Rt-distribution},
    \begin{align*}
    m_{t}^{\gR}
    &=
    p|\gB_t|-p(|\gR_t|-1)
    =p(\Delta_t+1),\\
    v_{t}^{\gR}
    &=
    \bigl(|\gB_t|+|\gR_t|-1\bigr)p(1-p)
    =(N-1)p(1-p).
    \end{align*}
    Similarly, by \eqref{eqn:D-Bt-distribution},
    \begin{align*}
    m_{t}^{\gB}
    &=
    p|\gR_t|-p(|\gB_t|-1)
    =p(1-\Delta_t),\\
    v_{t}^{\gB}
    &=
    \bigl(|\gR_t|+|\gB_t|-1\bigr)p(1-p)
    =(N-1)p(1-p).
    \end{align*}
    By Assumption~\ref{ass:sparse-er}, the common variance is asymptotic to
    \(b\log N\).
    \end{proof}

The camp-size identities in \eqref{eqn:camp-sizes-from-advantage-intro}, together
with \eqref{eqn:degree-differences}, give the following conditional
expectations.

\begin{lemma}[Conditional expectations of the next day's camp sizes and advantage]
    \label{lem:conditional-expectations-next-day}
    For all \(t \geq 0\), the conditional expectations of the camp sizes on the next day can be expressed as
    \begin{subequations}
    \begin{align}
        \E\left[ |\gB_{t+1}|\,\middle|\, \rvy_{t} \right]
&\, = |\gR_{t}|p_t^{\gR} + |\gB_{t}|q_t^{\gB}, \label{eq:B-next-expectation}\\
        \E\left[ |\gR_{t+1}|\,\middle|\, \rvy_{t} \right]
&\, = |\gR_{t}|q_t^{\gR} + |\gB_{t}|p_t^{\gB}. \label{eq:R-next-expectation}
    \end{align}
    \end{subequations}
    Furthermore, the conditional expectation of the advantage on the next day can be expressed as
    \begin{subequations}
        \begin{align}
            \E\left[\Delta_{t+1} \,\middle|\, \rvy_{t}\right]
            &\, = \Delta_t + 2 \left( |\gR_{t}|p_t^{\gR} - |\gB_{t}|p_t^{\gB} \right). \label{eq:advantage-next-expectation}\\
            &\, = N\bigl(p_t^{\gR}-p_t^{\gB}\bigr)
            +
            \Delta_t
            \bigl(1-p_t^{\gR}-p_t^{\gB}\bigr).
            \label{eq:advantage-expectation-expanded}
        \end{align}
    \end{subequations}
\end{lemma}
    \begin{proof}[Proof of Lemma~\ref{lem:conditional-expectations-next-day}]
        Condition on \(\rvy_t\). By the update rule in \eqref{eqn:opinion-update} and
        the degree differences in \eqref{eqn:degree-differences}, a red vertex
        \(v\in\gR_t\) contributes to \(\gB_{t+1}\) if and only if
        \(\rD_t^{\gR}(v)>0\), while it contributes to \(\gR_{t+1}\) if and only if
        \(\rD_t^{\gR}(v)\le0\). Similarly, a blue vertex \(u\in\gB_t\) contributes to
        \(\gR_{t+1}\) if and only if \(\rD_t^{\gB}(u)>0\), while it contributes to
        \(\gB_{t+1}\) if and only if \(\rD_t^{\gB}(u)\le0\). Therefore,
        \begin{align*}
        |\gB_{t+1}|
        &=
        \sum_{v\in\gR_t}\indi{\rD_t^{\gR}(v)>0}
        +
        \sum_{u\in\gB_t}\indi{\rD_t^{\gB}(u)\le0},\\
        |\gR_{t+1}|
        &=
        \sum_{v\in\gR_t}\indi{\rD_t^{\gR}(v)\le0}
        +
        \sum_{u\in\gB_t}\indi{\rD_t^{\gB}(u)>0}.
        \end{align*}
        Taking conditional expectations and using \eqref{eqn:p-q-R} and
        \eqref{eqn:p-q-B}, we obtain \eqref{eq:B-next-expectation} and
        \eqref{eq:R-next-expectation}.
        
        We now compute the conditional expectation of the advantage. By
        \eqref{eqn:camp-sizes-from-advantage-intro}, \eqref{eq:B-next-expectation}, and
        \eqref{eq:R-next-expectation},
        \begin{align*}
        \E\left[\Delta_{t+1}\,\middle|\,\rvy_t\right]
        &=
        \E\left[ |\gB_{t+1}|-|\gR_{t+1}|\,\middle|\,\rvy_t\right]\\
        &=
        |\gR_t|p_t^{\gR}+|\gB_t|q_t^{\gB}
        -
        \left(|\gR_t|q_t^{\gR}+|\gB_t|p_t^{\gB}\right)\\
        &=
        \Delta_t
        +
        2|\gR_t|p_t^{\gR}
        -
        2|\gB_t|p_t^{\gB} \\
        &=
        \Delta_t
        +
        2\left(|\gR_t|p_t^{\gR}-|\gB_t|p_t^{\gB}\right),
        \end{align*}
        where the third equality uses \(q_t^{\gR}=1-p_t^{\gR}\) and
        \(q_t^{\gB}=1-p_t^{\gB}\). This proves \eqref{eq:advantage-next-expectation}. It remains only to rewrite the same expression in terms of \(\Delta_t\) and \(N\). From \eqref{eqn:camp-sizes-from-advantage-intro}, \(2|\gR_t|=N-\Delta_t\) and \(2|\gB_t|=N+\Delta_t\). Substituting these identities into
        \eqref{eq:advantage-next-expectation} yields
        \begin{align*}
        \E\left[\Delta_{t+1}\,\middle|\,\rvy_t\right]
        &=
        \Delta_t
        +
        \left(N-\Delta_t\right)p_t^{\gR}
        -
        \left(N+\Delta_t\right)p_t^{\gB} \\
        &=
        N\bigl(p_t^{\gR}-p_t^{\gB}\bigr)
        +
        \Delta_t\bigl(1-p_t^{\gR}-p_t^{\gB}\bigr),
        \end{align*}
        which is \eqref{eq:advantage-expectation-expanded}. This completes the proof.
\end{proof}

Therefore, to understand the one-step evolution of the advantage, we need to understand the flip probabilities \(p_t^{\gR}\) and \(p_t^{\gB}\) in a very detailed manner. A natural idea is to approximate them using a Gaussian distribution via the Berry--Esseen theorem below.
\begin{lemma}[Berry--Esseen \citealp{esseen1956moment}]\label{lem:berry-esseen}
    Let $\rX_{1}, \rX_{2}, \ldots, \rX_{n}$ be independent random variables with zero means, variances $\sigma_{1}^{2}, \sigma_{2}^{2}, \ldots, \sigma_{n}^{2}$, respectively, and finite absolute third moments $\E[|\rX_{j}|^{3}] = \rho_{j} < \infty$. Define the normalized random variable \(\rS_{n}:=\sum_{j=1}^{n}\rX_{j}/\sqrt{\sum_{j=1}^{n}\sigma_{j}^{2}}\). Let $F_{n}(x)$ and $\Phi(x)$ denote the cumulative distribution functions of $\rS_{n}$ and the standard normal distribution, respectively. There exists a universal constant $\gC_{\mathrm{BE}}$ such that, for any \(n\),
    \[\sup_{x\in \R} |F_{n}(x) - \Phi(x)| \leq \gC_{\mathrm{BE}} \sum_{j=1}^{n} \rho_{j} ( \sum_{j=1}^{n} \sigma_{j}^{2})^{-3/2},
    \]
\end{lemma}
However, a straightforward application of Lemma~\ref{lem:berry-esseen} results in an error term of order \(O((\log N)^{-1/2})\), which neither accounts for the half-integer lattice correction nor achieves the sharper \(O((\log N)^{-1})\) error required in \eqref{eq:cc-gaussian-pR} and \eqref{eq:cc-gaussian-pB}. To address this, we instead employ the Poisson approximation, as developed in Section~\ref{sec:poisson-proof-cc-gaussian} below. Under the current sparsity Assumption~\ref{ass:sparse-er}, the Poisson approximation is more accurate than the final approximation requires. Indeed, Le Cam's inequality and maximal coupling couple \(\Bin(m,p)\) and \(\Pois(mp)\) with mismatch probability at most \(mp^2\). The two binomial variables in either one-vertex degree difference contain \(N-1\) trials in total, so two independent couplings and a union bound give \((N-1)p^2\) as an upper bound, which is
\[
O\left(\frac{(\log N)^2}{N}\right)
=o\left(\frac1{\log N}\right),
\]
under Assumption~\ref{ass:sparse-er}. Thus the coupling error is smaller than the accuracy required in Lemma~\ref{lem:cc-gaussian-flip-probabilities}. The remaining task is to obtain an \(O((\log N)^{-1})\) continuity-corrected approximation for the corresponding difference of independent Poisson variables.
\begin{lemma}[Continuity-corrected Gaussian estimates in the one-vertex Gaussian window]
    \label{lem:cc-gaussian-flip-probabilities}
    Recall \(x_t^{\gR},x_t^{\gB}\) from
    \eqref{eqn:standardized-degree-difference-definitions}. Fix \(c_0>0\)
    and \(M<\infty\), condition on \(\rvy_t\), and suppose that
    \begin{align}
    |\gR_t|\ge c_0N,
    \qquad
    |\gB_t|\ge c_0N,
    \label{eqn:linear-camp-condition}
    \end{align}
    while $x_t^{\gR}$ and $x_t^{\gB}$ in \eqref{eqn:standardized-degree-difference-definitions} satisfy
    \begin{align}
    |x_t^{\gR}|\le M,
    \qquad
    |x_t^{\gB}|\le M.
    \label{eqn:standardized-degree-differences}
    \end{align}
    Then, uniformly over such configurations,
    \begin{subequations}
    \begin{align}
    p_t^{\gR}
    &=
    \Phi\left(
    x_t^{\gR}-\frac{1}{2\sqrt{v_t^{\gR}}}
    \right)
    +O\left(\frac1{\log N}\right),
    \label{eq:cc-gaussian-pR}\\
    p_t^{\gB}
    &=
    \Phi\left(
    x_t^{\gB}-\frac{1}{2\sqrt{v_t^{\gB}}}
    \right)
    +O\left(\frac1{\log N}\right).
    \label{eq:cc-gaussian-pB}
    \end{align}
    \end{subequations}
    For \(\Delta_t\ge1\), there are constants \(c,C>0\) such that
    \begin{align}
    c\frac{p\Delta_t}{\sqrt{(N-1)p(1-p)}}
    \le
    p_t^{\gR}-p_t^{\gB}
    \le
    C\frac{p\Delta_t}{\sqrt{(N-1)p(1-p)}}.
    \label{eq:cc-gaussian-flip-difference}
    \end{align}
    Furthermore, for every fixed \(C_0<\infty\), when
    \(0\le\Delta_t\le C_0\sqrt{N/\log N}\), we have
    \begin{align}
    p_t^{\gR}-p_t^{\gB}
    =
    \sqrt{\frac{2}{\pi}}
    \frac{p\Delta_t}{\sqrt{(N-1)p(1-p)}}
    +O\left(\frac{p\Delta_t}{\log N}\right).
    \label{eq:cc-gaussian-critical-flip-difference}
    \end{align}
    All implicit constants and the constants \(c,C\) depend only on
    \(b,c_0,M\), and additionally on \(C_0\) in
    \eqref{eq:cc-gaussian-critical-flip-difference}.
\end{lemma}

\subsection{Gaussian and large-deviation estimates for sparse binomial differences}
\label{app:er-one-step-auxiliary}

Under Assumption~\ref{ass:sparse-er}, let
\begin{align}
\rD\coloneqq \rX-\rY,
\qquad
\rX\sim\Bin(s,p),
\qquad
\rY\sim\Bin(r,p),
\label{eqn:er-sparse-binomial-difference}
\end{align}
where \(\rX\) and \(\rY\) are independent and
\(r=r(N),s=s(N)\) are nonnegative integers at most \(N\). Set
\begin{align}
A_N\coloneqq\frac{br}{N},
\qquad
B_N\coloneqq\frac{bs}{N},
\qquad
\mu_N\coloneqq p(s-r),
\qquad
\sigma_N^2\coloneqq p(1-p)(s+r).
\label{eqn:er-sparse-binomial-parameters}
\end{align}
The results below are obtained by specializing
{\citep[Lemma~A.1, Corollary~A.2, Lemma~A.3, and
Corollary~A.4]{dumitriu2026majority}}: the two edge probabilities there are
both set equal to \(p\), and the two logarithmic density constants are both
set equal to \(b\). We state the resulting ER estimates without repeating
their proofs.

\subsubsection{Gaussian window: Berry--Esseen approximation}

The next two statements are the ER specializations of
{\citep[Lemma~A.1 and Corollary~A.2]{dumitriu2026majority}}.

\begin{lemma}[Gaussian approximation for sparse binomial differences]
\label{lem:gaussian-sparse-binomial-difference}
Suppose that \(A_N+B_N\ge c_0\) for some fixed constant \(c_0>0\). Then,
uniformly over all \(r,s\) satisfying this condition,
\begin{align}
\sup_{x\in\R}
\left|
\P\left(\frac{\rD-\mu_N}{\sigma_N}\le x\right)-\Phi(x)
\right|
\le
\frac{C}{\sqrt{\log N}},
\label{eqn:er-sparse-binomial-berry-esseen}
\end{align}
where \(C=C(b,c_0)>0\).
\end{lemma}

\begin{corollary}[Gaussian estimates for one-vertex flip probabilities]
\label{cor:gaussian-flip-probabilities}
Recall \(p_t^{\gR},q_t^{\gR}\) from \eqref{eqn:p-q-R},
\(p_t^{\gB},q_t^{\gB}\) from \eqref{eqn:p-q-B},
\(x_t^{\gR},x_t^{\gB}\) from
\eqref{eqn:standardized-degree-difference-definitions}.
Fix \(t\ge0\), and condition on a configuration \(\rvy_t\) satisfying
\eqref{eqn:linear-camp-condition} for some fixed constant \(c_0>0\).
Uniformly over all such configurations,
\begin{subequations}
\begin{align}
p_t^{\gR}
&=
\Phi(x_t^{\gR})+O(1/\sqrt{\log N}),
\label{eq:gaussian-pR}\\
q_t^{\gR}
&=
\Phi(-x_t^{\gR})+O(1/\sqrt{\log N}),
\label{eq:gaussian-qR}\\
p_t^{\gB}
&=
\Phi(x_t^{\gB})+O(1/\sqrt{\log N}),
\label{eq:gaussian-pB}\\
q_t^{\gB}
&=
\Phi(-x_t^{\gB})+O(1/\sqrt{\log N}).
\label{eq:gaussian-qB}
\end{align}
\end{subequations}
The constants in the \(O(\cdot)\) terms depend only on \(b\) and \(c_0\).
\end{corollary}

\subsubsection{Large-deviation tail: rate-function exponents}

Recall \(\ReLU(z)=\max\{0,z\}\), and define
\begin{align}
I(x,y)
\coloneqq
\left(\ReLU\left(\sqrt{x}-\sqrt{y}\right)\right)^2.
\label{eqn:LDP-rate-function}
\end{align}
The next two statements are the ER specializations of
{\citep[Lemma~A.3 and Corollary~A.4]{dumitriu2026majority}}.

\begin{lemma}[Sparse-binomial large deviations]
\label{lem:binomial-difference-ld-general}
Suppose that \(A_N,B_N\ge c_0\) for some fixed constant \(c_0>0\). Then,
uniformly over all \(r,s\) satisfying this condition,
\begin{subequations}
\begin{align}
\P(\rD>0)
&=
N^{-I(A_N,B_N)+o(1)},
\label{eqn:binomial-difference-upper-tail}\\
\P(\rD\le0)
&=
N^{-I(B_N,A_N)+o(1)}.
\label{eqn:binomial-difference-lower-tail}
\end{align}
\end{subequations}
\end{lemma}

\begin{corollary}[Large-deviation estimates for one-vertex flip probabilities]
\label{cor:flip-probability-ld}
Recall \(b_t^{\gR},b_t^{\gB}\) from
\eqref{eqn:scaled-camp-sizes}. Fix \(t\ge0\), condition on \(\rvy_t\), and
suppose that \eqref{eqn:linear-camp-condition} holds for some fixed
constant \(c_0>0\). Then, uniformly over all such configurations,
\begin{subequations}
\begin{align}
p_t^{\gB}
&=
N^{-I(b_t^{\gB},b_t^{\gR})+o(1)},
\label{eq:blue-to-red-flip-ld}\\
q_t^{\gR}
&=
N^{-I(b_t^{\gB},b_t^{\gR})+o(1)},
\label{eq:red-stays-red-ld}\\
p_t^{\gR}
&=
N^{-I(b_t^{\gR},b_t^{\gB})+o(1)},
\label{eq:red-to-blue-flip-ld}\\
q_t^{\gB}
&=
N^{-I(b_t^{\gR},b_t^{\gB})+o(1)}.
\label{eq:blue-stays-blue-ld}
\end{align}
\end{subequations}
The missing self-edge changes the relevant camp size by one and hence its
scaled parameter by \(b/N\), which is absorbed by the uniform \(o(1)\) terms.
\end{corollary}

\section{Continuity-Corrected Gaussian Estimates via Poisson Approximation}
\label{sec:poisson-proof-cc-gaussian}
In this section, we present the proof of
Lemma~\ref{lem:cc-gaussian-flip-probabilities} in
Section~\ref{sec:one-step-evolution-advantage}, together with the necessary
technical preparations.

The proof of Lemma~\ref{lem:cc-gaussian-flip-probabilities} proceeds in two steps. First, each sparse binomial variable in a one-vertex degree difference is coupled to a Poisson variable with the same mean, reducing the desired flip probability estimates to estimates for a difference of two independent Poisson variables. Second, a continuity-corrected Gaussian approximation for that Poisson difference is transferred back through the coupling. This yields the required Gaussian formulas while retaining the half-integer correction caused by the lattice support.

The current section develops those two ingredients. We first recall total
variation and maximal coupling, and then combine them with Le Cam's inequality
to couple a difference of independent binomial variables to the corresponding
Poisson difference. Finally, we establish the continuity-corrected distributional
and local estimates for that Poisson difference. Throughout, we use the
standing edge probability \(p\) from Assumption~\ref{ass:sparse-er} and the one-step notation from Section~\ref{sec:one-step-evolution-advantage}.

\subsection{Total variation and maximal coupling}
For probability measures \(\mu\) and \(\nu\) on a countable space
\(\mathcal X\), denote the total variation distance by
\[
d_{\mathrm{TV}}(\mu,\nu)
\coloneqq
\sup_{A\subseteq\mathcal X}|\mu(A)-\nu(A)|.
\]

\begin{theorem}[Le Cam's inequality
{\citealp{lecam1960poissonbinomial,steele1994lecam}}]
\label{thm:le-cam-inequality}
Let \(\xi_1,\ldots,\xi_n\) be independent Bernoulli random variables with
\(\P(\xi_i=1)=p_i\). We denote
\[
W\coloneqq\sum_{i=1}^n\xi_i,
\qquad
\lambda\coloneqq\sum_{i=1}^np_i,
\]
and let \(P\sim\Pois(\lambda)\). Then
\begin{align}
\sum_{k=0}^{\infty}
\left|
\P(W=k)-e^{-\lambda}\frac{\lambda^k}{k!}
\right|
\le
2\sum_{i=1}^np_i^2.
\label{eq:le-cam-l1}
\end{align}
Equivalently, under the previous total-variation convention, we have
\begin{align}
d_{\mathrm{TV}}\bigl(\mathcal L(W),\mathcal L(P)\bigr)
\le
\sum_{i=1}^np_i^2,
\label{eq:le-cam-tv}
\end{align}
where \(\mathcal L(W)\) and \(\mathcal L(P)\) denote the laws of \(W\) and \(P\), respectively.
\end{theorem}

\begin{lemma}[Maximal coupling]
\label{lem:maximal-coupling}
Let \(\mu\) and \(\nu\) be probability measures on a countable space
\(\mathcal X\). There is a coupling \((U,V)\) with
\(\mathcal L(U)=\mu\) and \(\mathcal L(V)=\nu\) such that
\begin{align}
\P(U\ne V)
=
d_{\mathrm{TV}}(\mu,\nu)
=
\frac12\sum_{x\in\mathcal X}|\mu(x)-\nu(x)|.
\label{eq:maximal-coupling}
\end{align}
\end{lemma}

\begin{proof}[Proof of Lemma~\ref{lem:maximal-coupling}]
Denote \(r(x)\coloneqq\min\{\mu(x),\nu(x)\}\) and
\(q\coloneqq\sum_xr(x)=1-d_{\mathrm{TV}}(\mu,\nu)\). With probability \(q\),
draw \(U=V\) from the probability mass function \(r/q\). On the complementary
event, draw \(U\) and \(V\) from the normalized residual masses
\((\mu-r)/(1-q)\) and \((\nu-r)/(1-q)\), respectively. The two residual
supports are disjoint, so \(U\ne V\) on this event. This coupling has the
prescribed marginals and satisfies \eqref{eq:maximal-coupling}; the cases
\(q\in\{0,1\}\) follow by the same construction with the empty component
omitted.
\end{proof}

\begin{lemma}[Le Cam coupling]
\label{lem:le-cam-binomial-difference-coupling}
Let \(X\sim\Bin(N_X,p)\) and \(Y\sim\Bin(N_Y,p)\) be independent random variables. There exists a joint realization \((\widehat X,\widehat Y,P_X,P_Y)\), where the marginals \(\widehat X\stackrel d=X\), \(\widehat Y\stackrel d=Y\) are independent, and \(P_X\sim\Pois(pN_{X})\), \(P_Y\sim\Pois(N_Yp)\) are independent, such that
\begin{align}
\P\left(
\widehat X-\widehat Y\ne P_X-P_Y
\right)
\le
(N_X+N_Y)p^2.
\label{eq:le-cam-difference-mismatch}
\end{align}
Moreover,
\begin{align}
\sup_{A\subseteq\Z}
\left|
\P(X-Y\in A)-\P(P_X-P_Y\in A)
\right|
\le
(N_X+N_Y)p^2.
\label{eq:le-cam-difference-events}
\end{align}
\end{lemma}

\begin{proof}[Proof of Lemma~\ref{lem:le-cam-binomial-difference-coupling}]
We write \(X\) as the sum of \(N_X\) independent \(\Ber(p)\) variables.
By applying Theorem~\ref{thm:le-cam-inequality} with \(p_i=p\), for some
\(P_X\sim\Pois(pN_{X})\), we have
\[
d_{\mathrm{TV}}\left(
\mathcal L(X),\Pois(pN_{X})
\right)
\le
p^{2}N_{X}.
\]
Lemma~\ref{lem:maximal-coupling} thus yields a joint realization
\((\widehat X,P_X)\) where \(\widehat X\stackrel d=X\) such that
\[
\P(\widehat X\ne P_X)\le p^{2}N_{X}.
\]
Similarly, a joint realization \((\widehat Y,P_Y)\) can be constructed such that
\[
\P(\widehat Y\ne P_Y)\leq p^{2}N_{Y},
\]
where \(\widehat Y\stackrel d=Y\) and
\(P_Y\sim\Pois(N_Yp)\). These two coupled pairs are drawn independently, which subsequently preserves the required independence within \((\widehat X,\widehat Y)\) and within \((P_X,P_Y)\).
Moreover,
\[
\left\{
\widehat X-\widehat Y\ne P_X-P_Y
\right\}
\subseteq
\{\widehat X\ne P_X\}\cup\{\widehat Y\ne P_Y\}.
\]
The union bound proves \eqref{eq:le-cam-difference-mismatch}. For every
\(A\subseteq\Z\), the two indicators of membership in \(A\) agree whenever
the coupled differences agree. Hence their expectation difference is at most
the mismatch probability, which proves
\eqref{eq:le-cam-difference-events}.
\end{proof}

\subsection{Continuity correction for a Poisson difference}
\begin{lemma}[Continuity correction for a Poisson difference]
\label{lem:poisson-difference-cc}
Let \(P_X\sim\Pois(\lambda_X)\) and
\(P_Y\sim\Pois(\lambda_Y)\) be independent, and set
\[
S\coloneqq P_X-P_Y,
\qquad
\lambda\coloneqq\lambda_X+\lambda_Y,
\qquad
m\coloneqq\lambda_X-\lambda_Y.
\]
Suppose that for some fixed \(M>0\), \(|m|\le M\sqrt{\lambda}\) as \(\lambda\to\infty\). Then uniformly over integers \(k,\ell\),
\begin{subequations}
\begin{align}
\P(S\le\ell)
&=
\Phi\left(
\frac{\ell+1/2-m}{\sqrt{\lambda}}
\right)
+O_M(\lambda^{-1}),
\label{eq:poisson-difference-cdf}\\
\P(S=k)
&=
\frac1{\sqrt{\lambda}}
\varphi\left(
\frac{k-m}{\sqrt{\lambda}}
\right)
+O_M(\lambda^{-3/2}).
\label{eq:poisson-difference-local}
\end{align}
\end{subequations}
\end{lemma}

\begin{proof}[Proof of Lemma~\ref{lem:poisson-difference-cc}]
Let \(U\) be uniform on \([-1/2,1/2]\), independently of \(S\), and define
\[
Y\coloneqq\frac{S+U-m}{\sqrt{\lambda}}.
\]
For every integer \(\ell\), the exact identity below follows from the continuity of \(U\),
\begin{align}
\P(S\le\ell)
=
\P\left(
Y\le\frac{\ell+1/2-m}{\sqrt{\lambda}}
\right),
\label{eq:poisson-uniform-smoothing-identity}
\end{align}
which then leads to the half-integer continuity correction. We then denote
\[
\operatorname{sinc}(x)\coloneqq
\begin{cases}
\sin(x)/x,&x\ne0,\\
1,&x=0.
\end{cases}
\]
    By independence of \(P_X\), \(P_Y\), and \(U\), the characteristic
    function of \(Y\) factors as
    \begin{align*}
    \psi_\lambda(u)
    &\coloneqq
    \E\exp\left(
    \frac{\ii u(S+U-m)}{\sqrt{\lambda}}
    \right)\\
    &=
    \exp\left(-\frac{\ii um}{\sqrt{\lambda}}\right)
    \E\exp\left(\frac{\ii uP_X}{\sqrt{\lambda}}\right)
    \E\exp\left(-\frac{\ii uP_Y}{\sqrt{\lambda}}\right)
    \E\exp\left(\frac{\ii uU}{\sqrt{\lambda}}\right).
    \end{align*}
    Here we denote the imaginary unit by \(\ii=\sqrt{-1}\). The Poisson characteristic-function identity
    \(\E e^{\ii zP}=\exp\{\theta(e^{\ii z}-1)\}\) for
    \(P\sim\Pois(\theta)\) gives
    \begin{align*}
    \E\exp\left(\frac{\ii uP_X}{\sqrt{\lambda}}\right)
    &=
    \exp\left\{
    \lambda_X\left(e^{\ii u/\sqrt{\lambda}}-1\right)
    \right\},\\
    \E\exp\left(-\frac{\ii uP_Y}{\sqrt{\lambda}}\right)
    &=
    \exp\left\{
    \lambda_Y\left(e^{-\ii u/\sqrt{\lambda}}-1\right)
    \right\}.
    \end{align*}
    Moreover, since \(U\) is uniform on \([-1/2,1/2]\),
    \begin{align*}
    \E\exp\left(\frac{\ii uU}{\sqrt{\lambda}}\right)
    &=
    \int_{-1/2}^{1/2}
    e^{\ii ux/\sqrt{\lambda}}
    \,dx =
    \frac{\sin\left(u/(2\sqrt{\lambda})\right)}
    {u/(2\sqrt{\lambda})}
    =
    \operatorname{sinc}\left(\frac{u}{2\sqrt{\lambda}}\right),
    \end{align*}
    where the value at \(u=0\) follows by continuity. Combining these three
    factors and using \(m=\lambda_X-\lambda_Y\), we obtain
    \begin{align}
    \psi_\lambda(u)
    ={}&
    \exp\left\{
    -\frac{\ii u(\lambda_X-\lambda_Y)}{\sqrt{\lambda}}
    +\lambda_X\left(e^{\ii u/\sqrt{\lambda}}-1\right)
    +\lambda_Y\left(e^{-\ii u/\sqrt{\lambda}}-1\right)
    \right\}
    \operatorname{sinc}\left(\frac{u}{2\sqrt{\lambda}}\right)
    \notag\\
    ={}&
    \exp\left\{
    \lambda_X
    \left(
    e^{\ii u/\sqrt{\lambda}}-1-\frac{\ii u}{\sqrt{\lambda}}
    \right)
    +\lambda_Y
    \left(
    e^{-\ii u/\sqrt{\lambda}}-1+\frac{\ii u}{\sqrt{\lambda}}
    \right)
    \right\}
    \operatorname{sinc}\left(\frac{u}{2\sqrt{\lambda}}\right).
    \label{eq:smoothed-poisson-difference-cf}
    \end{align}
Fix a sufficiently large constant \(A\), depending only on \(M\), and let
\[
T_\lambda\coloneqq\sqrt{A\log\lambda}
\]
denote the truncation threshold.
\begin{claim}
\label{claim:smoothed-poisson-central-expansion}
Uniformly for \(|u|\le T_\lambda\),
\begin{align}
\psi_\lambda(u)
=
e^{-u^2/2}
\left[
1+
O_M\left(
\frac{u^2+|u|^3+u^4}{\lambda}
\right)
\right].
\label{eq:smoothed-poisson-central-expansion}
\end{align}
\end{claim}
We defer the proof of the claim until the end of the proof.

Since \(Y\) has a continuous distribution, Fourier inversion for distribution
functions gives
\begin{align}
\P(Y\le x)-\Phi(x)
=
-\frac1{2\pi}
\int_{\R}
e^{-\ii ux}
\frac{\psi_\lambda(u)-e^{-u^2/2}}{\ii u}
\,du.
\label{eq:smoothed-poisson-cdf-inversion}
\end{align}
The integral is absolutely convergent; the estimates below also verify this
fact. On \(|u|\le T_\lambda\),
Claim~\ref{claim:smoothed-poisson-central-expansion} implies
\begin{align}
\int_{|u|\le T_\lambda}
\frac{|\psi_\lambda(u)-e^{-u^2/2}|}{|u|}
\,du
\le
\frac{C_M}{\lambda}.
\label{eq:smoothed-poisson-central-integral}
\end{align}

It remains to control the noncentral Fourier aliases. For
\[
I_j\coloneqq
\left[
(2j-1)\pi\sqrt{\lambda},
(2j+1)\pi\sqrt{\lambda}
\right],
\qquad j\in\Z,
\]
write \(u=2\pi j\sqrt{\lambda}+v\), where
\(|v|\le\pi\sqrt{\lambda}\). The modulus of the exponential factor in
\eqref{eq:smoothed-poisson-difference-cf} is
\[
\exp\left\{
\lambda\left(\cos(v/\sqrt{\lambda})-1\right)
\right\}
\le e^{-cv^2}
\]
for an absolute \(c>0\). On the central interval \(I_0\), this gives
\begin{align}
\int_{I_0\cap\{|u|>T_\lambda\}}
\frac{|\psi_\lambda(u)|}{|u|}
\,du
\le
Ce^{-cT_\lambda^2}.
\label{eq:smoothed-poisson-central-tail}
\end{align}
For \(j\ne0\), the zero of the sinc factor at
\(2\pi j\sqrt{\lambda}\) gives
\[
\left|
\operatorname{sinc}\left(\frac{u}{2\sqrt{\lambda}}\right)
\right|
\le
\frac{C|v|}{|j|\sqrt{\lambda}},
\qquad
|u|\ge c|j|\sqrt{\lambda}.
\]
Consequently,
\begin{align}
\int_{I_j}
\frac{|\psi_\lambda(u)|}{|u|}
\,du
\le
\frac{C}{j^2\lambda}
\int_{\R}|v|e^{-cv^2}\,dv
\le
\frac{C}{j^2\lambda}.
\label{eq:smoothed-poisson-alias-bound}
\end{align}
After summing over \(j\ne0\), choosing \(A\) sufficiently large in
\eqref{eq:smoothed-poisson-central-tail}, and using
\[
\int_{|u|>T_\lambda}\frac{e^{-u^2/2}}{|u|}\,du
\le Ce^{-T_\lambda^2/2},
\]
we obtain
\[
\int_{|u|>T_\lambda}
\frac{|\psi_\lambda(u)-e^{-u^2/2}|}{|u|}
\,du
\le\frac{C_M}{\lambda}.
\]
Together with \eqref{eq:smoothed-poisson-cdf-inversion} and
\eqref{eq:smoothed-poisson-central-integral}, this proves
\[
\sup_{x\in\R}|\P(Y\le x)-\Phi(x)|
\le\frac{C_M}{\lambda}.
\]
Equation~\eqref{eq:poisson-difference-cdf} now follows from
\eqref{eq:poisson-uniform-smoothing-identity}.

For the local estimate, omit the uniform smoothing factor and let
\[
\psi_\lambda^{(0)}(u)
\coloneqq
\exp\left\{
\lambda_X
\left(
e^{\ii \frac{u}{\sqrt{\lambda}}}-1-\frac{\ii u}{\sqrt{\lambda}}
\right)
+
\lambda_Y
\left(
e^{-\ii \frac{u}{\sqrt{\lambda}}}-1+\frac{\ii u}{\sqrt{\lambda}}
\right)
\right\}.
\]
Fourier inversion on one lattice period gives
\begin{align*}
\P(S=k)
=
\frac1{2\pi\sqrt{\lambda}}
\int_{-\pi\sqrt{\lambda}}^{\pi\sqrt{\lambda}}
e^{-\ii u(k-m)/\sqrt{\lambda}}
\psi_\lambda^{(0)}(u)
\,du.
\end{align*}
On \(|u|\le T_\lambda\), the same expansion as above, without the
\(O(u^2/\lambda)\) sinc term, gives
\[
\psi_\lambda^{(0)}(u)
=
e^{-u^2/2}
\left[
1+O_M\left(\frac{|u|^3+u^4}{\lambda}\right)
\right].
\]
The contribution of this error is \(O_M(\lambda^{-3/2})\). On the remainder
of the lattice period,
\[
|\psi_\lambda^{(0)}(u)|\le e^{-cu^2},
\]
and the Gaussian Fourier tail satisfies the same bound. Increasing \(A\), if
necessary, makes both tail contributions \(O_M(\lambda^{-3/2})\). Extending
the Gaussian integral to \(\R\) therefore yields
\[
\P(S=k)
=
\frac1{2\pi\sqrt{\lambda}}
\int_{\R}
e^{-\ii u(k-m)/\sqrt{\lambda}}e^{-u^2/2}
\,du
+O_M(\lambda^{-3/2}),
\]
which is \eqref{eq:poisson-difference-local}.

It remains to prove Claim~\ref{claim:smoothed-poisson-central-expansion}.
Taylor's formula in the exponent of
\eqref{eq:smoothed-poisson-difference-cf} gives, uniformly for
\(|u|\le T_\lambda\),
\begin{align*}
&\lambda_X
\left(
e^{\ii u/\sqrt{\lambda}}-1-\frac{\ii u}{\sqrt{\lambda}}
\right)
+
\lambda_Y
\left(
e^{-\ii u/\sqrt{\lambda}}-1+\frac{\ii u}{\sqrt{\lambda}}
\right)\\
&\qquad=
-\frac{u^2}{2}
-\frac{\ii m u^3}{6\lambda^{3/2}}
+\frac{u^4}{24\lambda}
+O\left(\frac{|u|^5}{\lambda^{3/2}}\right).
\end{align*}
Here the quadratic term uses
\(\lambda_X+\lambda_Y=\lambda\), while the cubic term uses
\(\lambda_X-\lambda_Y=m\). Since \(|m|\le M\sqrt{\lambda}\) and
\(T_\lambda=O(\sqrt{\log\lambda})\), the preceding display becomes
\begin{align}
&\lambda_X
\left(
e^{\ii u/\sqrt{\lambda}}-1-\frac{\ii u}{\sqrt{\lambda}}
\right)
+
\lambda_Y
\left(
e^{-\ii u/\sqrt{\lambda}}-1+\frac{\ii u}{\sqrt{\lambda}}
\right)\notag\\
&\qquad=
-\frac{u^2}{2}
+O_M\left(\frac{|u|^3+u^4}{\lambda}\right).
\label{eq:smoothed-poisson-exponent-expansion}
\end{align}
The error in \eqref{eq:smoothed-poisson-exponent-expansion} is \(o(1)\)
uniformly on the central range. Therefore,
\begin{align}
&\exp\left\{
\lambda_X
\left(
e^{\ii u/\sqrt{\lambda}}-1-\frac{\ii u}{\sqrt{\lambda}}
\right)
+
\lambda_Y
\left(
e^{-\ii u/\sqrt{\lambda}}-1+\frac{\ii u}{\sqrt{\lambda}}
\right)
\right\}\notag\\
&\qquad=
e^{-u^2/2}
\left[
1+
O_M\left(\frac{|u|^3+u^4}{\lambda}\right)
\right].
\label{eq:smoothed-poisson-exponential-expansion}
\end{align}
Finally, Taylor expansion at the origin gives
\begin{align}
\operatorname{sinc}\left(\frac{u}{2\sqrt{\lambda}}\right)
&=
1-\frac{u^2}{24\lambda}
+O\left(\frac{u^4}{\lambda^2}\right)
=
1+O\left(\frac{u^2}{\lambda}\right)
\label{eq:smoothed-poisson-sinc-expansion}
\end{align}
uniformly for \(|u|\le T_\lambda\). Multiplying
\eqref{eq:smoothed-poisson-exponential-expansion} and
\eqref{eq:smoothed-poisson-sinc-expansion} proves
\eqref{eq:smoothed-poisson-central-expansion}. The same calculation without
the sinc factor gives the unsmoothed expansion used in the local estimate,
which completes the proof.
\end{proof}

\subsubsection{Proof of Lemma~\ref{lem:cc-gaussian-flip-probabilities}}
\begin{proof}[Proof of
    Lemma~\ref{lem:cc-gaussian-flip-probabilities}]
    Condition throughout on \(\rvy_t\). We first prove
    \eqref{eq:cc-gaussian-pR}. By \eqref{eqn:D-Rt-distribution}, write
    \[
    \rD_t^{\gR}=X_t^{\gR}-Y_t^{\gR},
    \qquad
    X_t^{\gR}\sim\Bin(|\gB_t|,p),
    \qquad
    Y_t^{\gR}\sim\Bin(|\gR_t|-1,p),
    \]
    with the two binomial variables independent. Let
    \[
    P_{X,t}^{\gR}\sim\Pois(|\gB_t|p),
    \qquad
    P_{Y,t}^{\gR}\sim\Pois((|\gR_t|-1)p)
    \]
    be independent Poisson variables, and denote    
    \[
    S_t^{\gR}\coloneqq P_{X,t}^{\gR}-P_{Y,t}^{\gR},
    \qquad
    \lambda_t\coloneqq(N-1)p.
    \]
    The mean of \(S_t^{\gR}\) is \(m_t^{\gR}\), while its variance is
    \(\lambda_t\). Applying
    Lemma~\ref{lem:le-cam-binomial-difference-coupling} with
    \(N_X=|\gB_t|\) and \(N_Y=|\gR_t|-1\) gives
    \begin{align}
    \sup_{A\subseteq\Z}
    \left|
    \P(\rD_t^{\gR}\in A\mid\rvy_t)
    -
    \P(S_t^{\gR}\in A)
    \right|
    \le
    C(N-1)p^2
    =
    O\left(\frac{(\log N)^2}{N}\right).
    \label{eq:er-poisson-coupling}
    \end{align}
    
    Recall that
    \[
    v_t^{\gR}=(N-1)p(1-p)=\lambda_t(1-p),
    \qquad
    \lambda_t\asymp\log N.
    \]
    The assumed bound on \(m_t^{\gR}/\sqrt{v_t^{\gR}}\) therefore implies a
    fixed bound on \(m_t^{\gR}/\sqrt{\lambda_t}\). Applying
    \eqref{eq:poisson-difference-cdf} with \(\ell=0\) gives
    \[
    \P(S_t^{\gR}\le0)
    =
    \Phi\left(
    \frac{1/2-m_t^{\gR}}{\sqrt{\lambda_t}}
    \right)
    +O\left(\frac1{\log N}\right).
    \]
    Moreover,
    \[
    \left|
    (1/2-m_t^{\gR})
    \left(
    \frac1{\sqrt{\lambda_t}}
    -
    \frac1{\sqrt{v_t^{\gR}}}
    \right)
    \right|
    \le C_Mp
    =o\left(\frac1{\log N}\right).
    \]
    Since \(\Phi\) is Lipschitz and the coupling error in
    \eqref{eq:er-poisson-coupling} is \(o((\log N)^{-1})\), we conclude that
    \[
    \P(\rD_t^{\gR}\le0\mid\rvy_t)
    =
    \Phi\left(
    \frac{1/2-m_t^{\gR}}{\sqrt{v_t^{\gR}}}
    \right)
    +O\left(\frac1{\log N}\right).
    \]
    Taking complements and using \(1-\Phi(x)=\Phi(-x)\) proves
    \eqref{eq:cc-gaussian-pR}. The proof of
    \eqref{eq:cc-gaussian-pB} is identical, starting from
    \eqref{eqn:D-Bt-distribution}.
    
    We next prove the estimates for \(p_t^{\gR}-p_t^{\gB}\). Assume
    \(\Delta_t\ge1\), since both conclusions are immediate when \(\Delta_t=0\).
    For each integer \(0\le j\le\Delta_t\), let
    \(\rD_{t,j}^{\gR}\) and \(\rD_{t,j}^{\gB}\) have the respective
    distributions
    \begin{align*}
    \rD_{t,j}^{\gR}
    &\stackrel d=
    \Bin(|\gR_t|+j,p)-\Bin(|\gR_t|-1,p),\\
    \rD_{t,j}^{\gB}
    &\stackrel d=
    \Bin(|\gR_t|,p)-\Bin(|\gR_t|-1+j,p),
    \end{align*}
    with independent binomial variables in each difference. Their means are
    \(p(j+1)\) and \(p(1-j)\), and both have variance
    \((2|\gR_t|+j-1)p(1-p)\).
    
    The standardized-mean assumption in
    \eqref{eqn:standardized-degree-differences} gives
    \[
    p\Delta_t
    \le
    C_M\sqrt{(N-1)p(1-p)}.
    \]
    Consequently, uniformly over \(0\le j\le\Delta_t\),
    \[
    (2|\gR_t|+j-1)p\asymp\log N,
    \qquad
    \frac{|p(j+1)|+|p(1-j)|}
    {\sqrt{(2|\gR_t|+j-1)p}}
    \le C_M.
    \]
    For each \(j\), Lemma~\ref{lem:le-cam-binomial-difference-coupling}
    replaces the two binomial variables by independent Poisson variables with
    total mean \((2|\gR_t|+j-1)p\) and coupling error at most \(Np^2\).
    Applying \eqref{eq:poisson-difference-local}, and then replacing the Poisson
    variance by the binomial variance, gives, uniformly in
    \(0\le j\le\Delta_t\) and \(k\in\{0,1\}\),
    \begin{align*}
    \P(\rD_{t,j}^{\gR}=k)
    &=
    \frac{
    \varphi\left(
    \dfrac{k-p(j+1)}
    {\sqrt{(2|\gR_t|+j-1)p(1-p)}}
    \right)
    }{
    \sqrt{(2|\gR_t|+j-1)p(1-p)}
    }
    +O_M\left(\frac1{\log N}\right),\\
    \P(\rD_{t,j}^{\gB}=k)
    &=
    \frac{
    \varphi\left(
    \dfrac{k-p(1-j)}
    {\sqrt{(2|\gR_t|+j-1)p(1-p)}}
    \right)
    }{
    \sqrt{(2|\gR_t|+j-1)p(1-p)}
    }
    +O_M\left(\frac1{\log N}\right).
    \end{align*}
    Indeed, the local Poisson error is \(O_M((\log N)^{-3/2})\), the coupling
    error is
    \[
    Np^2=O\left(\frac{(\log N)^2}{N}\right)
    =o\left(\frac1{\log N}\right),
    \]
    and replacing \(p\) by \(p(1-p)\) in the variance changes the Gaussian
    expression by \(O_M(p/\sqrt{\log N})\).
    
    At \(j=0\), the two interpolating differences have the same distribution.
    Adding one independent \(\Ber(p)\) trial at each successive value of \(j\)
    gives
    \begin{align*}
    \P(\rD_{t,j+1}^{\gR}>0)-\P(\rD_{t,j}^{\gR}>0)
    &=
    p\P(\rD_{t,j}^{\gR}=0),\\
    \P(\rD_{t,j}^{\gB}>0)-\P(\rD_{t,j+1}^{\gB}>0)
    &=
    p\P(\rD_{t,j}^{\gB}=1).
    \end{align*}
    Summing these identities from \(j=0\) to \(j=\Delta_t-1\), and using the
    equality in distribution at \(j=0\), yields
    \[
    p_t^{\gR}-p_t^{\gB}
    =
    p\sum_{j=0}^{\Delta_t-1}
    \left[
    \P(\rD_{t,j}^{\gR}=0)
    +
    \P(\rD_{t,j}^{\gB}=1)
    \right].
    \]
    If \(0\le\Delta_t\le C_0\sqrt{N/\log N}\), the standardized means in this
    identity converge to zero uniformly, so each paired local mass equals
    \[
    \sqrt{\frac{2}{\pi}}
    \frac1{\sqrt{(N-1)p(1-p)}}
    +O\left(\frac1{\log N}\right).
    \]
    Substitution proves \eqref{eq:cc-gaussian-critical-flip-difference}. In the
    full Gaussian window, the standardized arguments remain in a fixed compact
    interval. Each local mass is therefore bounded above and below by positive
    constant multiples of \(((N-1)p(1-p))^{-1/2}\), which proves
    \eqref{eq:cc-gaussian-flip-difference}.
    \end{proof}

\section{Auxiliary Probabilistic Tools}
\label{app:probabilistic-tools}
We record the external central limit theorem and concentration inequality
used in the main argument, with their roles and assumptions stated explicitly.

\subsection{The Berkowitz--Devlin One-Step Central Limit Theorem}
\label{sec:berkowitz-devlin-one-step-clt}

For completeness, we record the external central limit theorem used in
Section~\ref{sec:er-critical-window-winner-selection}. The formulation below is a
restatement of \citet[Theorem~1]{berkowitz2020central}.

\begin{theorem}[Berkowitz--Devlin one-step CLT]
\label{thm:berkowitz-devlin-one-step-clt}
Let \(\gR_0,\gB_0\) be a deterministic partition of \([N]\). Independently
sample \(\gG\sim\mathbb{G}(N,p)\) and perform one synchronous majority update with
ties retaining their current colors. Define
\[
\sigma_N\coloneqq\sqrt{Np(1-p)}.
\]
For independent random variables
\[
W_R\sim\Bin(|\gR_0|,p),
\qquad
W_B\sim\Bin(|\gB_0|,p),
\]
set
\[
\chi_N
\coloneqq
\P(W_R\ge W_B)\P(W_R\le W_B).
\]
Then the following conclusions hold.
\begin{enumerate}[label=\textup{(\roman*)}]
\item For arbitrary parameter choices,
\begin{align}
\Var(|\gR_1|)
=
N\chi_N
+O\left(
|\Delta_0|+\frac{N}{\sigma_N}
\right).
\label{eq:berkowitz-devlin-variance}
\end{align}
\item Suppose that
\[
\sigma_N\longrightarrow\infty,
\qquad
\log(1/\chi_N)=o(\log\sigma_N).
\]
Then
\[
\Var(|\gR_1|)\sim N\chi_N,
\]
and, for every fixed \(a<c\),
\begin{align}
\P\left(
a\le
\frac{|\gR_1|-\E|\gR_1|}{\sqrt{N\chi_N}}
\le c
\right)
\longrightarrow
\Phi(c)-\Phi(a).
\label{eq:berkowitz-devlin-clt}
\end{align}
\end{enumerate}
In particular, the assumptions and conclusions in \textup{(ii)} hold if
\(\sigma_N\to\infty\) and
\[
|\Delta_0|p=O(\sigma_N).
\]
\end{theorem}

\begin{remark}[Qualitative nature of the one-step CLT]
\label{rem:berkowitz-devlin-qualitative}
Theorem~\ref{thm:berkowitz-devlin-one-step-clt} is a qualitative central
limit theorem: its method-of-moments proof gives convergence to the standard
normal distribution but does not provide a rate in Kolmogorov distance. This
is sufficient for our critical-window winner-selection result, which only
requires the uniform one-step approximation error to be \(o(1)\). In
Lemma~\ref{lem:critical-window-one-step-clt}, we therefore keep this
qualitative error separate from the explicit
\(O((\log N)^{-1/2})\) errors coming from the sparse-binomial approximation,
the variance normalization, and the conditional mean.
\end{remark}

\subsection{Read-k Chernoff Bound}
\label{sec:read-k-chernoff-bound}
We first introduce the concept of a read-\(k\) family.
\begin{definition}[Read-\(k\) family]
    Let \(\rX_1,\ldots,\rX_m\) be independent random variables. For each \(j\in[r]\), let \(P_j\subseteq[m]\), and let \(f_j\) be a Boolean function of $\{\rX_i\}_{i\in P_j}$. Suppose that each \(\rX_i\) is read by at most \(k\) of the functions \(f_j\), i.e., \(\bigl|\{j\in[r]: i\in P_j\}\bigr|\le k\) for every \(i\in[m]\). Then, the random variables $\{\rY_j \coloneqq f_j(\{\rX_i\}_{i\in P_j})\}_{j\in[r]}$ form a read-\(k\) family.
\end{definition}
Then we state the read-\(k\) Chernoff bound.
\begin{lemma}[Read-\(k\) Chernoff bound {\citealp[Theorem~1.1]{gavinsky2015tail}}]
    \label{lem:read-k}
     Let \(\{\rY_j\}_{j\in[r]}\) be a read-\(k\) family of indicator random variables with \(p_j \coloneqq\P(\rY_j=1)\), and denote their average by \(\bar p \coloneqq r^{-1}\sum_{j=1}^r p_j\).
    If \(0<\bar p<1\), then
    \begin{subequations}
    \begin{align}
    \P\bigg(
    \sum_{j=1}^r \rY_j\ge (\bar p+\varepsilon)r
    \bigg)
    &\le
    \exp\bigg(
    -\frac{r}{k}\D_{\mathrm{KL}}(\bar p+\varepsilon\|\bar p)
    \bigg),
    &&0<\varepsilon\le1-\bar p,
    \\
    \P\bigg(
    \sum_{j=1}^r \rY_j\le (\bar p-\varepsilon)r
    \bigg)
    &\le
    \exp\bigg(
    -\frac{r}{k}\D_{\mathrm{KL}}(\bar p-\varepsilon\|\bar p)
    \bigg),
    &&0<\varepsilon\le\bar p.
    \end{align}
    \end{subequations}
    where the Kullback-Leibler divergence is defined as
    \[
    \D_{\mathrm{KL}}(u\|v)\coloneqq
    u\log(u/v)+(1-u)\log((1-u)/(1-v)).
    \]
    We use the standard convention \(0\log0=0\) at \(u\in\{0,1\}\).
    If \(\bar p\in\{0,1\}\), all the indicators are almost surely constant and
    the corresponding tail statements are trivial.
    \end{lemma}

    \begin{corollary}[Centered concentration for a read-\(k\) sum
    {\citealp{dumitriu2026majority}}]
    \label{cor:read-k-chernoff-bound-sum}
    Let \(\{\rY_j\}_{j\in[r]}\) be a read-\(k\) family of indicator random
    variables, and set \(\rS\coloneqq\sum_{j=1}^r\rY_j\). Then, for every
    \(s>0\),
    \begin{align}
    \P\left(\left|\rS-\E \rS\right|\ge s\right)
    \le
    2\exp\left(-\frac{2s^2}{kr}\right).
    \label{eq:read-k-centered-concentration}
    \end{align}
    \end{corollary}

    \begin{lemma}[Read-\(2\) concentration for one-step red camp and advantage
    {\citealp{dumitriu2026majority}}]
    \label{lem:camp-size-read-two-concentration}
    Fix \(t\ge0\) and condition on \(\rvy_t\). Then, for every \(u>0\),
    \begin{subequations}
    \begin{align}
    \P\left(
    \left||\gR_{t+1}|-\E[|\gR_{t+1}|\mid\rvy_t]\right|\ge u
    \,\middle|\,\rvy_t
    \right)
    &\le2\exp(-u^2/N),
    \label{eq:Rtplus-one-read-two-concentration-er}\\
    \P\left(
    \left|
    \Delta_{t+1}
    -
    \E[\Delta_{t+1}\mid\rvy_t]
    \right|
    \ge u
    \,\middle|\,
    \rvy_t
    \right)
    &\le
    2\exp\left(-\frac{u^2}{4N}\right).
    \label{eq:delta-read-two-concentration-er}
    \end{align}
    \end{subequations}
    \end{lemma}

    \begin{proof}
    Condition on \(\rvy_t\). For each vertex, the indicator that it is red at
    time \(t+1\) is a Boolean function of the independent edge indicators
    incident to that vertex. Each edge indicator is therefore read by at most
    its two endpoints, so these \(N\) next-color indicators form a read-\(2\)
    family. Applying Corollary~\ref{cor:read-k-chernoff-bound-sum} with
    \(r=N\) and \(k=2\) to their sum \(|\gR_{t+1}|\) gives
    \eqref{eq:Rtplus-one-read-two-concentration-er}. Finally,
    \(\Delta_{t+1}=N-2|\gR_{t+1}|\), so
    \[
    \left|\Delta_{t+1}-\E[\Delta_{t+1}\mid\rvy_t]\right|
    =
    2\left||\gR_{t+1}|-\E[|\gR_{t+1}|\mid\rvy_t]\right|.
    \]
    Applying the first bound with threshold \(u/2\) proves
    \eqref{eq:delta-read-two-concentration-er}.
    \end{proof}
\end{document}